\documentclass[11pt, reqno]{amsart}
\usepackage{fullpage}
\usepackage[OT1]{fontenc}
\usepackage[utf8]{inputenc}
\usepackage{lmodern}

\usepackage{amsmath,amssymb,amsthm,amsfonts,amsaddr,mathtools,bm}
\usepackage[authoryear]{natbib}
\usepackage{graphicx}
\usepackage{booktabs,array,enumitem,microtype}
\usepackage[colorlinks,citecolor=blue,urlcolor=blue]{hyperref}
\usepackage[nameinlink,noabbrev]{cleveref}

\hypersetup{
  pdftitle={Sharp Asymptotic Theory of Maximum Likelihood Estimation for Gaussian Processes with an RBF Kernel},
  pdfsubject={Sharp Asymptotics of MLEs for GP with an RBF Kernel}
}

\allowdisplaybreaks[2]
\setlist[enumerate]{leftmargin=*,itemsep=3pt,topsep=5pt}
\setlist[itemize]{leftmargin=*,itemsep=3pt,topsep=5pt}
\numberwithin{equation}{section}

\theoremstyle{plain}
\newtheorem{theorem}{Theorem}[section]
\newtheorem{proposition}[theorem]{Proposition}
\newtheorem{corollary}[theorem]{Corollary}
\newtheorem{lemma}[theorem]{Lemma}

\theoremstyle{definition}
\newtheorem{definition}[theorem]{Definition}
\newtheorem{assumption}{Assumption}
\newtheorem{remark}[theorem]{Remark}

\crefname{assumption}{assumption}{assumptions}
\Crefname{assumption}{Assumption}{Assumptions}

\newcommand{\R}{\mathbb R}
\newcommand{\N}{\mathbb N}
\newcommand{\E}{\mathbb E}
\newcommand{\PP}{\mathbb P}
\newcommand{\Id}{I}
\newcommand{\T}{\top}
\newcommand{\tr}{\operatorname{tr}}
\newcommand{\diag}{\operatorname{diag}}

\newcommand{\Cov}{\operatorname{Cov}}
\newcommand{\Var}{\operatorname{Var}}
\newcommand{\op}{\mathrm{op}}
\newcommand{\Frob}{\mathrm F}
\newcommand{\TV}{\operatorname{TV}}
\newcommand{\Aff}{\operatorname{Aff}}

\newcommand{\calI}{\mathcal I}
\newcommand{\calX}{\mathcal X}

\newcommand{\spat}{\mathrm{sp}}
\newcommand{\norm}[1]{\left\lVert#1\right\rVert}

\newcommand{\thetan}{\widehat\theta_n}

\DeclareMathOperator{\intt}{int}
\DeclareMathOperator{\chol}{chol}

\begin{document}

\title[Sharp Asymptotics of MLEs for RBF Gaussian Processes]{Sharp Asymptotic Theory of Maximum Likelihood Estimation for Gaussian Processes with an RBF Kernel}

\author{Ameer Qaqish$^1$, Sudipto Banerjee$^2$, and Didong Li$^1$}
\address{$^1$Department of Biostatistics, University of North Carolina at Chapel Hill\\
$^2$Department of Biostatistics, University of California, Los Angeles}
\email{ameer@unc.edu, sudipto@ucla.edu, didongli@unc.edu}

%\subjclass[2020]{Primary 62F12, 62M30, 60G15; Secondary 62F10}
%\keywords{Gaussian process, RBF kernel, maximum likelihood, fixed-domain asymptotics, minimax optimal}

\begin{abstract}
Gaussian processes (GPs) are widely used across machine learning, spatial statistics, time-series analysis, optimization, Bayesian statistics, and scientific applications. A central component of a GP model is its kernel, which is typically specified through a parametric family. Among the most widely used choices is the radial basis function (RBF), also known as the squared exponential or Gaussian kernel, owing to its simple form, smoothness, and flexibility. In practice, the kernel parameters are routinely estimated by the maximum likelihood estimators (MLEs), as implemented by standard GP software. Despite this widespread use, the asymptotic behavior of the MLEs remains poorly understood under fixed-domain asymptotics, even for the RBF kernel. The main difficulty arises from the increasingly strong dependence among densely sampled observations and the nonlinear dependence of the covariance matrix on the kernel parameters. In this paper, we address this gap by providing, to the best of our knowledge, the first complete asymptotic characterization of the joint MLE of the spatial variance, lengthscale, and nugget variance under fixed-domain asymptotics. We establish consistency, derive convergence rates for all three parameters, prove joint asymptotic normality, and show that these rates are minimax optimal.
\end{abstract}

\maketitle

\vspace{-5pt}
\tableofcontents

\section{Introduction}\label{sec:introduction}

Gaussian processes (GPs) are a fundamental class of probabilistic models with broad applications across statistics, machine learning, and the sciences. In machine learning, GPs provide flexible nonparametric models for regression, classification, and nonlinear dimension-reduction methods~\citep{williams2006gaussian}. In spatial statistics and time-series analysis, GPs provide a natural framework for modeling spatial~\citep{stein1999interpolation, cressie2011statistics, gelfand2017bayesian, banerjee2025hierarchical} and temporal dependence~\citep{roberts2013gaussian}. They are also widely used as surrogate models in Bayesian optimization \citep{garnett2023bayesian}, and, more generally, as flexible prior distributions over unknown functions in Bayesian inference~\citep{ghosal2017fundamentals}. Applications have further extended to biomedical problems, such as spatial transcriptomics~\citep{svensson2018spatialde,weber2023nnsvg}. 

A key component of a GP is its covariance function, also known as the kernel function, which determines the dependence and smoothness of the underlying random function \citep{williams2006gaussian}. Among the most commonly used choices is the radial basis function (RBF) kernel, also known as the squared exponential or Gaussian kernel. Its popularity stems from its simple form and its smoothness. In this paper, we consider the RBF kernel
\begin{equation}\label{eq:rbf-intro}
K_{\sigma^2,l}(x,x')
\coloneqq 
\sigma^2
\exp\left(-\frac{\|x-x'\|^2}{2l^2}\right),
\end{equation}
where $\sigma^2$ is the spatial variance, which controls the magnitude of spatial variation, and $l$ is the lengthscale, which determines the scale over which the process remains correlated. We observe the GP at $n$ locations with additive, independent Gaussian measurement errors of variance $\tau^2$, known as the nugget; the precise observation model is given in \Cref{sec:gp-background}. In practice, the parameter $\theta=(\sigma^2,l,\tau^2)$ is routinely estimated by maximizing the Gaussian likelihood, leading to the maximum likelihood estimator (MLE). This is the standard approach implemented in widely used GP software, including \texttt{GPML}~\citep{rasmussen2010gaussian} and \texttt{scikit-learn}~\citep{pedregosa2011scikit} in \textsf{Python} and \texttt{DiceKriging}~\citep{roustant2012dicekriging} and \texttt{GPfit}~\citep{macdonald2015gpfit} in \textsf{R}.

However, the theoretical understanding of the MLEs is considerably less developed than their widespread practical use would suggest~\citep{stein1999interpolation}. Throughout this paper, we focus on fixed-domain asymptotics, where observations become increasingly dense within a fixed bounded domain as the sample size increases. Several recent results address related aspects of this problem. \citet{qaqish2025identifiability} established the identifiability of all parameters of the RBF kernel, showing that consistent estimation of these parameters is at least possible in principle. Identifiability alone, however, does not establish the consistency or asymptotic distribution of any particular estimator. Moving from identifiability to estimation, \citet{loh2000estimating} established consistency of sieve MLEs for smooth Gaussian random fields, including the RBF case, under regular-grid sampling and a restricted parameterization. More recently, \citet{loh2026consistent} constructed consistent estimators of RBF parameters from noiseless scattered observations using local differences. These estimators, however, are not MLEs. \citet{xu2017maximum} studied likelihood estimation for a smooth Gaussian random field model when the observations are generated from a deterministic function rather than from the assumed GP. From a different perspective, \citet{karvonen2023maximum} showed that lengthscale MLE can be ill-posed in noiseless GP regression. Despite these developments, a general asymptotic theory for the joint MLE of the RBF spatial variance, lengthscale, and nugget variance remains unavailable.

Consequently, even for the RBF kernel, a basic statistical question remains unresolved: \emph{what is the asymptotic behavior of the MLE that practitioners actually compute?} In particular, it is not even known in general whether the joint MLE of $\theta=(\sigma^2,l,\tau^2)$ is consistent, at what rates its three components converge, or whether it admits an asymptotically normal distribution. The difficulty is fundamentally different from standard likelihood problems based on independent or weakly dependent observations. As the sampling locations become increasingly dense, the observations arise from a single realization of the underlying GP and become strongly dependent. At the same time, the covariance matrix is a nonlinear function of the kernel parameters, and its dimension and dependence structure both change with $n$. Standard likelihood arguments based on sums of approximately independent contributions therefore do not apply directly.

In this paper, we resolve this problem for the RBF kernel. We establish consistency of the joint MLE of $\theta$. We further derive the convergence rates of all three parameters, characterize the Fisher information, and establish joint asymptotic normality. In particular, we show that the MLEs of the spatial variance $\sigma^2$, lengthscale $l$, and nugget variance $\tau^2$ converge at rates 
\[\left(\frac{\log n}{\log\log n}\right)^{-p/2},\quad \left(\frac{\log n}{\log\log n}\right)^{-(p+2)/2},\quad n^{-1/2},\] respectively, where $p$ is the dimension of the input domain. We further establish matching minimax lower bounds and show that the MLE attains these bounds uniformly over the parameter space. Thus, the above three convergence rates are minimax optimal. The theory accommodates irregular observation locations under a polynomial coverage condition and does not require a regular grid or a minimum separation condition. Numerical experiments complement the theoretical results and examine the convergence rates empirically.

Beyond the specific asymptotic results, a key methodological contribution of this work is a general strategy for statistical inference with strongly dependent observations. The central idea is to construct a function of the observations that is a lower triangular linear transformation of i.i.d. standard normals, while controlling the approximation error and showing that it vanishes sufficiently fast. This converts a difficult dependent-data problem into one that can be analyzed using tools closer to those for independent observations. In the RBF setting studied here, the analyticity of the kernel allows us to construct such a transformation through a growing collection of Taylor coefficients, whose approximation errors decay rapidly with the sample size. The resulting representation makes the information carried by different kernel parameters explicit and provides the key mechanism for establishing consistency, convergence rates, and asymptotic normality of the MLEs. More broadly, this dependence-to-independence transformation provides a potentially useful strategy for asymptotic inference in other strongly dependent models whenever an appropriate latent representation can be recovered with sufficiently small error.

The remainder of the paper is organized as follows. \Cref{sec:back} reviews background on fixed-domain GP inference and related work on kernel parameter inference. \Cref{sec:theory} presents the main theoretical results for the joint MLEs. \Cref{sec:simulation} reports simulation studies supporting the rates. \Cref{sec:discussion} discusses implications, limitations, and directions for future work. All proofs and additional experimental details are provided in the Appendix.

\section{Background and related work}\label{sec:back}
\subsection{Background}\label{sec:gp-background}

\begin{definition}[GP]
Let $\mathcal{X}\subset\mathbb{R}^p$ be an input domain. A stochastic process
$\{F(x):x\in\mathcal{X}\}$ is a GP, denoted by $F\sim \mathrm{GP}(m,K)$, if, for every finite collection
$x_1,\ldots,x_n\in\mathcal{X}$, the random vector
\[
\bigl(F(x_1),\ldots,F(x_n)\bigr)^\top\sim N(\mu,\Sigma),
\]
where $\mu=[m(x_1),\cdots,m(x_n)]^\top$ and $\Sigma_{ij}=K(x_i,x_j)$. $m$ is called the mean function while $K$ is called the covariance function or kernel function, which together uniquely determine the GP.
\end{definition}

Throughout this paper, we assume that the mean function $m$ is known to be zero and focus
on inference for the kernel parameters of RBF as defined in \eqref{eq:rbf-intro}. We observe noisy outcomes
\begin{equation}\label{eqn:observation}
Y_{n,i}=F(x_{n,i})+\varepsilon_{n,i},
\qquad i=1,\ldots,n,
\end{equation}
at deterministic locations $x_{n,1},\ldots,x_{n,n}\in\mathcal{X}$, where
$\varepsilon_{n,i}\stackrel{\mathrm{iid}}{\sim}N(0,\tau^2)$ and the errors
are independent of $F$. We first develop the theory for deterministic designs satisfying a weak coverage condition;
random designs, including independent sampling from densities bounded away from zero on a region of $\mathcal{X}$, are treated in \Cref{cor:random}. For the deterministic-design formulation, the covariance matrix of the observations is
\begin{equation*}
C_n(\theta)
=
\sigma^2 R_n(l)+\tau^2 I_n,
\qquad
\theta=(\sigma^2,l,\tau^2),
\end{equation*}
where
\[
[R_n(l)]_{ij}
=
\exp\left(
-\frac{\|x_{n,i}-x_{n,j}\|^2}{2l^2}
\right).
\]
Thus,
\[
Y_n=(Y_{n,1},\ldots,Y_{n,n})^\top
\sim N_n(0,C_n(\theta)).
\]

We study fixed-domain asymptotics: the domain $\mathcal{X}\subset\mathbb{R}^p$ remains
fixed and bounded while the number of observation locations increases. This regime differs
from increasing-domain asymptotics, in which the observation region itself
expands with $n$~\citep{stein1999interpolation}. Under fixed-domain
asymptotics, increasingly many observations correspond to a single realization
of $F$ over the same bounded region, resulting in strong dependence among
nearby observations. Whether kernel parameters can be consistently
estimated in this regime depends critically on the kernel family, which
motivates the identifiability and estimation literature reviewed next.

\subsection{Related work on RBF parameter inference}
\label{sec:related-estimation}

Parameter estimation under fixed-domain asymptotics differs substantially
across kernel families. A fundamental issue is whether different kernel parameters induce equivalent or orthogonal Gaussian measures.
Parameters that can be distinguished in this sense are commonly referred
to as microergodic parameters~\citep{stein1999interpolation}. For the RBF kernel, \citet{qaqish2025identifiability} showed that all RBF parameters
are identifiable. This result, however, concerns the
distinguishability of the induced Gaussian measures and does not establish the behavior of a particular estimator.

Several works have studied RBF parameter inference. \citet{loh2000estimating} considered an RBF GP observed on a regular lattice and derived an explicit likelihood using the special structure of the resulting correlation matrix. They established strong consistency of sieve MLEs for the kernel parameters. This was an important early consistency result for likelihood-based estimation of smooth kernels but its scope is substantially different from the problem considered here. In particular, the analysis relies on a highly structured regular design and on a sieve likelihood construction tailored to that setting. In contrast, the MLE routinely used in GP software is obtained by directly maximizing the Gaussian likelihood over a prescribed parameter space, and observations in applications are rarely restricted to a regular lattice. Our goal is to characterize this exact MLE under substantially more general observation designs.

More recently, \citet{loh2026consistent} constructed consistent estimators of the spatial variance and lengthscale of the RBF kernel from noiseless observations at scattered, space-filling locations. The result substantially relaxes the regular-grid requirement and establishes directly that these parameters can be consistently recovered. However, the estimators are specifically constructed for the theoretical analysis and are not obtained by maximizing the Gaussian likelihood. Moreover, the observations are noiseless, corresponding to $\tau^2=0$ in our notation. Measurement error is present in real data applications, and introducing an unknown nugget changes both the likelihood and the asymptotic problem. Thus, while \citet{loh2026consistent} establishes the existence of consistent estimators, it does not determine whether the routinely used MLE is consistent or characterize its asymptotic behavior.

\citet{xu2017maximum} studied a different question motivated by computer experiments. There, the observations are values of a deterministic function $f$, while an RBF GP is used as a statistical model for interpolation. Consequently, there is no true covariance parameter $\theta_0$ generating the observations, and consistency toward a true RBF parameter is not the relevant question. For regularly spaced observations on $[0,1]$ with the lengthscale fixed, they showed that the MLE of the spatial variance converges to zero when $f$ is constant and diverges to infinity when $f$ is linear; they also derived asymptotic orders for related polynomial examples. Their results therefore characterize the behavior of likelihood estimation under this deterministic-function formulation rather than estimation of the kernel parameters of a correctly specified RBF GP. Indeed, \citet{xu2017maximum} explicitly noted that establishing estimation theory when the GP model is correctly specified remained a substantially different and difficult problem.

Finally, \citet{karvonen2023maximum} studied the well-posedness of MLE-based GP regression rather than consistency of covariance parameter estimators. They showed that, in the noiseless setting, estimating the lengthscale by maximum likelihood can make the resulting predictive distribution non-Lipschitz with respect to perturbations of the observed data. In particular, for certain data configurations the likelihood drives the estimated lengthscale to infinity. This notion of ill-posedness does not contradict the consistency result established in this paper. Their result concerns sensitivity of the fitted predictor to perturbations of a fixed noiseless data set, whereas our results concern the probabilistic behavior of the MLE when the data are generated from an RBF GP with noise as $n$ increases. The nugget is particularly important: it keeps the covariance matrices uniformly nonsingular and excludes the noiseless setting in which their negative result occurs.

Taken together, the existing literature establishes identifiability of the RBF parameters, consistency of specially constructed estimators, consistency of sieve MLEs under highly structured designs, and both positive and negative results for MLEs in other formulations. None of these results characterizes the joint MLE of $\theta=(\sigma^2,l,\tau^2)$ under the observation model in \eqref{eqn:observation}. In particular, consistency, convergence rates, asymptotic normality, and minimax optimality of this MLE remain unresolved. These are the questions addressed in \Cref{sec:theory}.

\section{Theory}\label{sec:theory}

This section states the model and assumptions in \Cref{sec:setup}, establishes the asymptotic behavior in \Cref{sec:results} and minimax optimality of the MLE in \Cref{sec:minimax}, with a few useful corollaries in \Cref{sec:covariance}, and gives an outline of the proof with some key ideas in \Cref{sec:outline}. Complete proofs are deferred to the Appendices \ref{app:statements}, \ref{app:proof-main}, and \ref{app:lemma-proofs}.

\subsection{Problem setup, assumptions, and notations}\label{sec:setup}

Recall from \Cref{sec:gp-background} that
\[
Y_n\sim N_n(0,C_n(\theta)),\qquad
C_n(\theta)=\sigma^2R_n(l)+\tau^2 I_n,
\qquad
\theta=(\sigma^2,l,\tau^2).
\]
For notational simplicity, in the theoretical analysis we write
\[
(s,l,t)=(\sigma^2,l,\tau^2),
\]
so that $\theta=(s,l,t)\in\Theta\subset (0,\infty)^3$ and $C_n(\theta)=sR_n(l)+tI_n$. We also write $K_\theta=K_{s,l}$ for the RBF kernel \eqref{eq:rbf-intro},
   which does not depend on $t$.
Let $\theta_0=(s_0,l_0,t_0)$ denote the true parameter. We estimate $\theta$ by maximizing the Gaussian likelihood over $\Theta$. Ignoring an additive constant, the log likelihood is
\begin{equation}\label{eq:likelihood}
l _n(\theta)
=
-\frac12\log \det C_n(\theta)
-\frac12Y_n^\top C_n(\theta)^{-1}Y_n,
\end{equation}
and the MLE is
\begin{equation*}
\widehat\theta_n
\in\operatorname*{arg\,max}_{\theta\in\Theta}l _n(\theta).
\end{equation*}

We now state the two assumptions of our theoretical analysis.

\begin{assumption}[Positive compact parameter set]\label{ass:parameter}
The parameter space is
\[
\Theta=[s_-,s_+]\times[l_-,l_+]\times[t_-,t_+]
\Subset(0,\infty)^3,
\]
i.e., $\Theta$ is compactly contained in $(0,\infty)^3$, and assume that $\theta_0\in\operatorname{int}(\Theta)$.
\end{assumption}

For a fixed closed cube $Q\subset\mathcal{X}$, define its fill distance by 
\[
h_n(Q)=
\sup_{x\in Q}\min_{1\le i\le n}\|x-x_{n,i}\|.
\]

\begin{assumption}[Polynomial coverage]\label{ass:coverage}
There exist a fixed closed cube $Q\subset\mathcal{X}$ with positive side length and constants $C_h<\infty$ and $\kappa>0$ such that, for all sufficiently large $n$,
\begin{equation}\label{eq:coverage}
h_n(Q)\le C_h n^{-\kappa}.
\end{equation}

\end{assumption}

\Cref{ass:coverage} requires the observation locations to become sufficiently dense on at least one fixed region of the domain, but imposes no minimum spacing condition. In particular, irregular, clustered, and repeated observation locations are allowed. Necessarily $\kappa \leq 1/p$, since $n$ balls of radius $h_n(Q)$ must cover a set of positive $p$-dimensional volume.

We finally introduce the Fisher information and the normalization used in our asymptotic results. Let subscripts on $C_n$ denote parameter derivatives. At $\theta_0$, define the Fisher information matrix
\begin{equation}\label{eq:fisher}
 [\calI_n(\theta_0)]_{ab}
 =\frac12\tr\left(
 C_{n,0}^{-1}C_{n,a}(\theta_0)
 C_{n,0}^{-1}C_{n,b}(\theta_0)\right),
 \quad a,b\in\{s,l,t\},
\end{equation}
where $C_{n,0}=C_n(\theta_0)$. For $n\ge16$, let
\begin{equation*}
 b_n=\frac{\log n}{\log\log n},~~
 D_n=\diag(b_n^{p/2},b_n^{p/2+1},\sqrt n),~~
 J_n=D_n^{-1}\calI_n(\theta_0)D_n^{-1},~~\epsilon_n=b_n^{-p/2}.
\end{equation*}
Thus, $J_n$ is the normalized Fisher information, with the normalization $D_n$ corresponding to the convergence rates of the three kernel parameters that will be established below.

Throughout this article, constants denoted by $c,C$ may change between displays and can depend on the fixed parameter rectangle, dimension, bounded domain, cube, and coverage constants, but not on $n$. For symmetric matrices, $\preceq$ is the Loewner order. The symbols $\norm{\cdot}_{\op}$ and $\norm{\cdot}_{\Frob}$ denote operator and Frobenius norms. For a multiindex $\alpha\in\N_0^p$, write
$|\alpha|=\sum_j\alpha_j$, $\alpha!=\prod_j\alpha_j!$, and
$x^\alpha=\prod_jx_j^{\alpha_j}$.

\subsection{MLE asymptotics}\label{sec:results}
Our main result is the following.
\begin{theorem}[MLE asymptotics]\label{thm:rbf}
Under \Cref{ass:parameter,ass:coverage}, the following assertions hold.
\begin{enumerate}[label=(\roman*)]
\item \textbf{Fisher information.} The normalized Fisher information is uniformly nondegenerate:
\begin{equation}\label{eq:information-bounds}
 c\Id_3\preceq J_n\preceq C\Id_3
\end{equation}
for all sufficiently large $n$. Its spatial block $J_{n,\spat}$ satisfies
$c\Id_2\preceq J_{n,\spat}\preceq C\Id_2$, and
\begin{equation}\label{eq:information-block}
 J_n=
 \begin{pmatrix}
 J_{n,\spat}&o(1)\\
 o(1)&(2t_0^2)^{-1}+o(1)
 \end{pmatrix}.
\end{equation}
\item \textbf{Consistency and convergence rates.} The MLE $\widehat{\theta}_n$ admits the expansion
\begin{equation}\label{eq:mle-expansion}
 D_n(\thetan-\theta_0)
 =J_n^{-1}D_n^{-1}\nabla l_n(\theta_0)+o_{\PP}(1).
\end{equation}
In particular, $\thetan$ is consistent with rates
\begin{equation}\label{eq:rates}
 \widehat s_n-s_0=O_{\PP}(b_n^{-p/2}),\quad
 \widehat l_n-l_0=O_{\PP}(b_n^{-(p+2)/2}),\quad
 \widehat t_n-t_0=O_{\PP}(n^{-1/2}).
\end{equation}
\item \textbf{Asymptotic normality.} Let $M_n$ be any $3\times3$ matrices with
 $M_n^\top M_n=\calI_n(\theta_0)$. Then
\begin{equation}\label{eq:main-an}
 M_n(\thetan-\theta_0)\ \Longrightarrow\ N_3(0,\Id_3).
\end{equation}
Moreover,
\begin{equation}\label{eq:nugget-an}
 \sqrt n(\widehat t_n-t_0)\Longrightarrow N(0,2t_0^2).
\end{equation}
\item \textbf{Exponential tail bounds.} Write $\widehat h_n=D_n(\thetan-\theta_0)
 =(\widehat h_{n,\spat},\widehat h_{n,t})$, where
 $\widehat h_{n,\spat}\in\R^2$ collects the two spatial coordinates $s$ and $l$.
 There are constants $c,C>0$ and
 $n_0$ such that, for all $n\ge n_0$ and all $H\ge0$,
\begin{equation}\label{eq:tail-bound}
\begin{gathered}
 \PP_{\theta_0,n}\bigl\{\norm{\widehat h_{n,\spat}}\ge H\bigr\}
 \le C\exp\{-c\min(H^2,b_n^{p/2}H)\},\\
 \PP_{\theta_0,n}\bigl\{|\widehat h_{n,t}|\ge H\bigr\}
 \le Ce^{-cH^2}.
\end{gathered}
\end{equation}
\item \textbf{Moment convergence.} The normalized MLE has uniformly bounded exponential moments: for every
 $\lambda>0$ there is $n_\lambda$ such that
\begin{equation}\label{eq:exp-moments}
 \sup_{n\ge n_\lambda}\E_{\theta_0,n}
 \exp\bigl\{\lambda\norm{D_n(\thetan-\theta_0)}\bigr\}<\infty.
\end{equation}
 Consequently, let $Z\sim N_3(0,\Id_3)$, and let $f:\R^3\to\R$ be
 continuous with $|f(x)|\le Ce^{\lambda\norm x}$ for some constants
 $C,\lambda>0$. Then, for every choice of $M_n$ in (iii),
\begin{equation}\label{eq:f-convergence}
 \E_{\theta_0,n}f\bigl(D_n(\thetan-\theta_0)\bigr)
 -\E f\bigl(J_n^{-1/2}Z\bigr)\to0,
 \qquad
 \E_{\theta_0,n}f\bigl(M_n(\thetan-\theta_0)\bigr)\to\E f(Z).
\end{equation}
 In particular, all moments converge, and
\begin{equation}\label{eq:moment-convergence}
 D_n\E_{\theta_0,n}(\thetan-\theta_0)\to0,\qquad
 D_n\E_{\theta_0,n}\bigl[(\thetan-\theta_0)(\thetan-\theta_0)^\top\bigr]D_n
 -J_n^{-1}\to0.
\end{equation}
\end{enumerate}
\end{theorem}

\Cref{thm:rbf} provides a complete asymptotic characterization of the MLE. All three kernel parameters are consistently estimated, but at substantially different rates. The nugget variance $t$ has the standard parametric rate $n^{-1/2}$, whereas the spatial variance $s$ and lengthscale $l$ converge at the slower logarithmic rates. We show in \Cref{sec:minimax} that these rates are minimax optimal. % The exponential tail and moment bounds further strengthen these distributional results and will allow us to establish the corresponding mean squared error (MSE) rates and their optimality below.

The MLE is also jointly asymptotically normal after normalization by the Fisher information. An important feature of \Cref{thm:rbf} is that it does not require the normalized Fisher information $J_n$ to converge. If, for a particular design sequence, $J_{n,\spat}\to J_{\spat}$, then it also implies
\begin{equation*}
 D_n(\thetan-\theta_0)
 \Longrightarrow N_3\left(
 0,\begin{pmatrix}J_{\spat}^{-1}&0\\0&2t_0^2\end{pmatrix}
 \right).
\end{equation*}
Thus, when $J_n$ converges, the MLE admits a conventional asymptotic normal distribution under the deterministic normalization $D_n$. When it does not, the Fisher-normalized result in part~(iii) continues to provide a valid asymptotic normal approximation for the actual observation design.

Beyond convergence in probability and distribution, parts (iv) and (v) provide exponential tail bounds and moment convergence for the normalized MLE. In particular, the moment convergence allows the stochastic rates in part (ii) to be translated into corresponding MSE rates, which are used in \Cref{sec:minimax} to establish minimax optimality.

\begin{remark}[Uniformity in $\theta_0$]\label{rem:uniform}
Inspection of the proofs shows that parts (i) and (iv) also hold
uniformly in $\theta_0\in\Theta$. Parts~(ii), (iii), and the limits
in part~(v) hold uniformly over $\theta_0$ in compact subsets
$K$ of $\intt(\Theta)$. For example, writing
$\mathcal L_{\theta_0}$ for the law under $\theta_0$ and
$d_{\mathrm{BL}}$ for the bounded Lipschitz distance,
\[
 \sup_{\theta_0\in K}d_{\mathrm{BL}}\bigl\{\mathcal L_{\theta_0}
 \bigl(M_n(\thetan-\theta_0)\bigr),N_3(0,\Id_3)\bigr\}\to0.
\]
Proofs of these uniform claims are given in Appendix~\ref{app:uniform}.

Uniformity is what makes these asymptotic statements usable for
inference when $\theta_0$ is unknown. It ensures that Wald confidence
intervals based on $\calI_n(\theta_0)^{-1}$ (or $\calI_n(\hat{\theta}_n)^{-1}$, as will be shown later) attain their nominal coverage
asymptotically, uniformly over $\theta_0\in K$, rather than only at each
fixed parameter value.

\end{remark}

\subsection{Minimax optimality}\label{sec:minimax}
The convergence rates in \Cref{thm:rbf} are logarithmic for the spatial variance $s$ and lengthscale $l$. In fact, the logarithmic rates also arise for the specially constructed estimators of \citet{loh2026consistent} for noiseless observations, whose established rates are slower than those obtained here for the MLE. A natural question is therefore
whether the rates in \Cref{thm:rbf} can be improved by using a different estimator. The following result shows that they cannot: the rates in \Cref{thm:rbf} are minimax optimal for all three parameters, and the MLE attains these optimal rates.

\begin{theorem}[Minimax optimality]\label{thm:minimax}
Assume \Cref{ass:parameter}. Let $d_{n,j}=D_{n,jj}$ for $j\in\{s,l,t\}$. For each $j\in\{s,l,t\}$, there exist constants $a>0$ and $c_0>0$ such
that, for every fixed $\theta_0\in\intt(\Theta)$,
\[
 \theta_{1,n}^{(j)}
 =\theta_0+a d_{n,j}^{-1}e_j\in\Theta
\]
for all sufficiently large $n$, and
\begin{equation}\label{eq:probability-lower}
 \inf_{\widetilde\theta_{n,j}}
 \max_{\theta\in\{\theta_0,\theta_{1,n}^{(j)}\}}
 \PP_{\theta,n}\left\{
 |\widetilde\theta_{n,j}-\theta_j|
 \ge\frac{a}{2d_{n,j}}\right\}
 \ge c_0.
\end{equation}
The other two parameters are identical under the two alternatives, so
this lower bound continues to hold even when they are known. Consequently,
\begin{equation}\label{eq:risk-lower}
 \inf_{\widetilde\theta_{n,j}}
 \sup_{\theta\in\Theta}
 \E_{\theta,n}(\widetilde\theta_{n,j}-\theta_j)^2
 \ge c\,d_{n,j}^{-2}.
\end{equation}

If \Cref{ass:coverage} also holds, then the MLE satisfies
\begin{equation}\label{eq:minimax-upper}
 \sup_{\theta\in\Theta}
 \E_{\theta,n}(\widehat\theta_{n,j}-\theta_j)^2
 \le C\,d_{n,j}^{-2}.
\end{equation}
Therefore, for each parameter $j\in\{s,l,t\}$,
\begin{equation}\label{eq:minimax-rate}
 \inf_{\widetilde\theta_{n,j}}
 \sup_{\theta\in\Theta}
 \E_{\theta,n}(\widetilde\theta_{n,j}-\theta_j)^2
 \asymp d_{n,j}^{-2},
\end{equation}
and the MLE attains the minimax rate. Equivalently, the minimax root MSE rates for $(s,l,t)$ are
\[
 b_n^{-p/2},\qquad
 b_n^{-(p+2)/2},\qquad
 n^{-1/2},
\]
respectively.
\end{theorem}

\subsection{Some useful corollaries}\label{sec:covariance}
The preceding results characterize the asymptotic behavior and prove minimax optimality of the MLE. We next give three corollaries that connect this theory more directly to statistical practice. They %quantify the actual estimation
%error through the variance and MSE (\Cref{cor:mse}), 
show that the theory carries over to the commonly used log parametrization (\Cref{cor:log-scale}), justify confidence intervals based
on estimated Fisher information rather than the Fisher information at the unknown true parameter (\Cref{cor:plugin}), and extend the
deterministic-design theory to random designs (\Cref{cor:random}). All proofs are given in
Appendix~\ref{app:proof-cor}.

\subsubsection{Log scale}
As all three parameters are positive, GP implementations commonly optimize them on the log scale, which automatically enforces positivity. Our computations and simulations in \Cref{sec:simulation} therefore use $\lambda=\log\theta=(\log s,\log l,\log t)$, coordinatewise. By
equivariance, the MLE of $\lambda$ is $\widehat\lambda_n=\log\thetan$.
Let $\Lambda_0=\diag(s_0,l_0,t_0)$. Since
$\partial_{\log\theta_j}=\theta_j\partial_{\theta_j}$, the score and the
Fisher information for $\lambda$ at $\lambda_0$ are
$\Lambda_0\nabla l_n(\theta_0)$ and
\[
 \calI_n^{\log}(\theta_0)=\Lambda_0\calI_n(\theta_0)\Lambda_0
 =D_nJ_n^{\log}D_n,\qquad J_n^{\log}=\Lambda_0J_n\Lambda_0,
\]
because the diagonal matrices $\Lambda_0$ and $D_n$ commute.

\begin{corollary}[Log scale]\label{cor:log-scale}
Under \Cref{ass:parameter,ass:coverage}, all assertions of
\Cref{thm:rbf} and \Cref{thm:minimax} hold for $\widehat\lambda_n$, with
$\thetan-\theta_0$ replaced by $\widehat\lambda_n-\lambda_0$,
$\calI_n(\theta_0)$ by $\calI_n^{\log}(\theta_0)$, $J_n$ by $J_n^{\log}$,
and $\nabla l_n(\theta_0)$ by $\Lambda_0\nabla l_n(\theta_0)$, after
adjusting the constants $c$ and $C$. In particular,
$M_n(\widehat\lambda_n-\lambda_0)\Rightarrow N_3(0,\Id_3)$ whenever
$M_n^\top M_n=\calI_n^{\log}(\theta_0)$,
\[
 \sqrt n\,(\log\widehat t_n-\log t_0)\Longrightarrow N(0,2),
\]
and
\[
 \E(\log\widehat s_n-\log s_0)^2\asymp b_n^{-p},\quad
 \E(\log\widehat l_n-\log l_0)^2\asymp b_n^{-(p+2)},\quad
 n\E(\log\widehat t_n-\log t_0)^2\longrightarrow2.
\]
The uniformity statements of \Cref{rem:uniform} also carry over.
\end{corollary}

Thus, working on the log scale does not change the convergence rates or the first-order asymptotic theory. It has additional advantages of respecting positivity and yielding a parameter-free limiting variance for the nugget, and, as the simulations in \Cref{sec:simulation} illustrate, can provide a better normal 
approximation at moderate sample sizes.

\subsubsection{Confidence intervals}
Asymptotic normality is useful for practical statistical inference only when the asymptotic covariance matrix can be estimated from the data. In particular, to report standard errors or confidence intervals for the kernel parameters, the Fisher information at the unknown true parameter $\theta_0$ needs to be replaced by an estimated quantity. The next corollary shows that both the expected and observed information evaluated at the MLE $\thetan$ are valid replacements.

\begin{corollary}[Estimated information and confidence intervals]\label{cor:plugin}
Under \Cref{ass:parameter,ass:coverage}, let $\widehat\calI_n$ denote
either the plug-in Fisher information $\calI_n(\thetan)$ or the observed
information $-\nabla^2l_n(\thetan)$. Then the following assertions hold.
\begin{enumerate}[label=(\roman*)]
\item \textbf{Consistency of the estimated Fisher information.}
\begin{equation}\label{eq:plugin-information}
 D_n^{-1}\widehat\calI_nD_n^{-1}-J_n=O_\PP(\epsilon_n).
\end{equation}
In particular, $\widehat\calI_n$ is positive definite with probability
tending to one, and for each $j\in\{s,l,t\}$,
\[
 \frac{[\widehat\calI_n^{-1}]_{jj}}{[\calI_n(\theta_0)^{-1}]_{jj}}
 =1+O_\PP(\epsilon_n).
\]
\item \textbf{Coordinatewise intervals.} For each $j\in\{s,l,t\}$,
\[
 \frac{\widehat\theta_{n,j}-\theta_{0,j}}
 {[\widehat\calI_n^{-1}]_{jj}^{1/2}}\Longrightarrow N(0,1).
\]
Consequently, for every $\alpha\in(0,1)$, letting $z_{1-\alpha/2}$ denote the $(1-\alpha/2)$ quantile of the standard normal distribution, the Wald interval
$\widehat\theta_{n,j}\pm z_{1-\alpha/2}[\widehat\calI_n^{-1}]_{jj}^{1/2}$
contains $\theta_{0,j}$ with probability tending to $1-\alpha$.
\item \textbf{Joint inference.} Let $\widehat L_n$ be the
lower-triangular Cholesky factor of $\widehat\calI_n$. Then
\[
 \widehat L_n^\top(\thetan-\theta_0)\Longrightarrow N_3(0,\Id_3),
 \qquad
 (\thetan-\theta_0)^\top\widehat\calI_n(\thetan-\theta_0)
 \Longrightarrow\chi^2_3.
\]
Hence the ellipsoid
$\{\theta:(\thetan-\theta)^\top\widehat\calI_n(\thetan-\theta)
\le\chi^2_{3,1-\alpha}\}$ is an asymptotic $1-\alpha$ confidence region
for $\theta_0$.
\item \textbf{Log scale.} Let $\widehat\Lambda_n=\diag(\thetan)$ and
$\widehat\calI_n^{\log}=\widehat\Lambda_n\widehat\calI_n\widehat\Lambda_n$.
Then (i)--(iii) hold for $\widehat\lambda_n-\lambda_0$, with
$\widehat\calI_n$ replaced by $\widehat\calI_n^{\log}$ and $J_n$ by
$J_n^{\log}$. In particular, the interval
\[
 \Bigl[\widehat\theta_{n,j}
 \exp\bigl\{-z_{1-\alpha/2}[(\widehat\calI_n^{\log})^{-1}]_{jj}^{1/2}\bigr\},\;
 \widehat\theta_{n,j}
 \exp\bigl\{z_{1-\alpha/2}[(\widehat\calI_n^{\log})^{-1}]_{jj}^{1/2}\bigr\}\Bigr]
\]
contains $\theta_{0,j}$ with probability tending to $1-\alpha$. When
$\widehat\calI_n$ is the observed information, $\widehat\calI_n^{\log}$
coincides with the observed information $-\nabla_\lambda^2l_n$ at
$\widehat\lambda_n$ whenever $\thetan\in\intt(\Theta)$, an event of
probability tending to one.
\end{enumerate}
On the event that $\widehat\calI_n$ is not positive definite, the
intervals and regions in (ii)--(iv) may be defined arbitrarily.
\end{corollary}

The estimated standard errors thus have relative error
$O_\PP(\epsilon_n)=O_\PP(b_n^{-p/2})$. At moderate $n$, the estimated standard
errors of the spatial parameters therefore carry nonnegligible
uncertainty, which is a further reason to form intervals on the log scale
and exponentiate them; this also keeps them positive. For
$\widehat\calI_n=\calI_n(\thetan)$, the proof gives the explicit bound
$\norm{D_n^{-1}\calI_n(\thetan)D_n^{-1}-J_n}_{\op}
\le C\epsilon_n\norm{D_n(\thetan-\theta_0)}$ for every $\theta_0\in\Theta$.
Combined with \Cref{rem:uniform}, the coverage statements hold uniformly
over $\theta_0$ in compact subsets of $\intt(\Theta)$.

\subsubsection{Random observation locations}
So far, all results are formulated for deterministic observation locations, which allows the likelihood analysis to condition on the design. In many applications, however, the locations themselves are naturally modeled as random. The next corollary shows that the deterministic-design formulation is not a substantive restriction: a broad class of random designs satisfies the required coverage condition (\Cref{ass:coverage}) almost surely, so the preceding results continue to hold.

\begin{corollary}[Independent random sampling]\label{cor:random}
Let $X_1,X_2,\ldots$ be independent and identically distributed locations
in a fixed bounded set $\calX$, independent of the process and the errors.
Suppose their distribution has a Lebesgue density $f$ satisfying
$f(x)\ge f_->0$ for almost every $x$ in a fixed cube $Q\subset\calX$. Write
$\mathbf X=(X_1,X_2,\ldots)$ for the design sequence. The quantities
$\calI_n(\theta_0)$, $J_n$, and $M_n$ in \Cref{thm:rbf} are evaluated at the
realized locations, and are therefore random.
\begin{enumerate}[label=(\alph*)]
\item \textbf{Conditionally on the design.} For almost every realization
of $\mathbf X$, \Cref{ass:coverage} holds, and \Cref{thm:rbf,thm:minimax}
and \Cref{cor:log-scale,cor:plugin} hold conditionally on $\mathbf X$. The
constants $c$ and $C$ in these results do not depend on $\mathbf X$; only
the thresholds $n_0$ and $n_\lambda$ may.
\item \textbf{Unconditionally.} Under the joint law of the design, the
process, and the errors, the following assertions hold.
\begin{enumerate}[label=(\roman*)]
\item \Cref{thm:rbf}(i) holds almost surely for all sufficiently large
$n$.
\item \Cref{thm:rbf}(ii) and (iii) hold.
\item \Cref{thm:rbf}(iv) holds with an additional term $Ce^{-cn^{3/4}}$
on the right-hand side of each tail bound.
\item In \Cref{thm:rbf}(v), the exponential moment bound holds, and
$\E g(M_n(\thetan-\theta_0))\to\E g(Z)$ for every test function $g$ as
in \Cref{thm:rbf}(v). In particular, $D_n\E(\thetan-\theta_0)\to0$ and
$\E[M_n(\thetan-\theta_0)(\thetan-\theta_0)^\top M_n^\top]\to\Id_3$.
\item \Cref{thm:minimax} holds, so the MLE attains the minimax rate.
\item \Cref{cor:plugin} holds, and (i)--(v) above hold on the log scale.
\end{enumerate}
\end{enumerate}
\end{corollary}
 
The remaining statements of \Cref{thm:rbf}(v) compare the MLE with the
design-dependent matrix $J_n$, through $\E g(J_n^{-1/2}Z)$ and $J_n^{-1}$.
They hold conditionally on the design by part (a). Unconditional versions
would require integrability of $J_n^{-1}$ over the design, which we do not
pursue: near-degenerate designs, in which all locations are close
together, have small but positive probability and make $J_n$ nearly
singular. Part (b)(iv) gives the standardized analogues instead. A
continuous density that is bounded away from zero on a neighborhood of $Q$
satisfies the assumptions of the corollary. Thus, the corollary covers
standard random designs whose sampling density is sufficiently well
behaved on some region of the domain $\calX$.

\subsection{Proof outline and key ideas}\label{sec:outline}

The proof of \Cref{thm:rbf} requires overcoming difficulties that make standard likelihood asymptotics inapplicable. Under fixed-domain asymptotics, the observations become increasingly strongly dependent, so the likelihood cannot easily be decomposed into a growing number of approximately independent contributions. Our main strategy is to exploit the analytic structure of the RBF kernel to recover, from the observations, a growing collection of latent Gaussian components whose conditional variance structure makes their information about the kernel parameters tractable. This leads to a global separation bound for
the observation laws at the precise scales of the optimal convergence rates. After this global step is established, consistency is obtained via the method of \citet{ibragimov1981}, and then a Taylor expansion of the local log likelihood is used to derive the asymptotic expansion and
distribution of the MLE.

\textbf{The power series of the field.}
Writing $\norm{x-y}^2=\norm x^2+\norm y^2-2\langle x,y\rangle$ and expanding
$e^{\langle x,y\rangle/l^2}$ by the multinomial theorem gives
\[
 K_\theta(x,y)=s\sum_{\alpha\in\N_0^p}\phi_{\alpha,l}(x)\phi_{\alpha,l}(y),
 \qquad
 \phi_{\alpha,l}(x)=e^{-\norm x^2/(2l^2)}\frac{x^\alpha}{l^{|\alpha|}\sqrt{\alpha!}}.
\]
Hence $F$ can be realized as the entire random function
\[
 F(x)=\sqrt s\,e^{-\norm x^2/(2l^2)}\sum_{\alpha\in\N_0^p}
 \xi_\alpha\frac{x^\alpha}{l^{|\alpha|}\sqrt{\alpha!}},
 \qquad\xi_\alpha\ \text{independent}\ N(0,1).
\]
Its normalized Taylor coefficients at a fixed point,
$\zeta_\alpha=\sqrt{\alpha!}\,[x^\alpha]F(x)$, are linear in the
$\xi_\gamma$. Multiplication by $e^{-\norm x^2/(2l^2)}$ raises degrees in
steps of two, so $\zeta_\alpha$ depends only on $\xi_\alpha$ and on
variables $\xi_\gamma$ of strictly lower degree. In degree order, the map
from the $\xi$'s to the $\zeta$'s is therefore lower triangular with
diagonal entries $\sqrt s\,l^{-|\alpha|}$. Consequently, given all normalized
coefficients of lower degree, $\zeta_\alpha$ is normal with variance
$s\,l^{-2|\alpha|}$. On the log scale, this conditional variance is
$\log s-2|\alpha|\log l$, which is linear in the degree with intercept
$\log s$ and slope $-2\log l$.

This triangular representation is the key structural feature of the proof: it reveals that the highly dependent Gaussian field is a relatively simple function of a sequence of Gaussian variables with simple conditional variance structure. If these normalized coefficients were directly observed, the different information scales for $s$ and $l$ would already be visible from this linear dependence on the degree. The difficulty is that the coefficients are latent, and not all of them can be recovered accurately from the $n$ noisy observations at the irregular locations.

\textbf{Estimating the normalized coefficients.}
The normalized coefficients are not observed, but they can be estimated by a fixed
linear statistic of the data. Recovering a growing number of them with sufficiently small error is a
central technical step, particularly because we assume only polynomial
coverage and allow irregular, clustered, and repeated observation
locations (\Cref{ass:coverage}). We fit a polynomial of
total degree $K_n\asymp b_n = \log n/\log\log n$ to the data by a discrete weighted least squares regression onto
a Legendre basis for $L^2(Q)$, and then read off its normalized Taylor coefficients. We show that all normalized coefficients through degree
$m_n\asymp b_n$ are recovered with error $n^{-c}$,
uniformly in $\theta$ (\Cref{lem:recovery}). There are on the order of $b_n^p$ such
coefficients. The estimator $\widehat\zeta^{(m_n)}=H_nY_n$ of the exact normalized coefficients $\zeta^{(m_n)} = (\zeta_{\alpha})_{|\alpha| \leq m_n}$ does not depend on
$\theta$. The fact that $H_n$ is independent of $\theta$ is essential: if one can distinguish parameters using the observed $\hat{\zeta}^{(m_n)} = H_nY_n$, as we will show, then one can distinguish between parameters using the full data $Y_n$.

\textbf{From coefficient recovery to global likelihood separation.}
For laws $P$ and $Q$, the affinity $\Aff(P,Q)=\int\sqrt{dP\,dQ}$ is the
Bhattacharyya coefficient, and $-\log\Aff(P,Q)$ is the Bhattacharyya
distance~\citep{bhattacharyya1943}. The lower triangular representation of $\zeta^{(m_n)}$ allows us to derive an explicit lower bound on the
Bhattacharyya distance between its laws under $\theta$ and $\theta_0$. The uniform recovery result then transfers this bound from the latent coefficients $\zeta^{(m_n)}$ to their estimators $\widehat\zeta^{(m_n)}$. Finally, because $\widehat\zeta^{(m_n)}=H_nY_n$ is a parameter-independent statistic of the observations, the data processing inequality transfers the bound to the observation laws themselves. This yields a global separation bound in terms of $(\log s-\log s_0)^2$ and
$(\log l-\log l_0)^2$.

For the nugget, the factorial decay in the Taylor series also makes the
kernel matrix numerically of rank $O(b_n^p)$ along a common subspace,
uniformly in $\theta$ (\Cref{lem:projection}). The data component $W_n$
orthogonal to that subspace is essentially white noise with variance
$t$, and contributes a Bhattacharyya distance of order $n(\log t-\log t_0)^2$. 

The proof therefore separates two fundamentally different sources of
information: only $O(b_n^p)$ effective signal directions carry
information about the spatial parameters, whereas the remaining
directions provide essentially $n$ independent pieces of information
about the nugget. This distinction is what ultimately produces the
logarithmic rates for $(s,l)$ and the parametric $n^{-1/2}$ rate for $t$.

Putting the signal and nugget Bhattacharyya distance bounds together yields that in
rescaled coordinates $h=D_n(\theta-\theta_0)$, for all $\theta\in\Theta$,
\[
 -\log\Aff(P_{\theta,n},P_{\theta_0,n})
 \ge c\bigl\{\min(\norm{h_\spat}^2,b_n^{p/2}\norm{h_\spat})+h_t^2\bigr\}-o(1).
\]
In the case of i.i.d. data, such a lower bound comes very easily, but in this case, it is one of the hardest parts of the proof. Crucially, the bound holds uniformly over the entire parameter space, that is, the bound is uniform in $\theta_0$.

\textbf{From global separation to the limit theorem.}
The global separation bound first controls the global maximizer of the
likelihood. A chaining argument based on \citet{ibragimov1981} converts
this separation into the exponential tail bounds for
$\widehat h_n=D_n(\thetan-\theta_0)$ in \Cref{thm:rbf}(iv), thereby
establishing consistency at the rates \eqref{eq:rates} before any local
likelihood approximation is invoked. This localization is essential:
without first ruling out distant likelihood maximizers, a quadratic
expansion around $\theta_0$ alone cannot establish the behavior of the
global MLE.

Once the MLE has been localized, we establish a uniform quadratic expansion of the log likelihood on the relevant rescaled neighborhoods (\Cref{lem:quadratic}). The normalized score is a vector of Gaussian quadratic forms whose individual spectral contributions vanish, allowing
a Lyapunov central limit theorem argument to establish asymptotic normality of the normalized score, despite the strong dependence of the original observations. Together, these three ingredients---global separation and localization,
a uniform local quadratic approximation, and a central limit theorem for the normalized score---yield \eqref{eq:mle-expansion} and the joint asymptotic normality in \Cref{thm:rbf}(iii).

 \section{Simulation}\label{sec:simulation}

 We examine how closely the asymptotic theory in \Cref{thm:rbf} describes
 the MLE at finite sample sizes. Because the convergence rates for
 $\sigma^2$ and $l$ are logarithmic, we consider sample sizes spanning several orders of magnitude.
 Code to reproduce all results in this section, together with the
 simulation output, is available at
 \url{https://github.com/Ameer-eng/Simulation-Studies-for-Euclidean-RBF-MLE}.

\subsection{Simulation setup}\label{sec:sim-setup}

\textbf{Model and designs.}
We simulate from \eqref{eqn:observation} with the RBF kernel on the unit
cube $\calX=[0,1]^p$, for $p\in\{1,2,3\}$, with true parameter
\[
 \theta_0=(\sigma_0^2,l_0,\tau_0^2)=(1,\ 0.25,\ 0.01).
\]
This gives a signal-to-noise ratio $\sigma_0^2/\tau_0^2=100$, and the
lengthscale is one quarter of the side length of the domain. The locations
form a regular midpoint grid with $G$ points per axis,
\[
 x\in\{(i-\tfrac12)/G:\ i=1,\ldots,G\}^p,
\]
so that $n=G^p$. This design satisfies \Cref{ass:coverage} with
$\kappa=1/p$. For each $p$, we use nine sample sizes ranging from
approximately $10^2$ to $10^6$:
\begin{align*}
 G&\in\{10^2,\,3\cdot10^2,\,10^3,\ldots,3\cdot10^5,\,10^6\},
 &&p=1,\\
 G&\in\{10,18,32,56,100,178,316,562,1000\},
 &&p=2,\\
 G&\in\{5,7,10,14,22,32,46,68,100\},
 &&p=3.
\end{align*}
For every $(p,n)$, we generate $1000$ independent data sets, each with
a new realization of the field $F$ and measurement errors.

\textbf{Likelihood evaluation and MLE.}
Both data generation and likelihood evaluation exploit the power-series
representation of the RBF kernel and, for $p\ge2$, the Kronecker structure
of the regular grid. This avoids forming the full $n\times n$ covariance
matrix and makes likelihood optimization feasible for sample sizes up to
$n=10^6$. The MLE $\widehat\theta_n$ is computed over
\[
 \Theta=[10^{-2},10^2]\times[0.05,2]\times[10^{-4},1],
\]
which contains $\theta_0$ well inside its interior. Of the
$2.7\times10^4$ fitted data sets, only three give an estimate on the
boundary of $\Theta$: one estimate of $l$ for $p=1$ and $n=10^2$, and two
estimates of $\tau^2$ for $p=3$ and $n=125$.
Details of the likelihood computation, optimization,
and numerical validation are provided in Appendix~\ref{apdx:sim}.

\textbf{Error summaries.}
For each parameter $\theta_j\in\{\sigma^2,l,\tau^2\}$, $j=1,2,3$, we
report three quantities as functions of $n$.
\begin{enumerate}[label=(\alph*)]
\item The \emph{root MSE} (RMSE),
\[
 \mathrm{RMSE}_j
 =
 \left\{
 \frac1R\sum_{r=1}^R
 (\widehat\theta_j^{(r)}-\theta_{0j})^2
 \right\}^{1/2},
 \qquad R=1000.
\]
By \Cref{thm:rbf}(v), $\mathrm{RMSE}_j$ has the same order as the
corresponding entry of $D_n^{-1}$.

\item The \emph{robust standard deviation} $\mathrm{IQR}_j/1.349$, where
$\mathrm{IQR}_j$ is the interquartile range of the replicate errors
$\widehat\theta_j^{(r)}-\theta_{0j}$. The factor $1.349$ is the
interquartile range of a standard normal distribution, so
$\mathrm{IQR}/1.349$ equals the standard deviation under normality.
Unlike the RMSE, this quantity is insensitive to a small number of
extreme estimates.

\item The \emph{asymptotic standard deviation}
$[\calI_n(\theta_0)^{-1}]_{jj}^{1/2}$, where $\calI_n(\theta_0)$ is the
Fisher information in \eqref{eq:fisher}. By
\Cref{thm:rbf}(iii)--(v), the MLE is approximately normal with this
standard deviation, whose order is
$b_n^{-p/2}$, $b_n^{-(p+2)/2}$, and
$\sqrt{2}\tau_0^2n^{-1/2}$ for $\sigma^2$, $l$, and $\tau^2$,
respectively. This quantity is computed exactly for the actual grid
design at each $n$.
\end{enumerate}
Bootstrap $95\%$ intervals for the first two quantities are obtained by
resampling the replicates.

\textbf{Log parametrization.}
We report the results on the log scale. As shown in
\Cref{cor:log-scale}, the log transformation preserves the convergence
rates and yields the parameter-free nugget limit
\[
 \sqrt n(\log\widehat\tau^2-\log\tau_0^2)
 \Longrightarrow N(0,2).
\]
Results on the original scale are reported in Appendix~\ref{apdx:sim-original}
and lead to the same conclusions about the convergence rates.

\begin{figure}[!h]
\centering
\includegraphics[width=\textwidth]{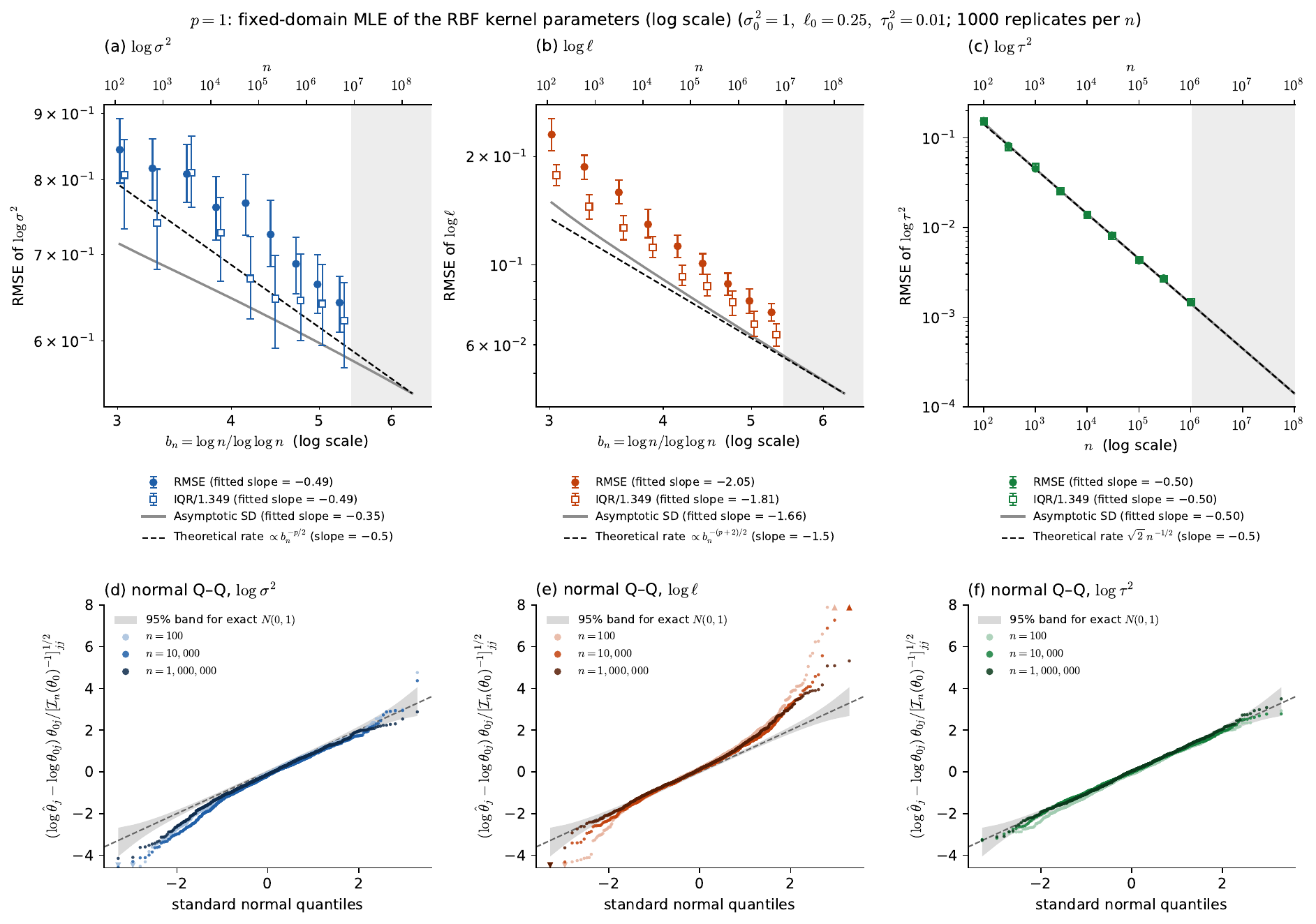}
\caption{Simulation results for $p=1$ on the log scale, with
$\theta_0=(1,0.25,0.01)$ and $1000$ replicates per $n$. Top: RMSE
(filled circles), robust standard deviation $\mathrm{IQR}/1.349$
(open squares), both with bootstrap $95\%$ intervals, and the exact
asymptotic standard deviation
$[\calI_n(\theta_0)^{-1}]_{jj}^{1/2}/\theta_{0j}$ (grey line). The grey region indicates sample sizes beyond the Monte Carlo range,
where repeated MLE computation is computationally infeasible but the Fisher information can still be evaluated. Dashed lines
show the theoretical rates in \Cref{thm:rbf}(ii). Bottom: normal Q--Q plots of the
standardized errors, with a $95\%$ band for exact normality; triangles
mark points outside the frame.}
\label{fig:sim-p1}
\end{figure}

\begin{figure}[!h]
\centering
\includegraphics[width=\textwidth]{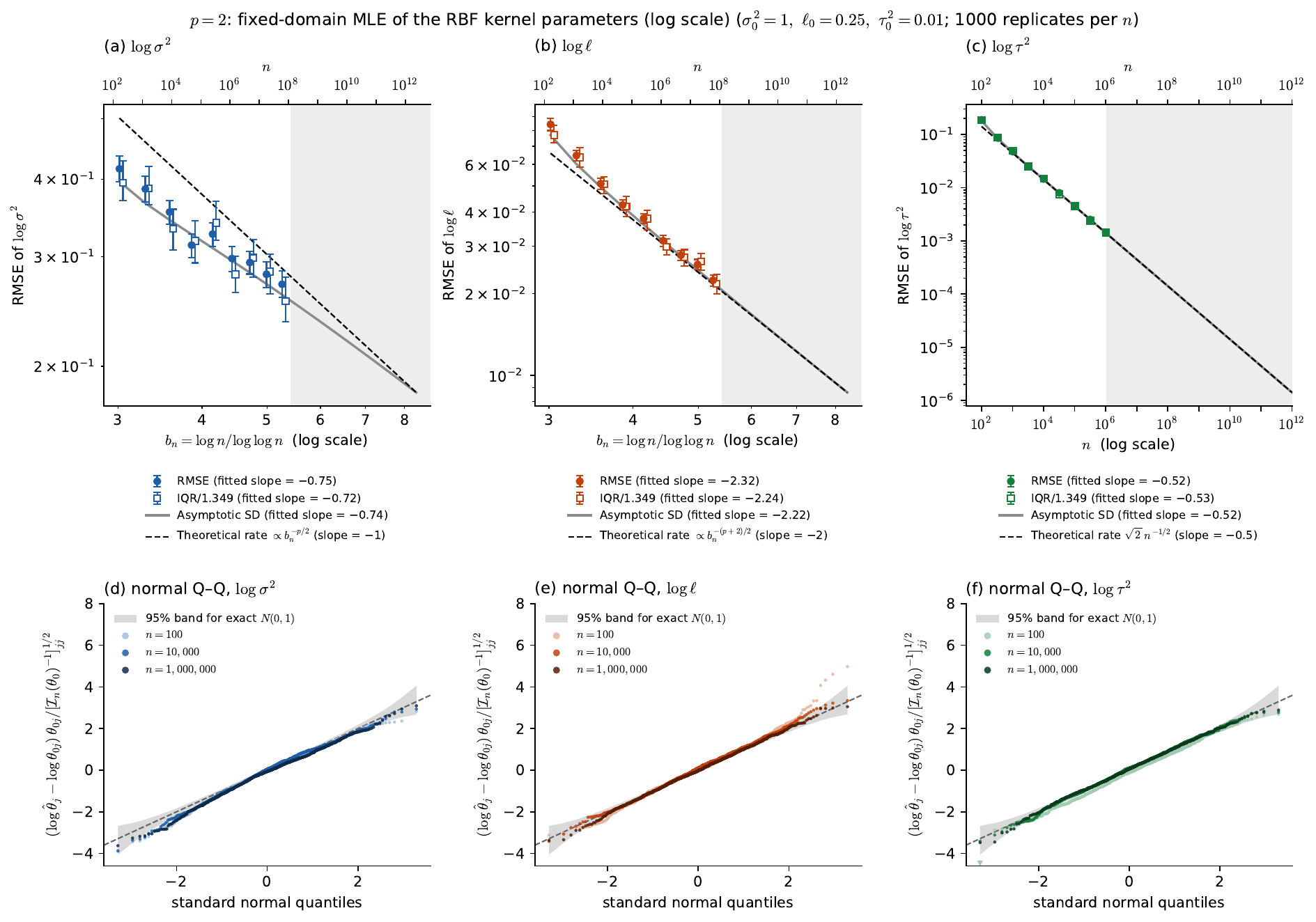}
\caption{Simulation results for $p=2$, with the layout and graphical elements as in \Cref{fig:sim-p1}.}
\label{fig:sim-p2}
\end{figure}

\begin{figure}[!h]
\centering
\includegraphics[width=\textwidth]{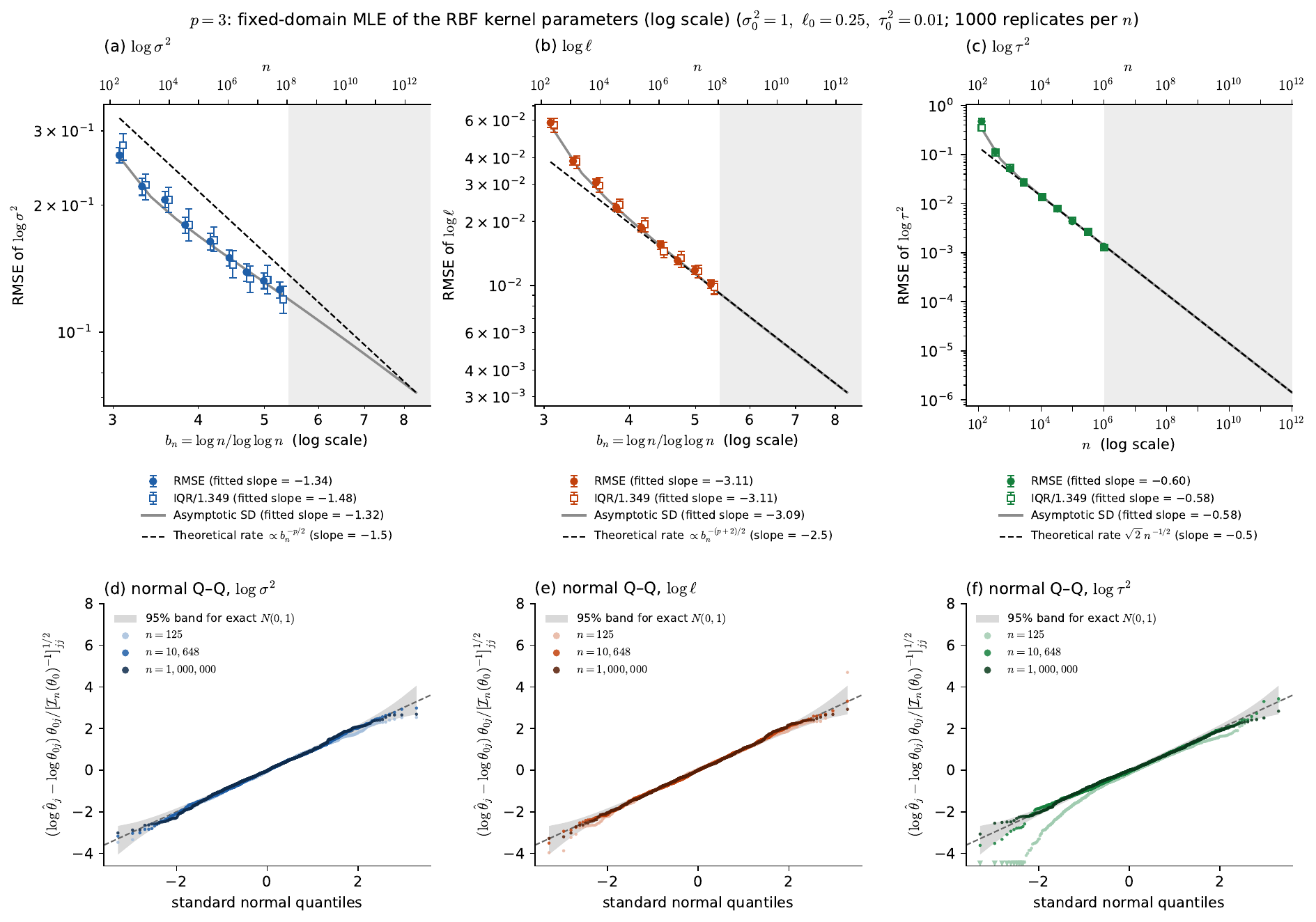}
\caption{Simulation results for $p=3$, with the layout and graphical elements as in \Cref{fig:sim-p1}.}
\label{fig:sim-p3}
\end{figure}

\subsection{Simulation results}\label{sec:sim-results}
\textbf{Graphical summaries.}
\Cref{fig:sim-p1,fig:sim-p2,fig:sim-p3} show the results for
$p=1,2,3$. The top row compares the Monte Carlo RMSE and robust standard
deviation with the asymptotic standard deviation and the theoretical
rates in \Cref{thm:rbf}(ii). The bottom row shows normal Q--Q plots of
the standardized errors at three representative sample sizes, providing
a direct assessment of the asymptotic normality in
\Cref{thm:rbf}(iii).

\textbf{The rates agree with the theory.}
In all three dimensions and for all three parameters, the Monte Carlo
errors decrease at rates consistent with \Cref{thm:rbf} and closely track
the exact asymptotic standard deviations. For $p=3$, the RMSEs of
$\log\widehat\sigma^2$ and $\log\widehat l$ are within $7\%$ of their
asymptotic standard deviations at every $n$ from $125$ to $10^6$. For
$p=2$, they are within $10\%$ at every $n$ and within $6\%$ for
$n\ge10^3$. For $p=1$, the RMSEs are somewhat larger but approach the
asymptotic standard deviations as $n$ increases.

\Cref{tab:sim-slopes} compares slopes fitted to the simulated errors with
those of the exact asymptotic standard deviation over the same range of
$n$, as well as with the limiting exponents in \Cref{thm:rbf}. For
$p=2,3$, the simulated RMSE slopes agree with those of the asymptotic standard
deviation to within $0.1$. The finite-sample slopes need
not yet equal their limiting values because $b_n=\log n/\log\log n$
grows extremely slowly; this issue is examined further in
Appendix~\ref{apdx:sim-spatial}.

\begin{table}[!h]
\centering
\caption{Least-squares slopes of $\log(\text{error})$ against $\log b_n$ for $\log\sigma^2$ and $\log l$, and against $\log n$ for $\log \tau^2$, over the nine simulated sample sizes ($10^2\lesssim n\le10^6$). The last
column gives the theoretical rate shown in \Cref{thm:rbf}.}
\label{tab:sim-slopes}
\vspace{3pt}
\begin{tabular}{llrrrr}
\toprule
& & RMSE & $\mathrm{IQR}/1.349$ & Asymptotic SD & Theory\\
\midrule
$\log\sigma^2$ & $p=1$ & $-0.49$ & $-0.49$ & $-0.35$ & $-0.5$\\
 & $p=2$ & $-0.75$ & $-0.72$ & $-0.74$ & $-1$\\
 & $p=3$ & $-1.34$ & $-1.48$ & $-1.32$ & $-1.5$\\
\midrule
$\log l$ & $p=1$ & $-2.05$ & $-1.81$ & $-1.66$ & $-1.5$\\
 & $p=2$ & $-2.32$ & $-2.24$ & $-2.22$ & $-2$\\
 & $p=3$ & $-3.11$ & $-3.11$ & $-3.09$ & $-2.5$\\
 \midrule
$\log \tau^2$ & $p=1$ & $-0.50$ & $-0.50$ & $-0.50$ & $-0.5$\\
 & $p=2$ & $-0.52$ & $-0.53$ & $-0.52$ & $-0.5$\\
 & $p=3$ & $-0.60$ & $-0.58$ & $-0.58$ & $-0.5$\\
\bottomrule
\end{tabular}
\end{table}

The contrast between the spatial parameters and the nugget is substantial.
From $n=10^2$ to $n=10^6$, the RMSE of $\log\widehat\sigma^2$ decreases
only by factors of $1.3$, $1.5$, and $2.1$ for $p=1,2,3$, respectively,
whereas the RMSE of $\log\widehat\tau^2$ decreases by factors of about
$110$, $130$, and $360$. This is the slow logarithmic behavior predicted by the
theory.

\textbf{The errors are close to normal.}
For $p=2$ and $p=3$, the standardized errors of
$\log\widehat\sigma^2$ and $\log\widehat l$ lie essentially within the
$95\%$ band for exact normality at the displayed sample sizes. The RMSE
and robust standard deviation are also very similar, providing further
evidence for the Fisher-normalized limit in \eqref{eq:main-an}. The
nugget errors show similarly good agreement except at the smallest
sample size for $p=3$. Convergence to normality is slower for $p=1$,
particularly for the lengthscale, where a heavier right tail remains
visible even at the largest simulated sample sizes. Additional normality
diagnostics are reported in Appendix~\ref{apdx:sim-normality}.

\textbf{The nugget is the easiest parameter.}
The nugget variance is estimated at the parametric rate. In every
dimension,
$\sqrt n\,\mathrm{RMSE}(\log\widehat\tau^2)$ is within $6\%$ of its
limit $\sqrt2$ for all $n\ge3\times10^3$. On the original scale, for
example,
\[
 \sqrt n\,\mathrm{RMSE}(\widehat\tau^2)=0.0143
\]
at $n=10^6$ for $p=2$, compared with the theoretical limit
$\sqrt2\,\tau_0^2=0.01414$. The approximation is less accurate at the
smallest sample sizes, especially for $p=3$. This behavior is consistent
with the low-rank structure underlying \Cref{thm:rbf}: most directions
eventually contain essentially only nugget noise, allowing $\tau^2$ to
be estimated at the usual parametric rate. A more detailed finite-sample
explanation is given in Appendix~\ref{apdx:sim-nugget}.

\textbf{Convergence rates of the spatial parameters are slow.}
For the spatial parameters $\sigma^2$ and $l$, the effective amount of information grows
with the number of recoverable Taylor coefficients, whose maximal degree
is of order $b_n=\log n/\log\log n$. This quantity grows very slowly:
over $10^2\le n\le10^6$, it increases only from approximately $3.0$ to
$5.3$. Consequently, the finite-sample slopes in
\Cref{tab:sim-slopes} need not closely match the limiting exponents even
when the MLE already tracks the Fisher-information benchmark well.
Indeed, the exact asymptotic standard deviation itself exhibits the same
finite-sample deviations from the limiting slopes. Further analysis of
these effects is provided in Appendix~\ref{apdx:sim-spatial}.

The different spatial rates are nevertheless already clearly visible.
Since $D_n$ contains an additional factor $b_n$ for $l$, the lengthscale
is estimated more accurately than the spatial variance by a factor of
order $b_n$, in agreement with the simulations.

\textbf{Higher dimensions help.}
The rates $b_n^{-p/2}$ and $b_n^{-(p+2)/2}$ imply that estimation of
$\sigma^2$ and $l$ becomes easier as the dimension $p$ increases. This
may initially appear counterintuitive, since all observations still come
from a single realization of the field. The explanation is that a
higher-dimensional RBF field contains more recoverable Taylor
coefficients. The number of coefficients of degree at most $m$ is
\[
 N_p(m)=\binom{m+p}{p},
\]
which is asymptotically of order $m$, $m^2/2$, and $m^3/6$ for
$p=1,2,3$, respectively. Thus a single realization contains more
independent information about the spatial kernel parameters as $p$
increases. This is reflected both in the smaller estimation errors and
in the improved normal approximation from $p=1$ to $p=3$. In contrast,
the nugget information is of order $n$ in every dimension.

\section{Discussion and future work}\label{sec:discussion}

In this article, we establish an asymptotic theory for MLEs of RBF kernel parameters under fixed-domain asymptotics. We prove consistency of the joint MLE of the spatial variance, lengthscale, and nugget variance, derive their convergence rates, establish joint asymptotic normality, and show that all three rates are minimax optimal. To the best of our knowledge, these results provide the first complete asymptotic characterization of the joint MLE for the RBF kernel in this setting. The results provide theoretical justification for the MLE routinely used in GP software. At the same time, they reveal an important limitation: while the nugget variance converges at the standard parametric rate $n^{-1/2}$, the spatial variance and lengthscale converge only at logarithmic rates. Since these rates are sharp, the slow convergence is not merely an artifact of the MLE or our analysis. In practice, substantial uncertainty in the estimated spatial variance and lengthscale may therefore remain, even with a large number of densely sampled observations, and these estimates should be interpreted with caution when they are used for inference or downstream tasks.

Beyond classical likelihood inference, our theoretical results may also inform modern spatial machine learning techniques, such as the Bayesian predictive stacking framework introduced by \citet{zhang2026bayesian}. Predictive stacking bypasses full posterior sampling by optimizing ensemble weights across a discretized parameter grid using predictive scoring rules. Our analysis establishes explicit convergence rates, namely $b_n^{-p/2}$ for the spatial variance, $b_n^{-(p+2)/2}$ for the lengthscale, and $n^{-1/2}$ for the nugget, together with exponential tail bounds. These suggest a principled guide for constructing the parameter grid: matching the grid spacing for each parameter to its convergence rate should allow the stacking weights to concentrate around the true parameter values without over-discretization, and without MCMC, even in high-dimensional spatial settings. Making this precise is an interesting future direction.

There are, however, some limitations of our study that motivate important directions for future work. First, our results concern the exact MLE, whose computation requires $O(n^3)$ operations and becomes prohibitive for large $n$. A large literature has developed scalable GP approximations in spatial statistics and nonparametric regression, including inducing-point, variational inference, nearest-neighbor, and Vecchia-type methods, which replace the exact likelihood with computationally tractable approximations \citep[see, e.g.,][and references therein]{quinonero2005unifying, banerjee2008gaussian, titsias2009variational, datta2016hierarchical, szabo2026vecchia}. Focused reviews of different classes of scalable GPs are provided by \citet{banerjee2017high} and \citet{liu2020gaussian}. Our simulation studies could compute the exact MLE at the sample sizes needed to examine the asymptotic behavior only by exploiting special structure, namely the power-series representation of the RBF kernel and, for $p\ge2$, the Kronecker structure of a regular grid (Appendix~\ref{apdx:sim-lik}); such structure is not available for general kernels and designs. Thus, an important next question is how these approximations affect kernel parameter inference. In particular, it remains to determine which scalable approximations yield consistent estimators, whether they preserve the convergence rates and asymptotic distributions of the exact MLE, and how the approximation accuracy or computational complexity must scale with $n$ to retain these statistical properties. Conversely, aggressive approximations may lose information about the kernel parameters and lead to slower rates or even inconsistency. Developing a statistical theory connecting computational approximation to kernel parameter inference would, therefore, provide useful guidance for scalable GP methodology.

A second limitation is that our analysis is specific to the RBF kernel. A key ingredient of our proof is its analytic structure, which allows a growing collection of normalized Taylor coefficients to be recovered from the observations with sufficiently small error. This argument does not extend directly to kernels with finite smoothness, most notably the widely used Mat\'ern family.    While \citet{tang2021identifiability} established the consistency of the
   MLEs of the microergodic parameter (see
   \Cref{sec:related-estimation}) and the nugget for the Mat\'ern kernel
   with $p\le3$ and a known smoothness parameter under additional
   strong assumptions, a complete asymptotic theory of the joint MLE remains an open
   challenge requiring different mathematical tools. More generally, understanding how kernel smoothness determines the amount of recoverable information and the resulting behavior of likelihood-based kernel parameter estimation is an important direction for future work.

Finally, our theoretical framework assumes a bounded Euclidean domain $\calX\subset\R^p$. Extending sharp fixed-domain MLE asymptotics for the RBF kernel to non-Euclidean spaces, such as compact Riemannian manifolds, building upon recent advances in parameter identifiability of manifold GPs~\citep{li2023inference}, presents another promising avenue for future research.

\appendix

\section*{Appendices: Overview}

In Appendix \ref{app:statements}, we state key lemmas to prove the main theorems, together with auxiliary lemmas to prove the key lemmas. In Appendix \ref{app:proof-main}, we prove the main theorems and their corollaries using the lemmas in Appendix \ref{app:statements}. In Appendix \ref{app:lemma-proofs}, we prove all lemmas in Appendix \ref{app:statements}. Finally, in Appendix \ref{apdx:sim}, we present additional experimental details for simulations in \Cref{sec:simulation}.

\setcounter{section}{0}

\section{Lemmas}\label{app:statements}

This appendix first records elementary Gaussian facts and a maximal inequality for continuous processes, and then states the auxiliary lemmas used in the proofs of \Cref{thm:rbf,thm:minimax} and their corollaries.
The lemma proofs appear in \Cref{app:lemma-proofs}. All design matrices and linear statistics introduced below are deterministic functions of the locations. 

\subsection{Gaussian affinity and elementary bounds}\label{app:gaussian-affinity}

For probability laws $P,Q$, define
\[
 \Aff(P,Q)=\int\sqrt{dP\,dQ},
 \qquad \TV(P,Q)=\sup_A|P(A)-Q(A)|.
\]
The affinity is also called the Bhattacharyya coefficient, and
$-\log\Aff(P,Q)$ is the Bhattacharyya distance.
For positive definite $d\times d$ covariance matrices $A,B$, abbreviate
$\Aff(A,B)=\Aff(N_d(0,A),N_d(0,B))$. Direct Gaussian integration gives
\begin{equation}\label{eq:gaussian-affinity}
 \Aff(A,B)=
 \frac{\det(A)^{1/4}\det(B)^{1/4}}
 {\det((A+B)/2)^{1/2}}.
\end{equation}
Applying the same statistic to both laws cannot decrease their
affinity.

\begin{lemma}[Monotonicity of affinity]\label{lem:affinity-monotone}
Let $P,Q$ be probability laws on a measurable space
$(\mathcal Y,\mathcal F)$, and let $T:\mathcal Y\to\mathcal Y'$ be a
measurable map into another measurable space. Then
\[
 \Aff(P\circ T^{-1},Q\circ T^{-1})\ge\Aff(P,Q).
\]
In particular, if $Y\sim N_d(0,A)$ under one law and $Y\sim N_d(0,B)$ under
the other, and $H$ is a fixed $r\times d$ matrix of rank $r$, then
$\Aff(HAH^\top,HBH^\top)\ge\Aff(A,B)$.
\end{lemma}

In the applications, $T$ is a linear statistic of the data that does not
depend on the parameter. Both uses are in the proof of
\Cref{lem:affinity}: the coefficient statistic of \Cref{lem:recovery}, and
the noise block $W_n$.

The next two lemmas are elementary. \Cref{lem:qf-moments} controls all
moments of the scores and is used in \Cref{lem:lr-process};
\Cref{lem:fraction} is used in the proof of \Cref{lem:affinity} to show
that a fixed fraction of polynomial degrees is separated.

\begin{lemma}[Moments of Gaussian quadratic forms]\label{lem:qf-moments}
Let $X\sim N_d(0,\Id_d)$ and let $B$ be a symmetric $d\times d$ matrix.
For every $q\ge1$ there is $C_q$, depending only on $q$, such that
$\E|X^\top BX-\tr B|^q\le C_q\norm{B}_\Frob^q$.
\end{lemma}

\begin{lemma}[A Paley--Zygmund-type bound]\label{lem:fraction}
Let $X$ be a random variable with $|X|\le M$ almost surely and
$\E X^2\ge\sigma^2>0$. Then, for every $\lambda\in[0,1)$,
\[
 \PP(|X|\ge\lambda\sigma)\ge(1-\lambda^2)\frac{\sigma^2}{M^2}.
\]
In particular, if $X=f(\iota)$ for an index $\iota$ uniformly distributed
on a finite set $\mathcal A$, then at least
$(1-\lambda^2)\sigma^2|\mathcal A|/M^2$ elements $\iota\in\mathcal A$ satisfy
$|f(\iota)|\ge\lambda\sigma$.
\end{lemma}

\subsection{A maximal inequality for continuous processes}\label{app:maximal}

The exponential tail bounds rest on the following chaining inequality,
a variant of the shell argument of 
\citet[Theorem~I.5.1]{ibragimov1981}.
\Cref{lem:kolmogorov,lem:convex-sup} are the building blocks,
\Cref{prop:cover} is the general bound, and \Cref{cor:envelope} is the
form used in \Cref{lem:tail}. Throughout, $m\ge1$, $\norm{\cdot}_{L^m}$
refers to a fixed probability measure, $|\cdot|$ is the Euclidean norm,
and a supremum over the empty set is zero.

\begin{lemma}[Kolmogorov continuity criterion]\label{lem:kolmogorov}
Let $k\in\N$, $m\ge1$, and $\alpha>k$. If a real process
$(\xi(x))_{x\in[0,1]^k}$ satisfies
$\norm{\xi(x)-\xi(y)}_{L^m}\le L|x-y|^{\alpha/m}$ for all $x,y$, then it
has a continuous modification, again denoted $\xi$, such that
\[
 \Bigl\|\sup_{x,y\in[0,1]^k,\,|x-y|\le h}|\xi(x)-\xi(y)|\Bigr\|_{L^m}
 \le C_{k,m,\alpha}Lh^{(\alpha-k)/m},\qquad h\in(0,1].
\]
\end{lemma}

\begin{lemma}[Supremum over a compact convex set]\label{lem:convex-sup}
Let $k\in\N$ and $m>k$. Let $\Gamma\subseteq\R^k$ be a nonempty compact
convex set of diameter at most $\delta>0$, and let $(\xi(u))_{u\in\Gamma}$
be a real process with continuous paths such that
$\sup_{u\in\Gamma}\norm{\xi(u)}_{L^m}\le M$ and
$\norm{\xi(u)-\xi(v)}_{L^m}\le L|u-v|$ for $u,v\in\Gamma$. Then
\[
 \E\sup_{u\in\Gamma}|\xi(u)|^m
 \le C_{k,m}\bigl\{M^m+(L\delta)^kM^{m-k}\bigr\}.
\]
\end{lemma}

\begin{proposition}[Maximal inequality over a cover]\label{prop:cover}
Let $k\in\N$, $m>k$, $U\subseteq\R^k$, and let $(\xi(u))_{u\in U}$ be a
real process with continuous paths. Let $(\Gamma_j)_{j\in\mathcal J}$ be a
countable family of nonempty compact convex subsets of $U$ of diameters
at most $\delta_j$. Suppose that
$\sup_{u\in\Gamma_j}\norm{\xi(u)}_{L^m}\le M_j$ and
$\norm{\xi(u)-\xi(v)}_{L^m}\le L_j|u-v|$ for $u,v\in\Gamma_j$. Then
\[
 \E\sup_{u\in\bigcup_j\Gamma_j}|\xi(u)|^m
 \le C_{k,m}\sum_{j\in\mathcal J}
 \bigl\{M_j^m+(L_j\delta_j)^kM_j^{m-k}\bigr\}.
\]
\end{proposition}

Only Lipschitz control inside each piece is required, so neither product
structure nor convexity of $U$ itself is needed. Taking the pieces to be
annular shells gives bounds of Ibragimov--Has'minskii type. For the
Bernstein-type envelopes arising here, unit cubes are more convenient.

\begin{corollary}[Bernstein-type envelopes]\label{cor:envelope}
Let $k=k_1+k_2$ with $k_1,k_2\in\N$, let $m>k$, and let
$U\subseteq\R^{k_1}\times\R^{k_2}$ be a product of compact intervals.
For $\gamma\in[1,\infty]$ and $r\ge0$ put $\psi_\gamma(r)=\min(r^2,\gamma r)$,
so that $\psi_\infty(r)=r^2$. Let $(\xi(u))_{u\in U}$ have continuous
paths, and suppose that, for constants $A\ge0$, $a>0$, $B\ge0$, and
$\gamma_1,\gamma_2\in[1,\infty]$,
\begin{align}
 \norm{\xi(u)}_{L^m}^m
 &\le A\exp\bigl[-a\{\psi_{\gamma_1}(|u_1|)+\psi_{\gamma_2}(|u_2|)\}\bigr],
 &&u\in U,\label{eq:envelope-moment}\\
 \norm{\xi(u)-\xi(v)}_{L^m}&\le B|u-v|,
 &&u,v\in U.\label{eq:envelope-lipschitz}
\end{align}
Put $a''=a(m-k)/(32m)$. Then, for all $H_1,H_2\ge0$,
\begin{equation}\label{eq:envelope-bound}
 \E\sup_{\substack{u\in U\\|u_1|\ge H_1,\,|u_2|\ge H_2}}|\xi(u)|^m
 \le C\bigl(A+B^kA^{1-k/m}\bigr)
 \exp\bigl[-a''\{\psi_{\gamma_1}(H_1)+\psi_{\gamma_2}(H_2)\}\bigr],
\end{equation}
where $C$ depends only on $k_1$, $k_2$, $m$, and $a$, and not on
$\gamma_1$, $\gamma_2$, or $U$.
\end{corollary}

\subsection{Auxiliary lemmas}\label{app:aux-lemmas}

The auxiliary lemmas follow the stages of the proof outline in
\Cref{sec:outline}.
\begin{itemize}
\item \Cref{lem:projection} splits the likelihood into a low-dimensional
 signal block and a noise block, and supplies the covariance derivative
 bounds used throughout.
\item \Cref{lem:recovery,lem:affinity} give global identifiability.
 \Cref{lem:recovery} recovers normalized Taylor coefficients of the field
 from the data, and \Cref{lem:affinity} turns their triangular covariance
 structure into an affinity bound between the observation laws.
\item \Cref{lem:lr-process}, \Cref{prop:cover}, and \Cref{cor:envelope}
 convert that affinity bound into the exponential tail bound of
 \Cref{lem:tail}, which is \Cref{thm:rbf}(iv).
\item \Cref{lem:information,lem:quadratic,lem:score} give the information
 bounds, the local quadratic expansion, and the score central limit
 theorem, which together prove parts (i)--(iii) of \Cref{thm:rbf}. Part
 (v) follows from parts (iii) and (iv).
\item \Cref{lem:random} is used only for \Cref{cor:random}.
\end{itemize}

The RBF kernel is numerically of low rank on a bounded
set. Its Taylor coefficients decay factorially, so only polynomial
features of degree $O(b_n)$ are visible at the noise level. The next
lemma makes this precise with a single projection that works for all
parameters simultaneously.

\begin{lemma}[Common projection and covariance derivatives]\label{lem:projection}
Assume \Cref{ass:parameter} and bounded observation locations. For any
fixed $A>0$, there is a parameter-independent orthogonal projection $P_n$
of rank $r_n=O(b_n^p)$ such that, with $S_n(\theta)=sR_n(l)$,
\begin{equation}\label{eq:projection-approx}
 \sup_{\theta\in\Theta}
 \norm{\partial_\theta^\gamma
 (S_n-P_nS_nP_n)}_{\op}\le n^{-A},
 \qquad |\gamma|\le3.
\end{equation}
For $\widetilde C_n=t\Id_n+P_nS_nP_n$ and its log likelihood,
\begin{equation}\label{eq:likelihood-approx}
 \sup_{\theta\in\Theta}
 |\partial_\theta^\gamma(l_n-\widetilde l_n)|
 =O_{\PP}(n^{4-A}),\qquad |\gamma|\le3.
\end{equation}

For a spatial derivative with $a$ derivatives in $s$ and $c$ in $l$,
$1\le a+c\le3$, a nonzero derivative has $a\le1$ and satisfies
\begin{align}
 \sup_\theta\norm{C_n^{-1/2}
   \partial_s^a\partial_l^c C_n C_n^{-1/2}}_{\op}
 &\le C b_n^c,\label{eq:raw-operator}\\
 \sup_\theta\norm{C_n^{-1/2}
   \partial_s^a\partial_l^c C_n C_n^{-1/2}}_{\Frob}
 &\le C b_n^{p/2+c}.\label{eq:raw-frobenius}
\end{align}
The same bounds hold after compression. Precisely, let $U$ be any
$n\times r$ matrix with orthonormal columns, $1\le r\le n$, that does not
depend on $\theta$, and let $A_\theta=U^\top C_n(\theta)U$ be the compressed
covariance. Then \eqref{eq:raw-operator} and \eqref{eq:raw-frobenius}
hold with $C_n$ replaced by $A_\theta$, that is, for the compressed
derivatives $\partial_s^a\partial_l^cA_\theta
=U^\top(\partial_s^a\partial_l^cC_n)U$ whitened by $A_\theta^{-1/2}$, with
constants that do not depend on $U$. In the coordinates
$\theta=\theta_0+D_n^{-1}h$, normalized spatial derivatives of total
order $k\le3$, whitened either by $C_n$ or by a compression $A_\theta$,
have respective bounds $C\epsilon_n^k$ in operator norm and
$C\epsilon_n^{k-1}$ in Frobenius norm. The first normalized nugget derivative has operator
norm $O(n^{-1/2})$ and Frobenius norm $O(1)$, and all covariance
derivatives of order at least two involving $t$ vanish.
\end{lemma}

In coordinates adapted to this projection, the reduced likelihood
separates into a signal block of dimension $r_n$ and a noise block of
dimension $d_n=n-r_n$ that depends only on the nugget. Fix a projection
from \Cref{lem:projection} with $A=20$. Choose orthonormal
matrices $U_n,V_n$ spanning its range and orthogonal complement, and write
\begin{equation}\label{eq:block-notation}
 Z_n=U_n^\top Y_n,\quad W_n=V_n^\top Y_n,\quad d_n=n-r_n,\quad
 A_\theta=t\Id_{r_n}+sU_n^\top R_n(l)U_n.
\end{equation}
The reduced negative twice log likelihood decomposes as
\begin{align}
 -2\widetilde l_n(\theta)&=F_n(t)+G_n(s,l,t),\notag\\
 F_n(t)&=d_n\log t+\norm{W_n}^2/t,\notag\\
 G_n(s,l,t)&=\log|A_\theta|+Z_n^\top A_\theta^{-1}Z_n.\notag
\end{align}

The spatial parameters are identified through the Taylor coefficients of
the field at a fixed point. For the next lemma, apply a fixed rigid motion to the spatial coordinates
so that $Q=[0,L]^p$ for some $L>0$. This preserves the covariance model.
The field $F$ has a version that is an entire function on $\R^p$
(see the proof of \Cref{lem:recovery}). Its normalized Taylor
coefficients at the origin are
\begin{equation}\label{eq:coefficient-definition}
 \zeta_\alpha=\sqrt{\alpha!}\,[x^\alpha]F(x),\qquad\alpha\in\N_0^p.
\end{equation}
Let $\zeta^{(m)}=(\zeta_\alpha:|\alpha|\le m)$,
$q_m=\binom{m+p}{p}$, and
$\Sigma_{m,\theta}=\Cov_\theta(\zeta^{(m)})$.

These coefficients have a triangular structure. By
\eqref{eq:triangular-transform} in the proof of \Cref{lem:recovery},
$\zeta^{(m)}=\sqrt s\,\mathcal T_{m,\theta}\xi^{(m)}$ with independent
standard normal $\xi_\alpha$ and a matrix $\mathcal T_{m,\theta}$ that is
lower triangular in degree order with diagonal entries $l^{-|\alpha|}$.
Given the coefficients of lower degree, $\zeta_\alpha$ therefore has
variance $s\,l^{-2|\alpha|}$. The next lemma shows that a fixed linear
statistic recovers these coefficients through degree $m_n\asymp b_n$,
with polynomially small error uniformly in $\theta$. The regression is
fitted to a higher degree $K_n\approx2m_n$: extracting a coefficient of
degree $m$ from the fitted polynomial amplifies errors by a factor
$e^{m\log m}$, which the truncation error $e^{-K\log K}$ of the fit must
overcome.

\begin{lemma}[Recovery of normalized Taylor coefficients]\label{lem:recovery}
Under \Cref{ass:parameter,ass:coverage}, set
\begin{equation*}
 \beta=\kappa p,\qquad
 K_n=\left\lfloor\frac{\beta}{2}b_n\right\rfloor,\qquad
 m_n=\left\lfloor\frac{\beta}{4}b_n\right\rfloor.
\end{equation*}
There is a parameter-independent linear statistic
$\widehat\zeta^{(m_n)}=H_nY_n$ such that
\begin{equation}\label{eq:recovery-mse}
 \sup_{\theta\in\Theta}\max_{|\alpha|\le m_n}
 \E_\theta(\widehat\zeta_\alpha-\zeta_\alpha)^2
 \le n^{-\beta/4+o(1)}.
\end{equation}
The exact coefficient covariance satisfies
\begin{equation}\label{eq:coefficient-eigenvalues}
 e^{-Cm}\Id_{q_m}\preceq\Sigma_{m,\theta}
 \preceq e^{Cm}\Id_{q_m},\qquad m\ge1,
\end{equation}
uniformly in $\theta$. Writing
$\widehat\Sigma_{n,\theta}=\Cov_\theta(H_nY_n)$, we have
\begin{equation}\label{eq:relative-recovery}
 \sup_{\theta\in\Theta}
 \norm{\Sigma_{m_n,\theta}^{-1/2}
 (\widehat\Sigma_{n,\theta}-\Sigma_{m_n,\theta})
 \Sigma_{m_n,\theta}^{-1/2}}_{\op}
 \le n^{-\beta/8+o(1)}.
\end{equation}
Consequently,
\begin{equation}\label{eq:recovery-affinity}
 \sup_{\theta,\theta'\in\Theta}
 \left|
 \log\Aff(\widehat\Sigma_{n,\theta},\widehat\Sigma_{n,\theta'})
 -\log\Aff(\Sigma_{m_n,\theta},\Sigma_{m_n,\theta'})
 \right|=o(1).
\end{equation}
\end{lemma}

The logarithm of the conditional variance $s\,l^{-2|\alpha|}$ is
$\log s-2|\alpha|\log l$, which is linear in the degree. Separating
$(s,l)$ from $(s_0,l_0)$ therefore amounts to a regression of
log-variances on degrees spread over $[0,m_n]$. With about $b_n^p$
coefficients, this explains the information scales $b_n^p$ for the
intercept $s$ and $b_n^{p+2}$ for the slope $l$. The next lemma turns this
heuristic into a global bound on the affinity between observation laws,
and adds the noise block to separate the nugget.

\begin{lemma}[Affinity separation]\label{lem:affinity}
Under \Cref{ass:parameter,ass:coverage}, there are $c>0$ and a
deterministic sequence $e_n\to0$ such that, for all $\theta=(s,l,t)\in\Theta$,
\begin{equation}\label{eq:affinity-separation}
 \Aff(P_{\theta,n},P_{\theta_0,n})
 \le \exp\left[
 e_n-cb_n^p\min\left\{\Lambda_n(\theta)^2,\Lambda_n(\theta)\right\}
 \right],
\end{equation}
where $\Lambda_n(\theta)\ge0$ is defined by
$\Lambda_n(\theta)^2=\log^2(s/s_0)+b_n^2\log^2(l/l_0)$. Moreover, write $h=D_n(\theta-\theta_0)=(h_\spat,h_t)$, where
$h_\spat\in\R^2$ collects the two spatial coordinates. After decreasing
$c$ and enlarging $e_n$, for all $\theta\in\Theta$,
\begin{equation}\label{eq:affinity-joint}
 \Aff(P_{\theta,n},P_{\theta_0,n})
 \le\exp\bigl[e_n-c\{\min(\norm{h_\spat}^2,b_n^{p/2}\norm{h_\spat})
 +h_t^2\}\bigr].
\end{equation}
The constant $c$ and the sequence $e_n$ do not depend on
$\theta_0\in\Theta$, and they depend on the design only through
$\calX$, $Q$, $C_h$, and $\kappa$. Both bounds hold at every
sufficiently large $n$ at which \eqref{eq:coverage} is satisfied.
\end{lemma}

To control the global maximizer, the likelihood ratio $Z_n(h)$ enters
only through the process $\xi_n=Z_n^{1/(2m)}$. The exponent $1/(2m)$ makes
$\E\xi_n(h)^m$ exactly the affinity, and a change of measure bounds the
$L^m$ Lipschitz constant of $\xi_n$ uniformly over the whole parameter
box. This holds even where covariances at different parameters are not
comparable.

\begin{lemma}[Likelihood-ratio process]\label{lem:lr-process}
Assume \Cref{ass:parameter} and bounded observation locations, and fix
$m\ge1$. On the box $U_n=D_n(\Theta-\theta_0)$ define
\[
 Z_n(h)=\exp\{l_n(\theta_0+D_n^{-1}h)-l_n(\theta_0)\},\qquad
 \xi_n(h)=Z_n(h)^{1/(2m)}.
\]
Then $\xi_n$ has continuous paths,
\begin{equation}\label{eq:lr-moment}
 \E_{\theta_0,n}\xi_n(h)^m
 =\Aff(P_{\theta_0+D_n^{-1}h,n},P_{\theta_0,n}),
 \qquad h\in U_n,
\end{equation}
and there is a constant $C_L$, depending only on $m$, $p$, $\Theta$, and
$\calX$, such that for all $n$ and all $u,v\in U_n$,
\begin{equation}\label{eq:lr-lipschitz}
 \norm{\xi_n(u)-\xi_n(v)}_{L^m(\PP_{\theta_0,n})}\le C_L\norm{u-v}.
\end{equation}
\end{lemma}

Combining \Cref{lem:affinity,lem:lr-process} with \Cref{cor:envelope}
gives the tail bound. The key observation is that
$\xi_n(\widehat h_n)\ge\xi_n(0)=1$, because $\thetan$ maximizes the
likelihood. Hence a large deviation of $\widehat h_n$ forces a large
supremum of $\xi_n$ far from the origin.

\begin{lemma}[Exponential tails of the global maximizer]\label{lem:tail}
Under \Cref{ass:parameter,ass:coverage}, write
$\widehat h_n=D_n(\thetan-\theta_0)
=(\widehat h_{n,\spat},\widehat h_{n,t})$. There are constants $c,C>0$
and $n_0$ such that \eqref{eq:tail-bound} holds for all $n\ge n_0$ and
$H\ge0$. The constants do not depend on $\theta_0\in\Theta$, and they
depend on the design only through $\calX$, $Q$, $C_h$, and $\kappa$. For
$n\ge n_0$, the bounds hold at every $n$ at which \eqref{eq:coverage} is
satisfied. Interiority of $\theta_0$ is not used.
\end{lemma}

The upper information bounds follow from the derivative bounds of
\Cref{lem:projection}. The lower bound on $J_{n,\spat}$ avoids
differentiating the coefficient recovery error. Instead, it compares a
local quadratic expansion of the affinity at distance one, in rescaled
coordinates, with the separation bound of \Cref{lem:affinity}.

\begin{lemma}[Information bounds]\label{lem:information}
Under \Cref{ass:parameter} and bounded observation locations,
\begin{equation*}
 \sup_{\theta\in\Theta}\calI_{ss,n}(\theta)\le Cb_n^p,\qquad
 \sup_{\theta\in\Theta}\calI_{ll,n}(\theta)\le Cb_n^{p+2},\qquad
 \sup_{\theta\in\Theta}\calI_{tt,n}(\theta)\le Cn.
\end{equation*}
If \Cref{ass:coverage} also holds, then
\eqref{eq:information-bounds} and \eqref{eq:information-block} hold.
\end{lemma}

The remaining stage is local. Because \Cref{thm:rbf}(iv) already gives
tightness, the main proof needs the quadratic expansion only on balls of
fixed radius. The next lemma gives it on balls of every radius
$o(b_n^{p/6})$, and \Cref{rem:radius} shows that this radius cannot be
enlarged.

\begin{lemma}[Uniform local quadratic expansion]\label{lem:quadratic}
Under \Cref{ass:parameter} and bounded observation locations, set
$\Delta_n=D_n^{-1}\nabla l_n(\theta_0)$ and
$L_n(h)=l_n(\theta_0+D_n^{-1}h)$.
Then $\E\Delta_n=0$, $\Cov(\Delta_n)=J_n$, and
\begin{equation}\label{eq:hessian-local}
 \nabla_h^2L_n(0)=-J_n+O_{\PP}(\epsilon_n).
\end{equation}
Let $H_n\ge1$ be any sequence with $\epsilon_nH_n\to0$. Then
\begin{equation}\label{eq:third-local}
 \sup_{\norm{h}\le H_n}\max_{a,b,c}
 |\partial_{h_ah_bh_c}^3L_n(h)|=O_{\PP}(\epsilon_n).
\end{equation}
For all sufficiently large $n$, the ball $\{\norm h\le H_n\}$ lies in the
rescaled parameter space $D_n(\Theta-\theta_0)$, and
\begin{equation}\label{eq:quadratic-remainder}
 \sup_{\norm{h}\le H_n}
 \left|L_n(h)-L_n(0)-h^\top\Delta_n+\frac12h^\top J_nh\right|
 =O_{\PP}(\epsilon_nH_n^3).
\end{equation}
In particular, if $\epsilon_nH_n^3\to0$, that is, $H_n=o(b_n^{p/6})$, then
\begin{equation}\label{eq:uniform-quadratic}
 \sup_{\norm{h}\le H_n}
 \left|L_n(h)-L_n(0)-h^\top\Delta_n+\frac12h^\top J_nh\right|
 =o_{\PP}(1).
\end{equation}
This applies, in particular, to every fixed radius.
\end{lemma}

\begin{remark}[The radius $b_n^{p/6}$ is sharp]\label{rem:radius}
The condition $\epsilon_nH_n^3\to0$ cannot be weakened. Under
\Cref{ass:parameter,ass:coverage}, there is $c>0$ such that, for every
fixed $\kappa>0$, the point $h_n=\kappa\epsilon_n^{-1/3}e_s$ in the
variance direction satisfies
\[
 \PP\Bigl\{L_n(h_n)-L_n(0)-h_n^\top\Delta_n+\frac12h_n^\top J_nh_n
 \ge c\kappa^3\Bigr\}\to1.
\]
Hence \eqref{eq:uniform-quadratic} fails for every sequence $H_n$ with
$\liminf_n\epsilon_nH_n^3>0$. The reason is that the third derivative of
the log likelihood in the variance direction is of exact order
$\epsilon_n$, as for the variance of $N\asymp b_n^p$ independent normal
observations.
\end{remark}

The score is a vector of centered Gaussian quadratic forms. Each
normalized whitened derivative has operator norm $O(\epsilon_n)$ but
Frobenius norm of order one, so the quadratic forms are sums of many
small independent contributions, and a Lyapunov argument gives the
normal limit.

\begin{lemma}[Fisher-normalized score central limit theorem]\label{lem:score}
Assume \Cref{ass:parameter,ass:coverage}, and let $M_n$ be $3\times3$
matrices with $M_n^\top M_n=\calI_n(\theta_0)$. Then
\begin{equation}\label{eq:score-clt}
 M_n^{-\T}\nabla l_n(\theta_0)
 \Longrightarrow N_3(0,\Id_3).
\end{equation}
The same quadratic-form argument applies to bounded sequences of linear
combinations of $\Delta_n$ whose variances converge to a positive limit.
\end{lemma}

The last lemma is used only for \Cref{cor:random}. Independent sampling
with a density bounded below on $Q$ satisfies \Cref{ass:coverage} almost
surely, with polynomially small failure probability at each $n$.

\begin{lemma}[Coverage for independent sampling]\label{lem:random}
Under the sampling assumptions of \Cref{cor:random}, almost surely,
\begin{equation}\label{eq:random-fill}
 h_n(Q)=O((\log n/n)^{1/p}).
\end{equation}
More precisely, for every $q>0$ there are constants $A_q,C_q<\infty$ such
that $\PP\{h_n(Q)>A_q(\log n/n)^{1/p}\}\le C_qn^{-q}$ for all $n\ge2$.
Thus \Cref{ass:coverage} holds almost surely for every fixed
$0<\kappa<1/p$.
\end{lemma}

\section{Proofs of the main results}\label{app:proof-main}

\subsection{\texorpdfstring{Proof of \Cref{thm:rbf}}{Proof of the main theorem}}\label{app:proof-thm}

\begin{proof}
\Cref{lem:information} proves \textbf{part (i)}, and \Cref{lem:tail} proves
\textbf{part (iv)}. Write $\widehat h_n=D_n(\thetan-\theta_0)$, recall
$L_n(h)=l_n(\theta_0+D_n^{-1}h)$ from \Cref{lem:quadratic}, and set
\[
 Q_n(h)=h^\top\Delta_n-\frac12h^\top J_nh,\qquad
 h_n^*=J_n^{-1}\Delta_n,\qquad
 R_{n,K}=\sup_{\norm h\le K}\bigl|L_n(h)-L_n(0)-Q_n(h)\bigr|.
\]

\textbf{Part (ii).}
For every fixed $K$, \eqref{eq:uniform-quadratic} applies with the fixed
radius $H_n=\max(K,1)$, so
$R_{n,K}=o_{\PP}(1)$; that is, $L_n(h)-L_n(0)=Q_n(h)+o_{\PP}(1)$
uniformly on $\{\norm h\le K\}$. The function $Q_n$ is quadratic, with
$\nabla Q_n(h_n^*)=\Delta_n-J_nh_n^*=0$ and $\nabla^2Q_n=-J_n$, so its
Taylor expansion about $h_n^*$ is exact:
\begin{equation}\label{eq:Q-taylor}
 Q_n(h)=Q_n(h_n^*)-\frac12(h-h_n^*)^\top J_n(h-h_n^*),\qquad h\in\R^3.
\end{equation}
In particular, $Q_n(h)\le Q_n(h_n^*)$ for all $h$.

Fix $\varepsilon>0$. The score satisfies $\E\norm{\Delta_n}^2=\tr J_n\le C$,
so $h_n^*=O_{\PP}(1)$ by part (i), and $\widehat h_n=O_{\PP}(1)$ by part
(iv). Choose $K$ such that the event
$E_n=\{\norm{\widehat h_n}\le K,\ \norm{h_n^*}\le K\}$ has
$\PP(E_n^c)\le\varepsilon$ for all large $n$. Because $\theta_0$ is
interior, the ball $\{\norm h\le K\}$ lies in $D_n(\Theta-\theta_0)$ for
large $n$. Hence, on $E_n$, the point $\theta_0+D_n^{-1}h_n^*$ belongs
to $\Theta$, and global optimality gives
$L_n(\widehat h_n)\ge L_n(h_n^*)$. Using the expansion at $h_n^*$, this
inequality, the expansion at $\widehat h_n$, and $Q_n\le Q_n(h_n^*)$, we
obtain on $E_n$
\[
 Q_n(h_n^*)-R_{n,K}\le L_n(h_n^*)-L_n(0)\le L_n(\widehat h_n)-L_n(0)
 \le Q_n(\widehat h_n)+R_{n,K}\le Q_n(h_n^*)+R_{n,K}.
\]
Thus the maximized log likelihood ratio is approximated by the maximum
of the quadratic:
\begin{equation}\label{eq:max-value}
 \bigl|L_n(\widehat h_n)-L_n(0)-Q_n(h_n^*)\bigr|\le R_{n,K}
 \qquad\text{on }E_n.
\end{equation}
By \eqref{eq:Q-taylor} at $h=\widehat h_n$, \eqref{eq:max-value}, and the
expansion at $\widehat h_n$, on $E_n$
\[
\begin{aligned}
& \frac12(\widehat h_n-h_n^*)^\top J_n(\widehat h_n-h_n^*)
 = Q_n(h_n^*)-Q_n(\widehat h_n)\\
&\qquad \qquad \le\bigl|L_n(\widehat h_n)-L_n(0)-Q_n(h_n^*)\bigr|
 +\bigl|L_n(\widehat h_n)-L_n(0)-Q_n(\widehat h_n)\bigr|
 \le2R_{n,K}.
\end{aligned}
\]
Since $J_n\succeq c\Id_3$ by part (i), for every $\eta>0$,
\[
 \PP\{\norm{\widehat h_n-h_n^*}>\eta\}
 \le\PP(E_n^c)+\PP\{R_{n,K}\ge c\eta^2/4\}
 \le\varepsilon+o(1).
\]
As $\varepsilon$ is arbitrary, $\widehat h_n-h_n^*=o_{\PP}(1)$, which is
\eqref{eq:mle-expansion}. The three rates and consistency follow.

\textbf{Part (iii).}
Let $T_n=D_nJ_n^{1/2}$, so that $\calI_n(\theta_0)=D_nJ_nD_n=T_nT_n^\top$.
Then
\[
 T_n^\top(\thetan-\theta_0)=J_n^{1/2}\widehat h_n,\qquad
 T_n^{-1}\nabla l_n(\theta_0)=J_n^{-1/2}D_n^{-1}\nabla l_n(\theta_0)
 =J_n^{-1/2}\Delta_n=J_n^{1/2}h_n^*,
\]
and therefore
\begin{equation}\label{eq:factor-remainder}
 T_n^\top(\thetan-\theta_0)-T_n^{-1}\nabla l_n(\theta_0)
 =J_n^{1/2}(\widehat h_n-h_n^*),
\end{equation}
whose norm is at most $\norm{J_n}_{\op}^{1/2}\norm{\widehat h_n-h_n^*}
=o_{\PP}(1)$ by parts (i) and (ii). Now let $M_n^\top M_n=\calI_n(\theta_0)$.
The matrix $O_n=M_nT_n^{-\T}$ is orthogonal, because
$O_nO_n^\top=M_n\calI_n(\theta_0)^{-1}M_n^\top=\Id_3$. Hence $M_n=O_nT_n^\T$ and
$M_n^{-\top}=O_nT_n^{-1}$, and \eqref{eq:factor-remainder} gives
\[
 M_n(\thetan-\theta_0)
 =M_n^{-\top}\nabla l_n(\theta_0)+O_nJ_n^{1/2}(\widehat h_n-h_n^*)
 =M_n^{-\top}\nabla l_n(\theta_0)+o_{\PP}(1).
\]
\Cref{lem:score} and Slutsky's lemma prove \eqref{eq:main-an}.

For the nugget, write $\Delta_{n,t}$ for the third score coordinate.
The block structure in \eqref{eq:information-block} and
\eqref{eq:mle-expansion} give
\[
 \sqrt n(\widehat t_n-t_0)=2t_0^2\Delta_{n,t}+o_{\PP}(1).
\]
Its score variance tends to $(2t_0^2)^{-1}$, and the final assertion of
\Cref{lem:score} gives
$\Delta_{n,t}\Rightarrow N(0,(2t_0^2)^{-1})$.
This proves \eqref{eq:nugget-an}.

\textbf{Part (v).}
Let $\mu>0$. For $n\ge n_0$, part (iv) and
$e^{-c\min(x,y)}\le e^{-cx}+e^{-cy}$ give
\[
 \E e^{\mu\norm{\widehat h_{n,\spat}}}
 =1+\int_0^\infty\mu e^{\mu H}\,
 \PP\{\norm{\widehat h_{n,\spat}}>H\}\,dH
 \le1+C\int_0^\infty\mu e^{\mu H}
 \bigl(e^{-cH^2}+e^{-cb_n^{p/2}H}\bigr)\,dH,
\]
which is bounded uniformly in $n$ once $cb_n^{p/2}\ge2\mu$. Similarly,
$\E e^{\mu|\widehat h_{n,t}|}\le1+C\int_0^\infty\mu e^{\mu H-cH^2}\,dH$.
Since $\norm{\widehat h_n}\le\norm{\widehat h_{n,\spat}}+|\widehat h_{n,t}|$,
the Cauchy--Schwarz inequality with $\mu=2\lambda$ proves
\eqref{eq:exp-moments}.

Let $f$ be as in part (v). By part (iii) with $M_n=T_n^\top$,
$Y_n=J_n^{1/2}\widehat h_n\Rightarrow Z$. The family
$\{f(\widehat h_n)\}$ is uniformly integrable, because
$f(\widehat h_n)^2\le C^2e^{2\lambda\norm{\widehat h_n}}$ has bounded
expectation by \eqref{eq:exp-moments}. Suppose that the first
convergence in \eqref{eq:f-convergence} fails. Then there are $\delta>0$
and a subsequence along which
$|\E f(\widehat h_n)-\E f(J_n^{-1/2}Z)|\ge\delta$. By part (i), a further
subsequence has $J_n\to J_*$ with $c\Id_3\preceq J_*\preceq C\Id_3$.
Along it, $\widehat h_n=J_n^{-1/2}Y_n\Rightarrow J_*^{-1/2}Z$ by Slutsky's
lemma, so uniform integrability gives
$\E f(\widehat h_n)\to\E f(J_*^{-1/2}Z)$. Also
$\E f(J_n^{-1/2}Z)\to\E f(J_*^{-1/2}Z)$ by dominated convergence, since
$|f(J_n^{-1/2}Z)|\le Ce^{\lambda c^{-1/2}\norm Z}$. This contradicts the
choice of $\delta$. For the second convergence, the proof of part (iii)
gives $M_n(\thetan-\theta_0)=O_nY_n$ with $O_n$ orthogonal, so
$\norm{M_n(\thetan-\theta_0)}=\norm{Y_n}\le C^{1/2}\norm{\widehat h_n}$.
By \eqref{eq:exp-moments}, $f(M_n(\thetan-\theta_0))$ is uniformly
integrable, and \eqref{eq:main-an} gives
$\E f(M_n(\thetan-\theta_0))\to\E f(Z)$.

Finally, polynomials satisfy the growth condition. Applying the first
convergence in \eqref{eq:f-convergence} to the coordinates $x_a$ and to
the products $x_ax_b$, and using $\E J_n^{-1/2}Z=0$ and
$\E J_n^{-1/2}ZZ^\top J_n^{-1/2}=J_n^{-1}$, proves
\eqref{eq:moment-convergence}.
\end{proof}

\subsection{Proof of \Cref{rem:uniform}}\label{app:uniform}

We show why the uniformity statements in \Cref{rem:uniform} hold. The arguments
are those of the preceding proofs, and we only point out where uniformity
enters.

\textbf{Part~(i).} The upper bounds in \Cref{lem:information} are suprema
over $\Theta$. The lower bound compares a second-order expansion of the
affinity on the unit sphere $\{\norm h=1\}$ with
\eqref{eq:affinity-joint}. The remainders in that expansion are
controlled by the derivative bounds of \Cref{lem:projection}, which hold
uniformly over $\Theta$, and the constants in \eqref{eq:affinity-joint}
do not depend on $\theta_0$. For $\theta_0$ near the boundary, the points
$\theta_0+D_n^{-1}h$ may leave $\Theta$. Applying the same lemmas on a
slightly larger rectangle $\Theta'\supset\Theta$, whose constants are
again uniform, removes this difficulty. Finally,
$J_{n,tt}\to(2t_0^2)^{-1}$ and the vanishing of the cross-information
hold uniformly in $t_0\in[t_-,t_+]$.

\textbf{Parts~(ii), (iii), and (v).} Let $K\subset\intt(\Theta)$ be
compact. Since $\norm{D_n^{-1}}_{\op}\to0$, for every fixed radius $H$
the ball $\{\norm h\le H\}$ lies in $D_n(\Theta-\theta_0)$ for all
$\theta_0\in K$ once $n$ is large. The remainders in \Cref{lem:quadratic}
are controlled by bounds whose laws do not depend on $\theta_0$. The
Hessian fluctuation is a centered Gaussian quadratic form with Frobenius
norm $O(\epsilon_n)$, and the third-derivative bound uses
$Z_n^\top A_{\theta_0}^{-1}Z_n$, which is exactly $\chi^2_{r_n}$ under
$\theta_0$. In \Cref{lem:score}, the Lyapunov ratio is at most
$C\epsilon_n^2$ uniformly in $\theta_0$. Along any sequence
$\theta_{0,n}\in K$, the Lindeberg--Feller theorem for triangular arrays
therefore gives the normal limit of each fixed linear combination of
the normalized score, and the Cram\'er--Wold device gives the joint
limit. The tail bounds of \Cref{lem:tail} are also uniform. Running the
argument of \Cref{app:proof-thm} along an arbitrary sequence
$\theta_{0,n}\in K$ therefore yields the conclusions of parts~(ii),
(iii), and~(v) along that sequence, which is equivalent to uniformity
over $K$.

Near the boundary the argument breaks down. It uses global optimality in
the form $l_n(\thetan) \ge l_n(\theta_0+D_n^{-1}h_n^*)$, which requires
the unconstrained maximizer $h_n^*$ of the quadratic approximation to be
feasible, that is, $\theta_0+D_n^{-1}h_n^*\in\Theta$. If $\theta_0$ lies
within distance of order $d_{n,j}^{-1}$ of the boundary in some coordinate
$j$, this fails with probability bounded away from zero. The constrained
maximizer $\widehat h_n$ then approximates the maximizer of the quadratic
over the feasible set, whose limit law is not normal.

\subsection{\texorpdfstring{Proof of \Cref{thm:minimax}}{Proof of the lower bounds}}\label{app:proof-lower}

\begin{proof}
Consider a path varying only coordinate $j$. Differentiating its Gaussian
density $p_u$ under the integral and using Cauchy--Schwarz gives
\[
 \norm{\partial_up_u}_{L^1}
 =\E_u|\partial_u\log p_u|
 \le\sqrt{\calI_{jj,n}(u)}.
\]
Consequently,
\begin{equation}\label{eq:tv-information}
 \TV(P_{u_1,n},P_{u_0,n})
 \le\frac12\int_{u_0}^{u_1}
       \sqrt{\calI_{jj,n}(u)}\,du.
\end{equation}
The uniform upper bounds in \Cref{lem:information} are
$\calI_{jj,n}\le C d_{n,j}^2$. Choose $a>0$ small enough that the
right side of \eqref{eq:tv-information} is at most $1/4$ when
$u_1-u_0=a/d_{n,j}$. Both alternatives lie in $\Theta$ eventually.

We now reduce estimation to testing between the two alternatives
(Le Cam's two-point method). Write $u_0=\theta_{0,j}$,
$u_1=u_0+a/d_{n,j}$, $P_0=P_{u_0,n}$, and $P_1=P_{u_1,n}$, and let
$\delta=(u_1-u_0)/2=a/(2d_{n,j})$ be half the separation. Let
$\widetilde\theta_{n,j}$ be any estimator of the $j$th coordinate, that is,
any measurable function of the data vector $Y_n$. It defines a test of
$P_0$ against $P_1$: decide for whichever of the two candidate values
$u_0,u_1$ is closer to the estimate,
\[
 \psi=\mathbf 1\bigl\{|\widetilde\theta_{n,j}-u_1|<|\widetilde\theta_{n,j}-u_0|\bigr\},
\]
so that $\psi=1$ means ``decide $u_1$'' and $\psi=0$ means ``decide
$u_0$''. Equivalently, $\psi=1$ exactly when $\widetilde\theta_{n,j}$
exceeds the midpoint $(u_0+u_1)/2$.

A wrong decision forces a large estimation error. If $\psi=1$, then
$|\widetilde\theta_{n,j}-u_0|>|\widetilde\theta_{n,j}-u_1|$, and the
triangle inequality
$|\widetilde\theta_{n,j}-u_0|+|\widetilde\theta_{n,j}-u_1|\ge u_1-u_0$
gives $|\widetilde\theta_{n,j}-u_0|>\delta$. Likewise, if $\psi=0$, then
$|\widetilde\theta_{n,j}-u_1|\ge\delta$. Hence
\[
 P_0\{|\widetilde\theta_{n,j}-u_0|\ge\delta\}\ge P_0(\psi=1),
 \qquad
 P_1\{|\widetilde\theta_{n,j}-u_1|\ge\delta\}\ge P_1(\psi=0).
\]

The total variation distance bounds how well any test can separate the
two laws. For the event $E=\{\psi=1\}$,
\[
 P_0(\psi=1)+P_1(\psi=0)=1-\{P_1(E)-P_0(E)\}
 \ge1-\TV(P_0,P_1)\ge\frac34.
\]
Combining the last two displays, the two estimation error
probabilities sum to at least $3/4$, so the larger of them is at least
$3/8$. This proves \eqref{eq:probability-lower} with $c_0=3/8$, since the
estimator was arbitrary.

Finally, by Markov's inequality,
$\E_{\theta,n}(\widetilde\theta_{n,j}-\theta_j)^2
\ge\delta^2\,\PP_{\theta,n}\{|\widetilde\theta_{n,j}-\theta_j|\ge\delta\}$.
Taking the larger of the two alternatives and using
\eqref{eq:probability-lower} gives
$\sup_{\theta\in\Theta}\E_{\theta,n}(\widetilde\theta_{n,j}-\theta_j)^2
\ge\frac38\cdot\frac{a^2}{4d_{n,j}^2}$, which is
\eqref{eq:risk-lower}.

For the upper bound, recall from \Cref{lem:tail} that the constants in
the tail bound \eqref{eq:tail-bound} do not depend on $\theta_0\in\Theta$,
and that interiority of $\theta_0$ is not needed. Integrating the tail
bound as in the proof of \Cref{thm:rbf}(v) therefore gives, for
$n\ge n_0$,
\[
 \sup_{\theta\in\Theta}
 \E_{\theta,n}\bigl\{d_{n,j}^2(\widehat\theta_{n,j}-\theta_j)^2\bigr\}
 \le C,
\]
with $C$ independent of $n$ and $\theta$. Since $\Theta$ is bounded, the
finitely many $n<n_0$ are covered by enlarging $C$. This proves
\eqref{eq:minimax-upper}.
Combining this with \eqref{eq:risk-lower} gives
\eqref{eq:minimax-rate}. Since
\[
 (d_{n,s}^{-1},d_{n,l}^{-1},d_{n,t}^{-1})
 =
 (b_n^{-p/2},b_n^{-(p+2)/2},n^{-1/2}),
\]
the stated minimax root MSE rates follow.
\end{proof}

\subsection{Proofs of the Corollaries in \texorpdfstring{\Cref{sec:covariance}}{Section 3.4}}\label{app:proof-cor}

\subsubsection{\texorpdfstring{Proof of \Cref{cor:log-scale}}{Proof of the log-scale corollary}}\label{app:proof-log}

\begin{proof}
Let $\theta_-=\min(s_-,l_-,t_-)$ and $\theta_+=\max(s_+,l_+,t_+)$, and
write $\widehat h_n=D_n(\thetan-\theta_0)$.

\textbf{Information.} The diagonal entries of $\Lambda_0$ lie in
$[\theta_-,\theta_+]$. Hence
$c\theta_-^2\Id_3\preceq J_n^{\log}=\Lambda_0J_n\Lambda_0\preceq
C\theta_+^2\Id_3$, and $J_n^{\log}$ inherits the block structure
\eqref{eq:information-block}, with $J_{n,tt}^{\log}=t_0^2J_{n,tt}\to1/2$.
This is part (i) of \Cref{thm:rbf} on the log scale.

\textbf{Expansion and rates.} A second-order Taylor expansion of $\log$,
coordinatewise on $[\theta_-,\theta_+]$, gives
\begin{equation}\label{eq:log-taylor}
 D_n(\widehat\lambda_n-\lambda_0)=\Lambda_0^{-1}\widehat h_n+R_n,
 \qquad
 \norm{R_n}\le\frac{\norm{D_n^{-1}}_{\op}\norm{\widehat h_n}^2}{2\theta_-^2}.
\end{equation}
Since $\norm{D_n^{-1}}_{\op}=\epsilon_n$ for large $n$,
\eqref{eq:exp-moments} gives $R_n=O_\PP(\epsilon_n)$, together with all
moments. By \eqref{eq:mle-expansion},
\[
 \Lambda_0^{-1}\widehat h_n
 =\Lambda_0^{-1}J_n^{-1}D_n^{-1}\nabla l_n(\theta_0)+o_\PP(1)
 =(J_n^{\log})^{-1}D_n^{-1}\{\Lambda_0\nabla l_n(\theta_0)\}+o_\PP(1),
\]
which is part (ii) on the log scale. The rates follow.

\textbf{Normal limits.} Let $M_n^\top M_n=\calI_n^{\log}(\theta_0)$. Then
$\norm{M_nD_n^{-1}}_{\op}^2=\norm{J_n^{\log}}_{\op}\le C$, and
\[
 M_nD_n^{-1}(J_n^{\log})^{-1}D_n^{-1}=M_n\calI_n^{\log}(\theta_0)^{-1}
 =M_n^{-\top}.
\]
Hence, by the previous step,
\[
 M_n(\widehat\lambda_n-\lambda_0)
 =M_n^{-\top}\Lambda_0\nabla l_n(\theta_0)+o_\PP(1)
 =\widetilde M_n^{-\top}\nabla l_n(\theta_0)+o_\PP(1),
 \qquad\widetilde M_n=M_n\Lambda_0^{-1}.
\]
Since $\widetilde M_n^\top\widetilde M_n
=\Lambda_0^{-1}\calI_n^{\log}(\theta_0)\Lambda_0^{-1}=\calI_n(\theta_0)$,
\Cref{lem:score} gives $M_n(\widehat\lambda_n-\lambda_0)\Rightarrow
N_3(0,\Id_3)$. For the nugget, the third coordinate of
\eqref{eq:log-taylor} reads
$\sqrt n(\log\widehat t_n-\log t_0)=t_0^{-1}\sqrt n(\widehat t_n-t_0)+R_{n,t}$
with $R_{n,t}=O_\PP(n^{-1/2})$, and \eqref{eq:nugget-an} gives the
$N(0,2)$ limit.

\textbf{Tails and moments.} Since $|\log x-\log y|\le|x-y|/\theta_-$ on
$[\theta_-,\theta_+]$, each coordinate satisfies
$|d_{n,j}(\widehat\lambda_{n,j}-\lambda_{0,j})|\le|\widehat h_{n,j}|/\theta_-$.
Hence \eqref{eq:tail-bound} holds for $D_n(\widehat\lambda_n-\lambda_0)$
with $H$ replaced by $\theta_-H$, and \eqref{eq:exp-moments} holds with
the constant in the exponent divided by $\theta_-$. For the limits in
part (v), $R_n\to0$ in probability in \eqref{eq:log-taylor}. Along any
subsequence with $J_n\to J_*$, $\Lambda_0^{-1}\widehat h_n\Rightarrow
\Lambda_0^{-1}J_*^{-1/2}Z$, which has law
$N_3\{0,(\Lambda_0J_*\Lambda_0)^{-1}\}$. The subsequence argument in the
proof of \Cref{thm:rbf}(v) therefore applies with $J_n$ replaced by
$J_n^{\log}$. In particular,
$D_n\E[(\widehat\lambda_n-\lambda_0)(\widehat\lambda_n-\lambda_0)^\top]D_n
=(J_n^{\log})^{-1}+o(1)$. This gives
the stated MSE orders, with
$[(J_n^{\log})^{-1}]_{tt}=t_0^{-2}[J_n^{-1}]_{tt}\to2$.

\textbf{Uniformity.} The additional constants above depend only on
$\theta_-$ and $\theta_+$, so the uniformity statements of
\Cref{rem:uniform} carry over.
\end{proof}

\subsubsection{\texorpdfstring{Proof of \Cref{cor:plugin}}{Proof of the plug-in information corollary}}\label{app:proof-plugin}

\begin{proof}
Write $\widehat h_n=D_n(\thetan-\theta_0)$ and
$\widehat J_n=D_n^{-1}\widehat\calI_nD_n^{-1}$. By \Cref{thm:rbf}(iv),
$\widehat h_n=O_\PP(1)$.

\textbf{Part (i) for the plug-in Fisher information.} For $\theta\in\Theta$,
put $J_n(\theta)=D_n^{-1}\calI_n(\theta)D_n^{-1}$, so that
$J_n(\theta_0)=J_n$. With
\begin{align*}
 B_a(\theta) &= d_{n,a}^{-1}C_n(\theta)^{-1/2}C_{n,a}(\theta)C_n(\theta)^{-1/2}, \quad \text{and}\quad
 B_{ab}(\theta) = d_{n,a}^{-1}d_{n,b}^{-1}
 C_n(\theta)^{-1/2}C_{n,ab}(\theta)C_n(\theta)^{-1/2},
\end{align*}
we have $[J_n(\theta)]_{ab}=\frac12\tr\{B_a(\theta)B_b(\theta)\}$.
Differentiating $\calI_n(\theta)$, using
$\partial_c C_n^{-1}=-C_n^{-1}C_{n,c}C_n^{-1}$, and normalizing gives, in
the coordinates $h=D_n(\theta-\theta_0)$,
\[
 \partial_{h_c}[J_n(\theta)]_{ab}
 =\frac12\tr\bigl\{-B_cB_aB_b-B_aB_cB_b+B_{ac}B_b+B_aB_{bc}\bigr\},
\]
with all matrices evaluated at $\theta$. By \Cref{lem:projection},
uniformly over $\theta\in\Theta$, $\norm{B_a(\theta)}_{\op}\le C\epsilon_n$
and $\norm{B_a(\theta)}_\Frob\le C$ for $a\in\{s,l,t\}$, where for the
nugget $\norm{B_t}_{\op}=O(n^{-1/2})\le C\epsilon_n$. Moreover,
$\norm{B_{ac}(\theta)}_\Frob\le C\epsilon_n$, and $B_{ac}=0$ when $a$ or
$c$ is $t$. Since $|\tr(XYZ)|\le\norm X_{\op}\norm Y_\Frob\norm Z_\Frob$
and $|\tr(XY)|\le\norm X_\Frob\norm Y_\Frob$, every entry of
$\nabla_hJ_n(\theta)$ is bounded by $C\epsilon_n$, uniformly over
$\Theta$. As $\Theta$ is convex, the mean value theorem along the segment
from $\theta_0$ to $\theta$ gives
\begin{equation}\label{eq:J-lipschitz}
 \norm{J_n(\theta)-J_n}_{\op}\le C\epsilon_n\norm{D_n(\theta-\theta_0)},
 \qquad\theta\in\Theta.
\end{equation}
Taking $\theta=\thetan$ proves \eqref{eq:plugin-information} for
$\widehat\calI_n=\calI_n(\thetan)$. This step uses neither interiority
of $\theta_0$ nor any local expansion, and it holds for every
$\theta_0\in\Theta$.

\textbf{Part (i) for the observed information.} Let $L_n$ be as in
\Cref{lem:quadratic}. Then
\[
 D_n^{-1}\{-\nabla^2l_n(\thetan)\}D_n^{-1}=-\nabla_h^2L_n(\widehat h_n).
\]
Fix $K\ge1$. On the event $\{\norm{\widehat h_n}\le K\}$, the mean value
theorem and \eqref{eq:third-local} with $H_n=K$ give
\[
 \norm{\nabla_h^2L_n(\widehat h_n)-\nabla_h^2L_n(0)}_{\op}
 \le CK\sup_{\norm h\le K}\max_{a,b,c}|\partial^3_{h_ah_bh_c}L_n(h)|
 =O_\PP(\epsilon_n),
\]
and \eqref{eq:hessian-local} gives $-\nabla_h^2L_n(0)=J_n+O_\PP(\epsilon_n)$.
Since $\PP\{\norm{\widehat h_n}>K\}$ is arbitrarily small for large $K$,
\eqref{eq:plugin-information} follows for the observed information.

\textbf{Consequences of \eqref{eq:plugin-information}.} Since
$c\Id_3\preceq J_n\preceq C\Id_3$, with probability tending to one
$\frac c2\Id_3\preceq\widehat J_n\preceq2C\Id_3$, so that
$\widehat\calI_n=D_n\widehat J_nD_n$ is positive definite. On this event,
\[
 \norm{\widehat J_n^{-1}-J_n^{-1}}_{\op}
 \le\norm{\widehat J_n^{-1}}_{\op}\norm{\widehat J_n-J_n}_{\op}
 \norm{J_n^{-1}}_{\op}=O_\PP(\epsilon_n).
\] Since
$[\widehat\calI_n^{-1}]_{jj}=d_{n,j}^{-2}[\widehat J_n^{-1}]_{jj}$,
$[\calI_n(\theta_0)^{-1}]_{jj}=d_{n,j}^{-2}[J_n^{-1}]_{jj}$, and
$[J_n^{-1}]_{jj}\ge C^{-1}$, the ratio statement in (i) follows.

\textbf{Part (ii).} Let $\Delta_n=D_n^{-1}\nabla l_n(\theta_0)$. By
\eqref{eq:mle-expansion}, $\widehat h_{n,j}=e_j^\top J_n^{-1}\Delta_n
+o_\PP(1)$. The variable
$V_n=e_j^\top J_n^{-1}\Delta_n/[J_n^{-1}]_{jj}^{1/2}$ has variance one,
since $\Cov(\Delta_n)=J_n$, and it is a linear combination of $\Delta_n$
with coefficient vector $J_n^{-1}e_j/[J_n^{-1}]_{jj}^{1/2}$ of norm at
most $C^{1/2}/c$. By the final assertion of \Cref{lem:score},
$V_n\Rightarrow N(0,1)$. As $[J_n^{-1}]_{jj}\ge C^{-1}$, also
$\widehat h_{n,j}/[J_n^{-1}]_{jj}^{1/2}=V_n+o_\PP(1)$. Finally,
\[
 \frac{\widehat\theta_{n,j}-\theta_{0,j}}{[\widehat\calI_n^{-1}]_{jj}^{1/2}}
 =\frac{\widehat h_{n,j}}{[J_n^{-1}]_{jj}^{1/2}}
 \left(\frac{[J_n^{-1}]_{jj}}{[\widehat J_n^{-1}]_{jj}}\right)^{1/2},
\]
and the last factor is $1+O_\PP(\epsilon_n)$ by part (i). Slutsky's lemma
gives the normal limit. The interval contains $\theta_{0,j}$ exactly
when the absolute value of this ratio is at most $z_{1-\alpha/2}$, which
gives the coverage statement.

\textbf{Part (iii).} For a positive definite matrix $A$, write $\chol(A)$
for its lower-triangular Cholesky factor with positive diagonal. Because
$D_n$ is diagonal with positive entries, $D_n\chol(\widehat J_n)$ is
lower triangular with positive diagonal and
$\{D_n\chol(\widehat J_n)\}\{D_n\chol(\widehat J_n)\}^\top
=\widehat\calI_n$. By uniqueness of the Cholesky factor,
$\widehat L_n=D_n\chol(\widehat J_n)$. The map $\chol$ is smooth on
positive definite matrices, hence Lipschitz on the compact convex set
$\{A:\frac c2\Id_3\preceq A\preceq2C\Id_3\}$, so
$\chol(\widehat J_n)-\chol(J_n)=O_\PP(\epsilon_n)$. Therefore
\[
\begin{aligned}
 \widehat L_n^\top(\thetan-\theta_0)&=\chol(\widehat J_n)^\top\widehat h_n
 =\chol(J_n)^\top\widehat h_n+O_\PP(\epsilon_n)\\
 &=M_n(\thetan-\theta_0)+O_\PP(\epsilon_n),
 \qquad M_n=\chol(J_n)^\top D_n.
\end{aligned}
\]
Since $M_n^\top M_n=D_n\chol(J_n)\chol(J_n)^\top D_n=\calI_n(\theta_0)$,
\Cref{thm:rbf}(iii) gives the first limit. The second follows from
$(\thetan-\theta_0)^\top\widehat\calI_n(\thetan-\theta_0)
=\norm{\widehat L_n^\top(\thetan-\theta_0)}^2$ and the continuous mapping
theorem, and the ellipsoid statement is a restatement of it.

\textbf{Part (iv).} Since $\widehat\Lambda_n-\Lambda_0=\diag(\thetan-\theta_0)$,
$\norm{\widehat\Lambda_n-\Lambda_0}_{\op}
\le\norm{D_n^{-1}}_{\op}\norm{\widehat h_n}=O_\PP(\epsilon_n)$. Because
diagonal matrices commute,
\[
 D_n^{-1}\widehat\calI_n^{\log}D_n^{-1}
 =\widehat\Lambda_n\widehat J_n\widehat\Lambda_n
 =\Lambda_0J_n\Lambda_0+O_\PP(\epsilon_n)
 =J_n^{\log}+O_\PP(\epsilon_n),
\]
which is part (i) on the log scale. By \Cref{cor:log-scale},
$D_n(\widehat\lambda_n-\lambda_0)
=(J_n^{\log})^{-1}D_n^{-1}\{\Lambda_0\nabla l_n(\theta_0)\}+o_\PP(1)$, and
the normal limit of \Cref{thm:rbf}(iii) holds on the log scale. The
proofs of parts (ii) and (iii) therefore apply verbatim, with
$\widehat h_n$ replaced by $D_n(\widehat\lambda_n-\lambda_0)$, $J_n$ by
$J_n^{\log}$, $\Delta_n$ by $\Lambda_0\Delta_n$, and $\calI_n(\theta_0)$ by
$\calI_n^{\log}(\theta_0)$. The displayed interval is the image under
$\exp$ of the log-scale Wald interval for $\lambda_{0,j}$. Finally, the
chain rule gives
\[
 -\nabla_\lambda^2l_n
 =\Lambda\{-\nabla_\theta^2l_n\}\Lambda-\diag(\theta\circ\nabla_\theta l_n),
 \qquad\Lambda=\diag(\theta),
\]
where $\circ$ denotes the entrywise product. At an interior maximizer the
gradient vanishes. Since $\theta_0\in\intt(\Theta)$ and $\thetan$ is
consistent, $\thetan\in\intt(\Theta)$ with probability tending to one.
\end{proof}

\subsubsection{\texorpdfstring{Proof of \Cref{cor:random}}{Proof of the random-design corollary}}\label{app:proof-random}
 
\begin{proof}
\textbf{A coverage event with stretched-exponential failure probability.}
Let $Q$ have side length $L$. For $n\ge1$, let
$k_n=\lfloor n^{1/(4p)}\rfloor$, and partition $Q$ into $k_n^p$ closed
subcubes of side length $L/k_n$. Let $H_n$ be the event that every
subcube contains at least one of $X_1,\ldots,X_n$. As in the proof of
\Cref{lem:random}, each subcube has sampling probability at least
$f_-L^p/k_n^p$, so a union bound gives
\begin{equation}\label{eq:coverage-failure}
 \PP(H_n^c)\le k_n^p\exp\Bigl(-\frac{nf_-L^p}{k_n^p}\Bigr)
 \le n^{1/4}\exp\bigl(-f_-L^pn^{3/4}\bigr),
\end{equation}
because $k_n^p\le n^{1/4}$. On $H_n$, every point of $Q$ lies within the
diameter $\sqrt p\,L/k_n$ of an observation. Since
$\lfloor y\rfloor\ge y/2$ for $y\ge1$, this gives
$h_n(Q)\le2\sqrt p\,Ln^{-1/(4p)}$. Hence, on $H_n$, \eqref{eq:coverage}
holds at $n$ with $\kappa=1/(4p)$ and $C_h=2\sqrt p\,L$. Fix these
constants from now on. The constants in
\Cref{lem:projection,lem:recovery,lem:affinity,lem:tail} depend on the
design only through $\calX$, $Q$, $C_h$, and $\kappa$, and the bounds of
\Cref{lem:tail} hold at every $n\ge n_0$ at which \eqref{eq:coverage} is
satisfied. Therefore every bound of \Cref{thm:rbf}(iv) and (v) holds on
$H_n$ for $n\ge n_0$, with constants that do not depend on the design.
By \eqref{eq:coverage-failure}, $\sum_n\PP(H_n^c)<\infty$, so by the first
Borel--Cantelli lemma, almost surely $H_n$ occurs for all sufficiently
large $n$.
 
Two bounds hold for every design in the bounded set $\calX$. The upper
information bounds of \Cref{lem:information} require only bounded
locations, so $\norm{J_n}_{\op}\le C$. The rescaled parameter set
$D_n(\Theta-\theta_0)$ has diameter at most $C\sqrt n$, so
$\norm{\widehat h_n}\le C\sqrt n$.
 
\textbf{Part (a).} The design is independent of the process and the
errors, so conditionally on $\mathbf X$ the observations follow the model
\eqref{eqn:observation} with deterministic locations. For almost every
realization of $\mathbf X$, $H_n$ occurs for all large $n$, so
\Cref{ass:coverage} holds with the constants fixed above. All results
proved for deterministic designs therefore apply conditionally on
$\mathbf X$, with design-independent constants and design-dependent
thresholds.
 
\textbf{Part (b)(i).} This is part (a) applied to \Cref{thm:rbf}(i), since
$J_n$ is a function of $\mathbf X$ alone.
 
\textbf{Convergence statements in (b)(ii) and (b)(vi).} Let $V_n$ be any
of the statistics in these statements whose limit law does not depend on
the design: the standardized errors in \eqref{eq:main-an},
\eqref{eq:nugget-an}, and \Cref{cor:log-scale,cor:plugin}, and the
quadratic forms with $\chi^2_3$ limits. Let $V$ have the limit law. For
bounded continuous $\varphi$, part (a) gives
$\E[\varphi(V_n)\mid\mathbf X]\to\E\varphi(V)$ almost surely, and dominated
convergence gives $\E\varphi(V_n)\to\E\varphi(V)$. Coverage probabilities
are treated in the same way. For statements of the form $R_n=o_\PP(1)$,
such as \eqref{eq:mle-expansion}, part (a) gives
$\PP(\norm{R_n}>\eta\mid\mathbf X)\to0$ almost surely for every $\eta>0$,
and dominated convergence gives $\PP(\norm{R_n}>\eta)\to0$. For
statements of the form $R_n=O_\PP(r_n)$, such as \eqref{eq:rates} and
\eqref{eq:plugin-information}, let
$q_{n,M}(\mathbf X)=\PP(\norm{R_n}>Mr_n\mid\mathbf X)$. Part (a) gives
$\lim_{M\to\infty}\limsup_nq_{n,M}(\mathbf X)=0$ almost surely. Since
$0\le q_{n,M}\le1$, the reverse Fatou lemma gives
\[
 \lim_{M\to\infty}\limsup_n\E q_{n,M}(\mathbf X)
 \le\lim_{M\to\infty}\E\Bigl\{\limsup_nq_{n,M}(\mathbf X)\Bigr\}=0.
\]
 
\textbf{Part (b)(iii).} For $n\ge n_0$, by \eqref{eq:coverage-failure},
\begin{align*}
    \PP\{\norm{\widehat h_{n,\spat}}\ge H\}
 &\le \E\bigl[\PP\{\norm{\widehat h_{n,\spat}}\ge H\mid\mathbf X\}
 \mathbf 1_{H_n}\bigr]+\PP(H_n^c) 
 \le C\exp\{-c\min(H^2,b_n^{p/2}H)\}+Ce^{-cn^{3/4}},
\end{align*}
and the same argument applies to the nugget coordinate.
 
\textbf{Part (b)(iv).} For $n\ge n_\lambda$, the exponential moment bound
of \Cref{thm:rbf}(v) on $H_n$, together with $\norm{\widehat h_n}\le
C\sqrt n$, gives
\begin{align*}
 \E\exp\{\lambda\norm{\widehat h_n}\}
 &\le\E\bigl[\E\{\exp(\lambda\norm{\widehat h_n})\mid\mathbf X\}
 \mathbf 1_{H_n}\bigr]+e^{\lambda C\sqrt n}\,\PP(H_n^c)
 \le C_\lambda+e^{\lambda C\sqrt n}\,n^{1/4}e^{-f_-L^pn^{3/4}},
\end{align*}
which is bounded in $n$. Now let $g$ be continuous with
$|g(x)|\le Ce^{\lambda\norm x}$. Since
$\norm{M_n(\thetan-\theta_0)}=\norm{J_n^{1/2}\widehat h_n}
\le C\norm{\widehat h_n}\le C\sqrt n$,
\[
 \E\bigl[|g(M_n(\thetan-\theta_0))|\mathbf 1_{H_n^c}\bigr]
 \le Ce^{\lambda C\sqrt n}\,\PP(H_n^c)\to0.
\]
On $H_n$, \Cref{thm:rbf}(v) bounds
$\E[\exp\{\lambda C\norm{\widehat h_n}\}\mid\mathbf X]$ uniformly in the
design, so $\E[g(M_n(\thetan-\theta_0))\mid\mathbf X]\mathbf 1_{H_n}$ is uniformly bounded. This converges to $\E g(Z)$ almost surely, by part (a) and since $\mathbf 1_{H_n}\to 1$ almost surely. Dominated convergence over the design gives $\E g(M_n(\thetan-\theta_0))\to\E g(Z)$.
 
\textbf{Part (b)(v).} The proof of the lower bounds in
\Cref{thm:minimax} uses only bounded locations. For every realized design, it shows that any estimator satisfies
$\PP_{\theta_0}\{|\widetilde\theta_{n,j}-\theta_{0,j}|\ge\delta\mid\mathbf X\}+\PP_{\theta_{1,n}^{(j)}}\{|\widetilde\theta_{n,j}-\theta_{1,n,j}^{(j)}|\ge\delta
\mid\mathbf X\}\ge3/4$, with $\delta=a/(2d_{n,j})$. Taking expectations over the design, the two unconditional error probabilities also sum to at least $3/4$, so the larger is at least $3/8$. This gives \eqref{eq:probability-lower}, and \eqref{eq:risk-lower} follows as in the proof of \Cref{thm:minimax}. For the upper bound, the constants of
\Cref{lem:tail} also do not depend on $\theta_0\in\Theta$. Hence, for $n\ge n_0$ and every $\theta\in\Theta$,
\[
 \E_\theta\{d_{n,j}^2(\widehat\theta_{n,j}-\theta_j)^2\}
 \le\E_\theta\bigl[\E_\theta\{d_{n,j}^2(\widehat\theta_{n,j}-\theta_j)^2
 \mid\mathbf X\}\mathbf 1_{H_n}\bigr]
 +C\,d_{n,j}^2\,\PP(H_n^c)\le C,
\]
because $\Theta$ is bounded, $d_{n,j}^2\le n$, and
$n\,\PP(H_n^c)\to0$ by \eqref{eq:coverage-failure}. The finitely many
$n<n_0$ are covered by enlarging $C$. This proves
\eqref{eq:minimax-upper}.
 
\textbf{Part (b)(vi).} The convergence in distribution and coverage
statements were treated above. For the plug-in Fisher information,
\eqref{eq:J-lipschitz} gives
$\norm{D_n^{-1}\calI_n(\thetan)D_n^{-1}-J_n}_{\op}\le
C\epsilon_n\norm{\widehat h_n}$ for every bounded design, and
$\widehat h_n=O_\PP(1)$ unconditionally by part (b)(ii). For the observed
information, \eqref{eq:plugin-information} holds conditionally by part
(a), and the $O_\PP$ argument above gives the unconditional version.
\end{proof}

\section{Proofs of the lemmas in Appendix \ref{app:statements}}\label{app:lemma-proofs}

\subsection{\texorpdfstring{Proof of \Cref{lem:affinity-monotone}}{Proof of the affinity monotonicity lemma}}\label{app:affinity-monotone}

\begin{proof}
Let $\rho=(P+Q)/2$, a probability law dominating $P$ and $Q$, and let
$p=dP/d\rho$ and $q=dQ/d\rho$. The affinity does not depend on the
dominating measure, so $\Aff(P,Q)=\E_\rho\sqrt{pq}$. Put
$\rho_T=\rho\circ T^{-1}$, which dominates $P\circ T^{-1}$ and
$Q\circ T^{-1}$. For $F\in\mathcal F'$,
\[
 P\{T^{-1}(F)\}=\E_\rho\bigl[p\,\mathbf 1\{T\in F\}\bigr]
 =\E_\rho\bigl[\E_\rho(p\mid T)\,\mathbf 1\{T\in F\}\bigr],
\]
so the density of $P\circ T^{-1}$ with respect to $\rho_T$, evaluated at
$T$, is $\E_\rho(p\mid T)$, $\rho$-almost surely. The same holds for $Q$.
Hence
\[
 \Aff(P\circ T^{-1},Q\circ T^{-1})
 =\E_\rho\sqrt{\E_\rho(p\mid T)\,\E_\rho(q\mid T)}.
\]
The conditional Cauchy--Schwarz inequality gives
$\E_\rho(\sqrt{pq}\mid T)\le\sqrt{\E_\rho(p\mid T)\,\E_\rho(q\mid T)}$
almost surely. Taking expectations yields
$\Aff(P,Q)=\E_\rho\sqrt{pq}\le\Aff(P\circ T^{-1},Q\circ T^{-1})$.
The Gaussian statement is the case $T(y)=Hy$, since $HY$ is
$N_r(0,HAH^\top)$ or $N_r(0,HBH^\top)$, and these covariances are positive
definite because $H$ has full row rank.
\end{proof}

\subsection{\texorpdfstring{Proof of \Cref{lem:qf-moments}}{Proof of the quadratic-form moment lemma}}\label{app:qf-moments}

\begin{proof}
If $B=0$ there is nothing to prove. Diagonalizing $B$, the variable
$Q=X^\top BX-\tr B$ has the law of $\sum_j\lambda_j(\eta_j^2-1)$ for
independent standard normal $\eta_j$, where
$\sum_j\lambda_j^2=\norm{B}_\Frob^2$. For $|v|\le1/2$,
\[
 -\tfrac12\log(1-v)-\tfrac v2=\sum_{r\ge2}\frac{v^r}{2r}
 \le\frac{v^2}{4(1-|v|)}\le\frac{v^2}2 .
\]
Hence, for $|s|\le1/(4\norm B_\Frob)$, so that $|2s\lambda_j|\le1/2$ for
every $j$,
\[
 \log\E e^{sQ}
 =\sum_j\Bigl\{-\frac12\log(1-2s\lambda_j)-s\lambda_j\Bigr\}
 \le2s^2\norm B_\Frob^2\le\frac18.
\]
Applying this with $s=\pm1/(4\norm B_\Frob)$ gives
$\E\exp\{|Q|/(4\norm B_\Frob)\}\le2e^{1/8}$. Since $x^q\le C_qe^x$ for
$x\ge0$, it follows that $\E|Q|^q\le2e^{1/8}C_q(4\norm B_\Frob)^q$.
\end{proof}

\subsection{\texorpdfstring{Proof of \Cref{lem:fraction}}{Proof of the Paley-Zygmund-type bound}}\label{app:fraction}

\begin{proof}
Splitting the second moment according to whether $|X|<\lambda\sigma$,
\[
 \sigma^2\le\E X^2
 =\E\bigl[X^2\mathbf 1\{|X|<\lambda\sigma\}\bigr]
 +\E\bigl[X^2\mathbf 1\{|X|\ge\lambda\sigma\}\bigr]
 \le\lambda^2\sigma^2+M^2\,\PP(|X|\ge\lambda\sigma).
\]
Rearranging gives the first claim. For the second, $\PP(|f(\iota)|\ge\lambda\sigma)$
is the number of $\iota\in\mathcal A$ with $|f(\iota)|\ge\lambda\sigma$,
divided by $|\mathcal A|$.
\end{proof}

\subsection{\texorpdfstring{Proof of \Cref{lem:kolmogorov}}{Proof of the Kolmogorov criterion}}\label{app:kolmogorov}

\begin{proof}
Inequality (1.10) of \citet{talagrand2021} shows that there is a
countable dense set $G\subseteq[0,1]^k$ such that, for all
$j\in\N_0$,
\[
 \Bigl\|\sup_{x,y\in G,\,|x-y|\le2^{-j}}|\xi(x)-\xi(y)|\Bigr\|_{L^m}
 \le C_{k,m,\alpha}L2^{-j(\alpha-k)/m}.
\]
Letting $j\to\infty$ shows that $\xi|_G$ is almost surely uniformly
continuous. Its continuous extension to $[0,1]^k$ is a modification of
$\xi$, because $\xi$ is continuous in $L^m$. For $h\in(0,1)$, choose
$j\in\N_0$ with $2^{-j-1}\le h<2^{-j}$. If $|x-y|\le h$ and
$x_r\to x$, $y_r\to y$ along $G$, then $|x_r-y_r|<2^{-j}$ eventually, so
the supremum over $|x-y|\le h$ is at most the supremum over $G$ at scale
$2^{-j}$. Since $2^{-j}\le2h$, this proves the claim for $h<1$. The case
$h=1$ follows from $h=1/2$ by the triangle inequality through midpoints.
\end{proof}

\subsection{\texorpdfstring{Proof of \Cref{lem:convex-sup}}{Proof of the convex-set supremum lemma}}\label{app:convex-sup}

\begin{proof}
If $M=0$, then $\xi=0$ almost surely on a countable dense subset of
$\Gamma$, and hence everywhere by continuity. Assume $M>0$. Fix
$a\in\Gamma$, let $\pi$ be the metric projection of $\R^k$ onto $\Gamma$,
which is $1$-Lipschitz, and put
$\eta(x)=\xi(\pi(a-\delta\mathbf 1+2\delta x))$ for $x\in[0,1]^k$, where
$\mathbf 1=(1,\ldots,1)$. Since $\Gamma\subseteq a+[-\delta,\delta]^k$ and
$\pi$ is the identity on $\Gamma$,
$\sup_{[0,1]^k}|\eta|=\sup_\Gamma|\xi|$. Moreover, $\eta$ has continuous
paths, $\sup_x\norm{\eta(x)}_{L^m}\le M$, and
$\norm{\eta(x)-\eta(y)}_{L^m}\le2\delta L|x-y|$. By
\Cref{lem:kolmogorov} with $\alpha=m>k$, applied to $\eta$, whose
continuous modification is indistinguishable from $\eta$,
\[
 \Bigl\|\sup_{|x-y|\le h}|\eta(x)-\eta(y)|\Bigr\|_{L^m}
 \le C\delta Lh^{1-k/m},\qquad h\in(0,1].
\]

Let $h\in(0,1]$ and cover $[0,1]^k$ by at most $\lceil1/h\rceil^k\le(2/h)^k$
closed cubes $B$ of side $h$, choosing a point $x_B$ in each. Every
$x\in[0,1]^k$ lies within $\sqrt k\,h$ of some $x_B$, so
\[
 \sup_x|\eta(x)|\le\max_B|\eta(x_B)|
 +\sup_{|x-y|\le\sqrt kh}|\eta(x)-\eta(y)|.
\]
Since $\norm{\max_B|X_B|}_{L^m}\le(\sum_B\norm{X_B}_{L^m}^m)^{1/m}$, the
first term has $L^m$ norm at most $(2/h)^{k/m}M$. Dividing the segment
from $x$ to $y$ into $\lceil\sqrt k\rceil$ pieces of length at most $h$,
the second term has $L^m$ norm at most $C\delta Lh^{1-k/m}$. Thus
\[
 \norm{\sup_\Gamma|\xi|}_{L^m}\le C\{Mh^{-k/m}+\delta Lh^{1-k/m}\},
 \qquad0<h\le1.
\]
If $M\ge\delta L$, take $h=1$ to obtain the bound $2CM$. Otherwise take
$h=M/(\delta L)$ to obtain $2CM^{1-k/m}(\delta L)^{k/m}$. Raising to the
power $m$ proves the claim.
\end{proof}

\subsection{\texorpdfstring{Proof of \Cref{prop:cover}}{Proof of the cover maximal inequality}}\label{app:cover}

\begin{proof}
Each $\sup_{\Gamma_j}|\xi|$ is measurable, being a supremum over a
countable dense subset of $\Gamma_j$. Since
$\sup_{u\in\bigcup_j\Gamma_j}|\xi(u)|^m\le\sum_j\sup_{u\in\Gamma_j}|\xi(u)|^m$,
monotone convergence and \Cref{lem:convex-sup} prove the claim.
\end{proof}

\subsection{\texorpdfstring{Proof of \Cref{cor:envelope}}{Proof of the envelope corollary}}\label{app:envelope}

\begin{proof}
We use two elementary properties of $\psi=\psi_\gamma$. First, $\psi$ is
nondecreasing, and $\psi(r/2)\ge\psi(r)/4$. Hence, for $c\ge0$,
\begin{equation}\label{eq:psi-shift}
 \psi\bigl((r-c)_+\bigr)\ge\psi(r)/4-c^2,\qquad r\ge0:
\end{equation}
if $r\ge2c$ this follows from $(r-c)_+\ge r/2$, and if $r<2c$ the right
side is negative. Second, $\psi_\gamma\ge\psi_1$ because $\gamma\ge1$.

Let $A_H=\{u\in U:|u_1|\ge H_1,|u_2|\ge H_2\}$, which we may assume
nonempty. For $z=(z_1,z_2)\in\mathbb Z^{k_1}\times\mathbb Z^{k_2}$ put
$\Gamma_z=U\cap(z+[0,1]^k)$, and let $\mathcal Z$ be the set of $z$ with
$\Gamma_z\cap A_H\ne\emptyset$. Each $\Gamma_z$ with $z\in\mathcal Z$ is a
nonempty product of compact intervals, hence compact and convex, of
diameter at most $\sqrt k$, and these sets cover $A_H$. If
$u\in\Gamma_z$, then $|u_i-z_i|\le\sqrt{k_i}$, so
$|u_i|\ge\tau_i(z_i):=(|z_i|-\sqrt{k_i})_+$. If moreover $z\in\mathcal Z$,
then $|z_i|\ge H_i-\sqrt{k_i}$. Let
$E_z=\exp[-a\sum_i\psi_{\gamma_i}(\tau_i(z_i))]\le1$. By
\eqref{eq:envelope-moment}, $\sup_{\Gamma_z}\norm{\xi}_{L^m}^m\le AE_z$,
and $E_z\le E_z^{1-k/m}$. \Cref{prop:cover}, applied with $L_z=B$ and
$\delta_z=\sqrt k$, gives
\[
 \E\sup_{A_H}|\xi|^m
 \le C\sum_{z\in\mathcal Z}
 \bigl\{AE_z+B^kA^{1-k/m}E_z^{1-k/m}\bigr\}
 \le C\bigl(A+B^kA^{1-k/m}\bigr)\prod_{i=1}^2S_i,
\]
where, with $a'=a(m-k)/m$,
\[
 S_i=\sum_{z_i\in\mathbb Z^{k_i}:\,|z_i|\ge H_i-\sqrt{k_i}}
 \exp\bigl[-a'\psi_{\gamma_i}(\tau_i(z_i))\bigr].
\]

Fix $i$ and suppress it from the notation, writing $k'=k_i$ and
$\psi=\psi_{\gamma_i}$. The number of $z\in\mathbb Z^{k'}$ with
$j\le|z|<j+1$ is at most $C(1+j)^{k'-1}$, and for such $z$,
\eqref{eq:psi-shift} gives $\psi(\tau(z))\ge\psi(j)/4-k'$. Only integers
$j\ge J:=(H-\sqrt{k'}-1)_+$ occur. Furthermore,
\begin{align*}
 S_i &\le Ce^{a'k'}\sum_{j\ge J}(1+j)^{k'-1}e^{-a'\psi(j)/4}
 \le Ce^{a'k'}e^{-a'\psi(J)/8}
 \sum_{j\ge0}(1+j)^{k'-1}e^{-a'\min(j^2,j)/8}
 \le Ce^{-a'\psi(J)/8}\;,
\end{align*}
where we have used that facts that $\psi$ is nondecreasing and $\psi\ge\psi_1$. By \eqref{eq:psi-shift} again, $\psi(J)\ge\psi(H)/4-(\sqrt{k'}+1)^2$, so $S_i\le C\exp\{-a'\psi_{\gamma_i}(H_i)/32\}$. Since $a''=a'/32$, this proves \eqref{eq:envelope-bound}. No constant depends on $\gamma_i$ or on $U$.
\end{proof}

\subsection{\texorpdfstring{Proof of \Cref{lem:projection}}{Proof of the common projection lemma}}\label{app:projection}

\begin{proof}
\textbf{A common projection.}
Since $\calX$ is bounded and $l\ge l_->0$, the quantity
$u_l(x,y)=\norm{x-y}^2/(2l^2)$ is uniformly bounded. Set
\[
 T_{M,l}(x,y)=\sum_{j=0}^{M}\frac{(-1)^j}{j!}u_l(x,y)^j.
\]
For each fixed $k\le3$, differentiating $l^{-2j}$ introduces a
polynomial factor of degree $k$ in $j$, with bounded additional
powers of $l^{-1}$. Therefore, for a fixed $B<\infty$,
\begin{equation}\label{eq:taylor-derivative-tail}
 \sup_{x,y\in\calX,\,l\in[l_-,l_+]}
 |\partial_l^k(R_l(x,y)-T_{M,l}(x,y))|
 \le C_k\sum_{j>M}\frac{(j+1)^kB^j}{j!}
 \le e^{-M\log M+C_kM}.
\end{equation}
The last inequality is the case $\gamma=1$ of the following factorial
tail bound. Since $e^j\ge j^j/j!$, Stirling's lower bound
\begin{equation}\label{eq:stirling}
 \log j!\ge j\log j-j,\qquad j\ge1,
\end{equation}
holds, together with the trivial upper bound $\log j!\le j\log j$. Fix
$k\ge0$, $B>0$, and $\gamma\in(0,1]$, and put $a_j=(j+1)^kB^j/(j!)^\gamma$.
Then
\begin{equation}\label{eq:factorial-tail}
 \sum_{j>M}a_j\le e^{-\gamma M\log M+CM},\qquad M\ge1,
\end{equation}
where $C$ depends only on $k$, $B$, and $\gamma$. Indeed,
$a_{j+1}/a_j=\{(j+2)/(j+1)\}^kB(j+1)^{-\gamma}\le2^kB(j+1)^{-\gamma}$, which
is at most $1/2$ once $j+1\ge(2^{k+1}B)^{1/\gamma}$. For $M$ beyond this
threshold, the tail is at most $2a_{M+1}$. By \eqref{eq:stirling},
$\log(M+1)!\ge(M+1)\log(M+1)-(M+1)\ge M\log M-M-1$, so
%{\footnotesize
\begin{align*}
     \log(2a_{M+1}) &\le\log2+k\log(M+2)+(M+1)\log^+B-\gamma(M\log M-M-1) \\ 
     &\le -\gamma M\log M+CM.
\end{align*}
%}
For the finitely many smaller $M$, the tail is at most
$\sum_ja_j<\infty$, which is absorbed by enlarging $C$.

Let $P_n$ project onto evaluations of polynomials of total degree at most
$2M$. Each column and row of the matrix $T_{M,l}$ lies in that space,
as do its first three lengthscale derivatives. Thus
$P_n(\partial_l^kT_{M,l})P_n=\partial_l^kT_{M,l}$ and
\[
 \norm{\partial_l^kR_n-P_n(\partial_l^kR_n)P_n}_{\op}
 \le 2n e^{-M\log M+C_kM}.
\]
Choose $M=\lceil c_Ab_n\rceil$ with a fixed $c_A>A+1$. Then
$\log M=(1+o(1))\log\log n$, so $M\log M=(c_A+o(1))\log n$, while
$C_kM=o(\log n)$. Hence the last bound is $n^{1-c_A+o(1)}\le n^{-A}$ for
large $n$. The rank is at most
$\binom{2M+p}{p}=O(b_n^p)$. Derivatives in $s$ cause only constant
factors; derivatives of $S_n$ in $t$ vanish. This proves
\eqref{eq:projection-approx}.

Both exact and reduced covariance matrices are bounded below by
$t_-\Id_n$. Put $\Omega_n=C_n-\widetilde C_n=S_n-P_nS_nP_n$, so that
$\norm{\Omega_n}_{\op}\le n^{-A}$ uniformly over $\Theta$ by
\eqref{eq:projection-approx}. The resolvent identity
$C_n^{-1}-\widetilde C_n^{-1}=-C_n^{-1}\Omega_n\widetilde C_n^{-1}$ gives
\[
 \norm{C_n^{-1}-\widetilde C_n^{-1}}_{\op}\le t_-^{-2}n^{-A}.
\]
Every raw covariance derivative through order three has operator norm
$O(n)$ uniformly over $\Theta$. Formulas for likelihood derivatives
through order three contain finitely many traces and quadratic forms
with at most three covariance derivative factors. Replacing factors one
at a time, using \eqref{eq:projection-approx}, bounds each difference by
$Cn^{4-A}(1+\norm{Y_n}^2/n)$. For the zeroth-order term,
\[
 l_n(\theta)-\widetilde l_n(\theta)
 =-\frac12\bigl\{\log\det C_n-\log\det\widetilde C_n\bigr\}
 -\frac12Y_n^\top\bigl(C_n^{-1}-\widetilde C_n^{-1}\bigr)Y_n.
\]
The quadratic part is at most $t_-^{-2}n^{-A}\norm{Y_n}^2$. For the log
determinant, let $C_n(r)=\widetilde C_n+r\Omega_n=(1-r)\widetilde C_n+rC_n$
for $r\in[0,1]$. As a convex combination of two matrices bounded below
by $t_-\Id_n$, each $C_n(r)$ satisfies $C_n(r)\succeq t_-\Id_n$. Jacobi's
formula gives $\frac d{dr}\log\det C_n(r)=\tr\{C_n(r)^{-1}\Omega_n\}$, and
$|\tr(AB)|\le n\norm A_{\op}\norm B_{\op}$ for $n\times n$ matrices.
Hence
\[
 \bigl|\log\det C_n-\log\det\widetilde C_n\bigr|
 \le\int_0^1\bigl|\tr\{C_n(r)^{-1}\Omega_n\}\bigr|\,dr
 \le t_-^{-1}n^{1-A}.
\]
Thus all differences in \eqref{eq:likelihood-approx} are at most
$Cn^{4-A}(1+\norm{Y_n}^2/n)$ uniformly over $\Theta$. Since
$\E_{\theta_0}\norm{Y_n}^2=\tr C_n(\theta_0)=n(s_0+t_0)$,
\eqref{eq:likelihood-approx} follows.

\textbf{Feature differentiation.}
Define
\[
 \phi_{\alpha,l}(x)
 =e^{-\norm{x}^2/(2l^2)}
   \frac{x^\alpha}{l^{|\alpha|}\sqrt{\alpha!}},
 \qquad R_l(x,y)=\sum_{\alpha\in\N_0^p}
                 \phi_{\alpha,l}(x)\phi_{\alpha,l}(y).
\]
This feature expansion of the kernel is derived in
\eqref{eq:kernel-expansion} in the proof of \Cref{lem:recovery}. The
series and its fixed-order lengthscale derivatives converge absolutely
and uniformly on bounded sets. Direct differentiation gives
\begin{equation}\label{eq:feature-raising}
 \partial_l\phi_{\alpha,l}
 =l^{-1}\left[
 -|\alpha|\phi_{\alpha,l}
 +\sum_{j=1}^p\sqrt{(\alpha_j+1)(\alpha_j+2)}
                \phi_{\alpha+2e_j,l}\right].
\end{equation}
For $M\ge0$, let $I_M=\{\alpha\in\N_0^p:|\alpha|\le M\}$, which has
$q_M=\binom{M+p}{p}$ elements, and let $B_M(l)\in\R^{n\times q_M}$ be the
observation matrix of the features of degree at most $M$:
\[
 [B_M(l)]_{i\alpha}=\phi_{\alpha,l}(x_{n,i}),
 \qquad 1\le i\le n,\ \alpha\in I_M.
\]
Thus the column of $B_M(l)$ indexed by $\alpha$ is the feature
$\phi_{\alpha,l}$ evaluated at the $n$ sampling points, and
\[
 [B_M(l)B_M(l)^\top]_{ij}
 =\sum_{\alpha\in I_M}\phi_{\alpha,l}(x_{n,i})\phi_{\alpha,l}(x_{n,j})
\]
is the degree-$M$ truncation of $R_n(l)$. For $M\le M'$, let
$E_{M,M'}\in\R^{q_{M'}\times q_M}$ be the coordinate embedding,
$E_{M,M'}e_\alpha=e_\alpha$ for $\alpha\in I_M$. It has orthonormal
columns, so $\norm{E_{M,M'}}_{\op}=1$, and $B_M=B_{M'}E_{M,M'}$ because
the columns of $B_M$ are among those of $B_{M'}$.

Applying \eqref{eq:feature-raising} column by column gives
$\partial_lB_M(l)=B_{M+2}(l)V_M(l)$, where
$V_M(l)\in\R^{q_{M+2}\times q_M}$ has columns
\[
 V_M(l)e_\alpha
 =l^{-1}\Bigl[-|\alpha|e_\alpha
 +\sum_{j=1}^p\sqrt{(\alpha_j+1)(\alpha_j+2)}\,e_{\alpha+2e_j}\Bigr],
 \qquad\alpha\in I_M.
\]
In particular, $V_M(l)=l^{-1}V_M^\circ$ with $V_M^\circ$ independent of
$l$. Since $\sqrt{(\alpha_j+1)(\alpha_j+2)}\le|\alpha|+2\le M+2$, every
column has norm at most $(p+1)(M+2)/l$. A bound on column norms does not
by itself bound an operator norm, but it does once the matrix is split
into pieces with orthogonal columns. Write
\begin{align*}
 & lV_M(l)=-E_{M,M+2}N_M+\sum_{j=1}^pW_{M,j},\quad N_Me_\alpha=|\alpha|e_\alpha,\quad\text{and}\quad W_{M,j}e_\alpha=\sqrt{(\alpha_j+1)(\alpha_j+2)}\,e_{\alpha+2e_j}.
\end{align*}
In each of these $p+1$ matrices, distinct columns are multiples of
distinct coordinate vectors, because the maps $\alpha\mapsto\alpha$ and
$\alpha\mapsto\alpha+2e_j$ are injective. The columns are therefore
orthogonal. For a matrix $W$ with orthogonal columns,
$\norm{Wx}^2=\sum_\alpha x_\alpha^2\norm{We_\alpha}^2$, so its operator
norm equals its largest column norm, here at most $M+2$. Hence
\begin{equation}\label{eq:raising-norm}
 \norm{V_M(l)}_{\op}\le l_-^{-1}(p+1)(M+2)\le C(M+1).
\end{equation}
Equivalently, each row of $V_M(l)$ also has at most $p+1$ nonzero entries, so the Schur test applies. We now show that, for $k\le3$,
\begin{equation}\label{eq:truncated-derivative}
 \partial_l^k\{B_M(l)B_M(l)^\top\}
 =B_{M+2k}(l)H_{M,k}(l)B_{M+2k}(l)^\top,
 \quad \norm{H_{M,k}(l)}_{\op}\le C_k(M+1)^k,
\end{equation}
where $H_{M,k}(l)=l^{-k}H_{M,k}^\circ$ is a symmetric
$q_{M+2k}\times q_{M+2k}$ matrix and $H_{M,k}^\circ$ does not depend on
$l$. For $k=0$, take $H_{M,0}=\Id_{q_M}$. Suppose that
\eqref{eq:truncated-derivative} holds for some $k$, and abbreviate
$M_k=M+2k$ and $E=E_{M_k,M_k+2}$. Since $B_{M_k}=B_{M_k+2}E$ and
$\partial_lB_{M_k}=B_{M_k+2}V_{M_k}$, the product rule gives
\[
 \partial_l^{k+1}\{B_MB_M^\top\}
 =B_{M_k+2}\bigl\{V_{M_k}H_{M,k}E^\top+E(\partial_lH_{M,k})E^\top
 +EH_{M,k}V_{M_k}^\top\bigr\}B_{M_k+2}^\top.
\]
The matrix in braces is $H_{M,k+1}$. It is symmetric, and it has the form
$l^{-(k+1)}H_{M,k+1}^\circ$ because $V_{M_k}=l^{-1}V_{M_k}^\circ$ and
$\partial_lH_{M,k}=-kl^{-1}H_{M,k}$. Since $\norm E_{\op}=1$, by
\eqref{eq:raising-norm},
\[
 \norm{H_{M,k+1}}_{\op}
 \le\bigl\{2C(M+2k+1)+kl_-^{-1}\bigr\}\norm{H_{M,k}}_{\op}
 \le C_{k+1}(M+1)^{k+1}.
\]
This is how all smaller feature spaces are embedded in the common space
through degree $M+2k$: the embeddings $E$ carry every intermediate term
into $\R^{q_{M+2k}}$, so that the single feature matrix $B_{M+2k}$
appears on both sides.

Positive feature truncation gives
$sB_{M+2k}B_{M+2k}^\top\preceq C_n$ and therefore
\begin{equation}\label{eq:feature-contraction}
 \norm{\sqrt{s}\,C_n^{-1/2}B_{M+2k}}_{\op}\le1.
\end{equation}
The whitened truncated derivative consequently has operator norm
$O(M^k)$ and rank $O(M^p)$, so its Frobenius norm is $O(M^{p/2+k})$.

To control the derivative tail, collect terms with $|\alpha|=j$.
Absolute summation of the monomials bounds these terms by $B^j/j!$
on the fixed bounded set, and a $k$th lengthscale derivative adds at
most $C_k(j+1)^k$. Hence \eqref{eq:factorial-tail} with $\gamma=1$ gives
the tail bound in \eqref{eq:taylor-derivative-tail}. For $M$ a sufficiently large multiple
of $b_n$, its operator norm is an arbitrarily small power of $n^{-1}$,
and its Frobenius norm is at most $\sqrt n$ times that bound.
Whitening multiplies these errors by at most $t_-^{-1}$.
This proves \eqref{eq:raw-operator} and \eqref{eq:raw-frobenius} for
lengthscale derivatives. One $s$-derivative removes the factor $s$;
two $s$-derivatives give zero.

For a compression $A_\theta=U^\top C_nU$, the derivatives are
$U^\T(\partial_l^kC_n)U$, so by \eqref{eq:truncated-derivative} the
truncated part becomes
$s(U^\top B_{M+2k})H_{M,k}(U^\top B_{M+2k})^\T$. Since
$sB_{M+2k}B_{M+2k}^\top\preceq C_n$, also
$sU^\top B_{M+2k}B_{M+2k}^\top U\preceq U^\top C_nU=A_\theta$, which gives the
analogue $\norm{\sqrt s\,A_\theta^{-1/2}U^\top B_{M+2k}}_{\op}\le1$ of
\eqref{eq:feature-contraction}. The rank is still at most $q_{M+2k}$.
The compressed tail $U^\top\{\partial_l^k(R_n-B_MB_M^\T)\}U$ has operator
norm at most that of the uncompressed tail and Frobenius norm at most
$\sqrt r\le\sqrt n$ times it, and $A_\theta\succeq t_-\Id_r$. Hence every
estimate above holds verbatim for $A_\theta$, with the same constants.

Finally,
\[
 \partial_{h_s}^a\partial_{h_l}^c
 =b_n^{-(a+c)p/2-c}\partial_s^a\partial_l^c.
\]
Combining this identity with the preceding estimates gives
$O(\epsilon_n^{a+c})$ in operator norm and
$O(\epsilon_n^{a+c-1})$ in Frobenius norm. Since
$C_{n,t}=\Id_n$, its normalized relative derivative is
$n^{-1/2}C_n^{-1}$, which has operator norm at most
$t_-^{-1}n^{-1/2}$ and Frobenius norm at most $t_-^{-1}$.
All higher covariance derivatives involving $t$ vanish.
\end{proof}

\subsection{\texorpdfstring{Proof of \Cref{lem:recovery}}{Proof of the coefficient recovery lemma}}\label{app:recovery}

\begin{proof}
\textbf{Kernel expansion.}
For $x,y\in\R^p$, $\norm{x-y}^2=\norm x^2+\norm y^2-2\langle x,y\rangle$.
By the multinomial theorem,
$\langle x,y\rangle^j=\sum_{|\alpha|=j}\frac{j!}{\alpha!}x^\alpha y^\alpha$,
and therefore
\[
 \exp\Bigl(\frac{\langle x,y\rangle}{l^2}\Bigr)
 =\sum_{j=0}^\infty\frac{\langle x,y\rangle^j}{l^{2j}j!}
 =\sum_{j=0}^\infty\sum_{|\alpha|=j}\frac{x^\alpha y^\alpha}{l^{2j}\alpha!}
 =\sum_{\alpha\in\N_0^p}\frac{x^\alpha y^\alpha}{l^{2|\alpha|}\alpha!}.
\]
The rearrangement is justified because the same series with every term
replaced by its absolute value sums to
$\exp(\sum_i|x_iy_i|/l^2)<\infty$. Hence
\begin{equation}\label{eq:kernel-expansion}
 k_\theta(x,y)
 =s\,e^{-\norm x^2/(2l^2)}e^{-\norm y^2/(2l^2)}e^{\langle x,y\rangle/l^2}
 =s\sum_{\alpha\in\N_0^p}\phi_{\alpha,l}(x)\phi_{\alpha,l}(y),
\end{equation}
where
\[
 \phi_{\alpha,l}(x)=e^{-\norm x^2/(2l^2)}
 \frac{x^\alpha}{l^{|\alpha|}\sqrt{\alpha!}}.
\]
In particular,
$\sum_\alpha\phi_{\alpha,l}(x)^2=k_\theta(x,x)/s=1$.

\textbf{An analytic representation.}
Let $\xi_\alpha$, $\alpha\in\N_0^p$, be independent standard normal
variables, and define
\begin{equation}\label{eq:analytic-field}
 F(x)=\sqrt{s}\sum_{\alpha\in\N_0^p}\xi_\alpha\phi_{\alpha,l}(x)
 =\sqrt{s}\,e^{-\norm{x}^2/(2l^2)}
 \sum_{\alpha\in\N_0^p}
 \xi_\alpha\frac{x^\alpha}{l^{|\alpha|}\sqrt{\alpha!}}.
\end{equation}
For $\rho>0$ and $x\in[-\rho,\rho]^p$, $|\phi_{\alpha,l}(x)|\le
(\rho/l_-)^{|\alpha|}/\sqrt{\alpha!}$, so
\[
 \E\sum_\alpha|\xi_\alpha|\sup_{x\in[-\rho,\rho]^p}|\phi_{\alpha,l}(x)|
 \le\sum_{\alpha}\frac{(\rho/l_-)^{|\alpha|}}{\sqrt{\alpha!}}
 =\Bigl(\sum_{a=0}^\infty\frac{(\rho/l_-)^a}{\sqrt{a!}}\Bigr)^p<\infty.
\]
Hence, almost surely, the series in \eqref{eq:analytic-field} converges
absolutely and uniformly on every box, and the power series in it
defines an entire function. For fixed $x,y$, the partial sums also
converge in $L^2$ because $\sum_\alpha\phi_{\alpha,l}(x)^2=1$, and
independence of the $\xi_\alpha$ together with \eqref{eq:kernel-expansion}
gives
\[
 \Cov\{F(x),F(y)\}=s\sum_\alpha\phi_{\alpha,l}(x)\phi_{\alpha,l}(y)
 =k_\theta(x,y).
\]
Thus \eqref{eq:analytic-field} is an analytic version of the centered
Gaussian process with covariance \eqref{eq:rbf-intro}.

\textbf{Taylor coefficients and their tails.}
Let
\[
 H(x)=e^{\norm x^2/(2l^2)}F(x)
 =\sqrt s\sum_{\gamma\in\N_0^p}\xi_\gamma
 \frac{x^\gamma}{l^{|\gamma|}\sqrt{\gamma!}},
 \qquad a_l=-\frac1{2l^2},
\]
so that $F(x)=e^{a_l\norm x^2}H(x)$. Since
$e^{a_l\norm x^2}=\prod_je^{a_lx_j^2}=\sum_{r\in\N_0^p}a_l^{|r|}x^{2r}/r!$,
the Cauchy product of these absolutely convergent series gives
\[
 F(x)=\sqrt s\sum_{r\in\N_0^p}\sum_{\gamma\in\N_0^p}
 \frac{a_l^{|r|}}{r!\,l^{|\gamma|}\sqrt{\gamma!}}\,\xi_\gamma\,x^{2r+\gamma}.
\]
Collecting the terms with $2r+\gamma=\alpha$, the Taylor coefficients
$f_\alpha=[x^\alpha]F(x)$ are
\begin{equation}\label{eq:coefficient-explicit}
 f_\alpha=\sqrt{s}
 \sum_{2r\le\alpha}
 \frac{a_l^{|r|}}
 {r!\,l^{|\alpha|-2|r|}\sqrt{(\alpha-2r)!}}
 \xi_{\alpha-2r}.
\end{equation}
The summands involve distinct independent normal variables. For each
coordinate, the terms in
\[
 \sum_{r=0}^{\lfloor q/2\rfloor}
 \frac{q!}{r!r!(q-2r)!}
 |a_l|^{2r}(l^{-2})^{q-2r}
\]
form a subset of the nonnegative terms in the multinomial expansion
of $(2|a_l|+l^{-2})^q=(2l^{-2})^q$. Therefore
\[
 \alpha!\Var_\theta(f_\alpha)
 \le s(2l^{-2})^{|\alpha|},
 \qquad
 \norm{f_\alpha}_{L^2(P_\theta)}
 \le \frac{C^{|\alpha|+1}}{\sqrt{\alpha!}}.
\]
Let $T_KF(x)=\sum_{|\alpha|\le K}f_\alpha x^\alpha$, and write
$|x|^\alpha=\prod_i|x_i|^{\alpha_i}$. The number of $\alpha\in\N_0^p$ with
$|\alpha|=j$ is $\binom{j+p-1}{p-1}$. Hence the Cauchy--Schwarz
inequality over these indices and the multinomial theorem give
\begin{align}
 \Bigl(\sum_{|\alpha|=j}\frac{|x|^\alpha}{\sqrt{\alpha!}}\Bigr)^2
 &\le\binom{j+p-1}{p-1}\sum_{|\alpha|=j}\frac{x^{2\alpha}}{\alpha!}
 =\binom{j+p-1}{p-1}\frac1{j!}\sum_{|\alpha|=j}\frac{j!}{\alpha!}
 \prod_{i=1}^p(x_i^2)^{\alpha_i}\notag\\
 &=\binom{j+p-1}{p-1}\frac{\norm{x}^{2j}}{j!}.
 \label{eq:degree-sum-bound}
\end{align}
Since $|F(x)-T_KF(x)|\le\sum_{j>K}\sum_{|\alpha|=j}|f_\alpha||x|^\alpha$,
Minkowski's inequality, the bound on $\norm{f_\alpha}_{L^2(P_\theta)}$,
\eqref{eq:degree-sum-bound}, and $\binom{j+p-1}{p-1}\le(j+1)^{p-1}$ give,
for $x$ in a fixed bounded set,
\[
 \norm{F(x)-T_KF(x)}_{L^2(P_\theta)}
 \le C\sum_{j>K}\frac{(j+1)^{(p-1)/2}(C\norm x)^j}{(j!)^{1/2}}.
\]
By \eqref{eq:factorial-tail} with $\gamma=1/2$, which rests on
Stirling's bound $\log j!\ge j\log j-j$, the right side is at most
$e^{-(K\log K)/2+CK}$. Squaring yields
\begin{equation}\label{eq:field-taylor-tail}
 \sup_{\theta\in\Theta}\sup_{x\in\calX}
 \E_\theta|F(x)-T_KF(x)|^2
 \le e^{-K\log K+CK}.
\end{equation}

\textbf{Weights and polynomial Gram matrix.}
Partition $Q$ into nearest-observation cells $V_{n,i}$, with a fixed
tie-breaking rule. Define $w_{n,i}=|V_{n,i}|/|Q|$. Then
\begin{equation}\label{eq:voronoi-weights}
 \sum_iw_{n,i}=1,\qquad
 0\le w_{n,i}\le Ch_n(Q)^p\le Cn^{-\beta}.
\end{equation}
Indeed, $V_{n,i}\subseteq B(x_{n,i},h_n(Q))$. Zero weights and repeated
locations are allowed.

Use the product Legendre basis on $[0,L]^p$, restricted to total degree
at most $K$, and normalized for uniform probability measure on $Q$.
Write its elements as $\psi_\nu$, $|\nu|\le K$, and set
\[
 \Phi_{i,\nu}=\psi_\nu(x_{n,i}),\quad
 W=\diag(w_{n,1},\ldots,w_{n,n}),\quad M_K=\Phi^\top W\Phi.
\]
The shifted Legendre polynomials have the explicit form
\begin{equation}\label{eq:legendre-coefficients}
 P_j(2x-1)
 =\sum_{a=0}^j(-1)^{j-a}
       \binom ja\binom{j+a}{a}x^a.
\end{equation}
This follows from Rodrigues' formula
$P_j(t)=(2^jj!)^{-1}\frac{d^j}{dt^j}(t^2-1)^j$
\citep[\S18.5]{dlmf}. With $t=2x-1$, we have $t^2-1=4(x^2-x)$ and
$\frac{d}{dt}=\frac12\frac{d}{dx}$, so
\begin{align*}
 P_j(2x-1) &=\frac1{j!}\frac{d^j}{dx^j}\bigl\{x^j(x-1)^j\bigr\}
 =\frac1{j!}\sum_{a=0}^j\binom ja(-1)^{j-a}\frac{d^j}{dx^j}x^{j+a} =\sum_{a=0}^j(-1)^{j-a}\binom ja\frac{(j+a)!}{j!\,a!}x^a,
\end{align*}
which is \eqref{eq:legendre-coefficients}. This identity implies bounds $e^{CK}K^C$ for these basis polynomials and their first derivatives on any fixed enlarged box. The same bounds hold for products of two basis polynomials, after changing $C$. Comparing the integral over each $V_{n,i}$ with its value at $x_{n,i}$ gives
\[
 |[M_K-\Id]_{\mu\nu}|
 \le h_n(Q)
 \sup|\nabla(\psi_\mu\psi_\nu)|.
\]
The supremum is taken on a fixed box containing all relevant line
segments. This remains valid when the assigned node is outside $Q$.
The number of basis elements is polynomial in $K$, and hence
\begin{equation}\label{eq:gram-approx}
 \norm{M_K-\Id}_{\op}\le h_n(Q)e^{CK}K^C.
\end{equation}
For $K=K_n$, this is $n^{-\kappa+o(1)}$, so $M_{K_n}$ and its inverse
are uniformly bounded for all sufficiently large $n$.

\textbf{Weighted regression and coefficient extraction.}
Fit the polynomial with Legendre coefficient vector
\[
 \widehat c=M_{K_n}^{-1}\Phi^\top WY_n.
\]
The basis has $q_K=\binom{K+p}{p}$ elements. For $|\alpha|\le K$, let
$L_{\alpha,K}\in\R^{1\times q_K}$ be the row vector that maps the Legendre
coefficient vector of a polynomial to its normalized monomial
coefficient of index $\alpha$:
\[
 L_{\alpha,K}c=\sqrt{\alpha!}\,[x^\alpha]
 \sum_{|\nu|\le K}c_\nu\psi_\nu(x),
 \qquad c\in\R^{q_K}.
\]
By
\eqref{eq:legendre-coefficients}, individual unnormalized Legendre
coefficients are at most $8^j$ on $[0,1]$. Rescaling to $[0,L]$,
taking products, and allowing polynomial dimension factors gives
\begin{equation}\label{eq:extraction-bound}
 \norm{L_{\alpha,K}}^2
 \le e^{m\log m+CK},\qquad |\alpha|\le m\le K.
\end{equation}
Here $\alpha!\le|\alpha|!\le m!\le e^{m\log m}$ was used, by the trivial
upper counterpart of \eqref{eq:stirling}.
Define $\widehat\zeta_\alpha=L_{\alpha,K_n}\widehat c$ for
$|\alpha|\le m_n$. This is a fixed linear function of $Y_n$.

Write $Y_n=F_n+\bar\varepsilon$, where
$F_n=(F(x_{n,1}),\ldots,F(x_{n,n}))^\T$ and
$\bar\varepsilon=(\varepsilon_{n,1},\ldots,\varepsilon_{n,n})^\T$.
Weighted regression reproduces $T_KF$ exactly. If
$r_i=F(x_{n,i})-T_KF(x_{n,i})$, the residual coefficient error obeys
\[
 |L_{\alpha,K}M_K^{-1}\Phi^\top Wr|^2
 \le (L_{\alpha,K}M_K^{-1}L_{\alpha,K}^\T)\,r^\top Wr.
\]
For the noise error, $\Cov_\theta(\bar\varepsilon)=t\Id_n$ and
$W^2\preceq(\max_iw_{n,i})W$ imply
\[
 \Var_\theta(L_{\alpha,K}M_K^{-1}\Phi^\top W\bar\varepsilon)
 \le t_+\max_iw_{n,i}\,
        L_{\alpha,K}M_K^{-1}L_{\alpha,K}^\top.
\]
The field residual and noise are independent. Using
\eqref{eq:field-taylor-tail}, \eqref{eq:voronoi-weights},
\eqref{eq:gram-approx}, and \eqref{eq:extraction-bound}, we obtain
\begin{equation}\label{eq:recovery-error-calculation}
 \sup_\theta\E_\theta(\widehat\zeta_\alpha-\zeta_\alpha)^2
 \le e^{m_n\log m_n+CK_n}
      \{n^{-\beta}+e^{-K_n\log K_n+CK_n}\}.
\end{equation}
Since $\log m_n$ and $\log K_n$ are both $(1+o(1))\log\log n$, we have
$m_n\log m_n=(\beta/4+o(1))\log n$ and
$K_n\log K_n=(\beta/2+o(1))\log n$, while $CK_n=o(\log n)$. Hence the two
contributions in
\eqref{eq:recovery-error-calculation} are respectively
$n^{-3\beta/4+o(1)}$ and $n^{-\beta/4+o(1)}$.
This proves \eqref{eq:recovery-mse}.

\textbf{Conditioning of the exact coefficient covariance.}
Let $e_\alpha$ denote the coordinate vector indexed by $\alpha$ in the
space $|\alpha|\le m$. Define
\[
 D_{m,l}e_\alpha=l^{-|\alpha|}e_\alpha
\]
and a truncated raising matrix
\[
 Q_me_\alpha=
 \sum_{\substack{1\le j\le p\\|\alpha|+2\le m}}
 \sqrt{(\alpha_j+1)(\alpha_j+2)}\,e_{\alpha+2e_j}.
\]
Since $Q_m$ raises the total degree by two, $Q_m^r=0$ for $2r>m$. As
for $V_M$ in the proof of \Cref{lem:projection}, $Q_m$ splits into $p$
matrices with orthogonal columns of norm at most $m$, so
$\norm{Q_m}_{\op}\le pm$.

We now derive \eqref{eq:triangular-transform} below in four steps. For a
function $f$ given by a power series $f(x)=\sum_\alpha([x^\alpha]f)\,x^\alpha$
that converges absolutely on $\R^p$, its normalized coefficients are
\[
 [\mathcal Nf]_\alpha=\sqrt{\alpha!}\,[x^\alpha]f,
 \qquad\text{so that}\qquad
 f(x)=\sum_{\alpha\in\N_0^p}[\mathcal Nf]_\alpha\,
 \frac{x^\alpha}{\sqrt{\alpha!}}.
\]
Write $\mathcal N_mf=([\mathcal Nf]_\alpha:|\alpha|\le m)\in\R^{q_m}$. By
\eqref{eq:coefficient-definition}, $\zeta^{(m)}=\mathcal N_mF$.

\textbf{Step 1: multiplication by $\norm x^2$.} Since
$\norm x^2f(x)=\sum_{j=1}^p\sum_\alpha([x^\alpha]f)\,x^{\alpha+2e_j}$,
for every $\beta$,
\[
 [x^\beta](\norm x^2f)=\sum_{j:\,\beta_j\ge2}[x^{\beta-2e_j}]f.
\]
Multiply by $\sqrt{\beta!}$ and write $\alpha=\beta-2e_j$. Then
$\beta!/\alpha!=(\alpha_j+1)(\alpha_j+2)$, so
\[
 [\mathcal N(\norm x^2f)]_\beta
 =\sum_{j:\,\beta_j\ge2}\sqrt{(\alpha_j+1)(\alpha_j+2)}\,
 [\mathcal Nf]_{\beta-2e_j}.
\]
For $|\beta|\le m$, the right side involves only indices of degree at
most $m-2$, and it is exactly the $\beta$th entry of $Q_m\mathcal N_mf$.
Hence $\mathcal N_m(\norm x^2f)=Q_m\mathcal N_mf$.

\textbf{Step 2: multiplication by $e^{a\norm x^2}$.} By induction,
$\mathcal N_m(\norm x^{2r}f)=Q_m^r\mathcal N_mf$ for every $r\ge0$. For
$a\in\R$, $e^{a\norm x^2}f=\sum_{r\ge0}a^r\norm x^{2r}f/r!$, and all these
power series converge absolutely, so coefficients may be extracted term
by term. Since $Q_m^r=0$ for $2r>m$,
\[
 \mathcal N_m\bigl(e^{a\norm x^2}f\bigr)
 =\sum_{r\le m/2}\frac{a^r}{r!}Q_m^r\mathcal N_mf
 =e^{aQ_m}\mathcal N_mf.
\]

\textbf{Step 3: the coefficients of $H$.} By the definition of $H$,
$[x^\gamma]H=\sqrt s\,\xi_\gamma/(l^{|\gamma|}\sqrt{\gamma!})$, so
$[\mathcal NH]_\gamma=\sqrt s\,l^{-|\gamma|}\xi_\gamma$, that is,
$\mathcal N_mH=\sqrt s\,D_{m,l}\xi^{(m)}$.

\textbf{Step 4: the coefficients of $F$.} Since $F=e^{a_l\norm x^2}H$
with $a_l=-1/(2l^2)$, Steps 2 and 3 give
\begin{equation}\label{eq:triangular-transform}
 \zeta^{(m)}=\mathcal N_mF
 =e^{a_lQ_m}\mathcal N_mH
 =\sqrt{s}\,e^{-Q_m/(2l^2)}D_{m,l}\xi^{(m)}.
\end{equation}
Entrywise, this is \eqref{eq:coefficient-explicit} multiplied by
$\sqrt{\alpha!}$.
The transform and its inverse have operator norms at most $e^{Cm}$:
$\norm{e^{\pm Q_m/(2l^2)}}_{\op}\le e^{pm/(2l_-^2)}\le e^{Cm}$ and
$\norm{D_{m,l}^{\pm1}}_{\op}\le e^{Cm}$. Constants can be absorbed
into the exponent for $m\ge1$. This proves
\eqref{eq:coefficient-eigenvalues}.

\textbf{Relative covariance error and affinity.}
Let $\delta_n=\widehat\zeta^{(m_n)}-\zeta^{(m_n)}$, a centered random
vector. Since $q_{m_n}\le(m_n+1)^p=n^{o(1)}$ and $e^{Cm_n}=n^{o(1)}$,
\eqref{eq:recovery-mse} and \eqref{eq:coefficient-eigenvalues} give
\begin{align*}
 \sup_\theta\E_\theta\norm{\delta_n}^2
 &=\sup_\theta\sum_{|\alpha|\le m_n}
 \E_\theta(\widehat\zeta_\alpha-\zeta_\alpha)^2
 \le n^{-\beta/4+o(1)},\\
 \sup_\theta\E_\theta\norm{\zeta^{(m_n)}}^2
 &=\sup_\theta\tr\Sigma_{m_n,\theta}\le q_{m_n}e^{Cm_n}=n^{o(1)}.
\end{align*}
Since $\widehat\zeta^{(m_n)}=\zeta^{(m_n)}+\delta_n$,
\[
 \widehat\Sigma_{n,\theta}-\Sigma_{m_n,\theta}
 =\E_\theta\bigl[\zeta^{(m_n)}\delta_n^\top\bigr]
 +\E_\theta\bigl[\delta_n\zeta^{(m_n)\top}\bigr]
 +\E_\theta\bigl[\delta_n\delta_n^\top\bigr].
\]
For unit vectors $u,v$, the Cauchy--Schwarz inequality gives
\[
 \bigl|u^\top\E_\theta[\zeta^{(m_n)}\delta_n^\T]v\bigr|
 \le\bigl\{\E_\theta(u^\top\zeta^{(m_n)})^2\,
 \E_\theta(v^\top\delta_n)^2\bigr\}^{1/2}
 \le\bigl\{\E_\theta\norm{\zeta^{(m_n)}}^2\,
 \E_\theta\norm{\delta_n}^2\bigr\}^{1/2},
\]
and similarly for the transposed term, while
$\norm{\E_\theta[\delta_n\delta_n^\T]}_{\op}
=\sup_{\norm u=1}\E_\theta(u^\top\delta_n)^2\le\E_\theta\norm{\delta_n}^2$.
Therefore
\begin{align*}
 \norm{\widehat\Sigma_{n,\theta}-\Sigma_{m_n,\theta}}_{\op}
 &\le \E_\theta\norm{\delta_n}^2+
 2\sqrt{\E_\theta\norm{\delta_n}^2\,
             \E_\theta\norm{\zeta^{(m_n)}}^2}\\
 &\le n^{-\beta/4+o(1)}+n^{-\beta/8+o(1)}
 =n^{-\beta/8+o(1)}.
\end{align*}
Finally,
$\norm{\Sigma^{-1/2}A\Sigma^{-1/2}}_{\op}\le\norm{\Sigma^{-1}}_{\op}\norm A_{\op}$
and $\norm{\Sigma_{m_n,\theta}^{-1}}_{\op}\le e^{Cm_n}=n^{o(1)}$ by
\eqref{eq:coefficient-eigenvalues}, which proves
\eqref{eq:relative-recovery}.

Let $\eta_n=n^{-\beta/8+o(1)}$ be a common bound in that equation.
For large $n$,
\[
 (1-\eta_n)\Sigma_{m_n,\theta}
 \preceq\widehat\Sigma_{n,\theta}
 \preceq(1+\eta_n)\Sigma_{m_n,\theta}.
\]
The same inequalities hold for sums of two such covariance matrices.
Using \eqref{eq:gaussian-affinity} and the corresponding log determinant
bounds, the absolute difference of log affinities is at most
$Cq_{m_n}\eta_n=o(1)$. This proves \eqref{eq:recovery-affinity}.
The comparison is multiplicative for affinity, which will be essential
when the affinity itself is exponentially small.
\end{proof}

\subsection{\texorpdfstring{Proof of \Cref{lem:affinity}}{Proof of the affinity separation lemma}}\label{app:affinity}

\begin{proof}
\textbf{Triangular structure.}
Order the indices $\alpha\in I_m=\{|\alpha|\le m\}$ by total degree, with
an arbitrary order within each degree. Write $\alpha(1),\ldots,\alpha(q)$
for the ordered indices, where $q=q_m$, write $\zeta_k$ and $\xi_k$ for
the coordinates of $\zeta^{(m)}$ and $\xi^{(m)}$ with index $\alpha(k)$,
and put $\zeta_{<k}=(\zeta_1,\ldots,\zeta_{k-1})$ and
$\xi_{<k}=(\xi_1,\ldots,\xi_{k-1})$. By \eqref{eq:triangular-transform},
$\zeta^{(m)}=\sqrt s\,\mathcal T_\theta\xi^{(m)}$ with
$\mathcal T_\theta=e^{-Q_m/(2l^2)}D_{m,l}$. Since $Q_m$ raises total degree
by two, every nonzero entry of $Q_m^r$, $r\ge1$, lies in a row of
strictly higher degree than its column. Hence
$e^{-Q_m/(2l^2)}=\sum_{r\le m/2}(-2l^2)^{-r}Q_m^r/r!$ is unit lower
triangular in this order, and $\mathcal T_\theta$ is lower triangular with
diagonal entries $l^{-|\alpha(k)|}$:
\[
 \zeta_k=\sqrt s\,l^{-|\alpha(k)|}\xi_k
 +\sqrt s\sum_{i<k}(\mathcal T_\theta)_{ki}\,\xi_i,
 \qquad k=1,\ldots,q.
\]
The leading $(k-1)\times(k-1)$ block of $\mathcal T_\theta$ is invertible,
so $\zeta_{<k}$ and $\xi_{<k}$ are invertible linear functions of each
other, and $\xi_k$ is independent of $\xi_{<k}$. Therefore, under
$P_\theta$, the conditional law of $\zeta_k$ given $\zeta_{<k}$ is
$N(\mu_{\theta,k}(\zeta_{<k}),\omega_{\theta,k})$, where $\mu_{\theta,k}$ is a
linear function of $\zeta_{<k}$ and
\begin{equation}\label{eq:conditional-variance}
 \omega_{\theta,k}=\Var_\theta(\zeta_k\mid\zeta_{<k})
 =s\,l^{-2|\alpha(k)|}.
\end{equation}
The conditional variance does not depend on $\zeta_{<k}$.

\textbf{Univariate affinity.}
For normal densities with means $\mu,\mu'$ and variances $\omega,\omega'$, the
exponent of the product of the two square roots satisfies
\[
 -\frac{(x-\mu)^2}{4\omega}-\frac{(x-\mu')^2}{4\omega'}
 =-\frac{(x-\bar\mu)^2}{2w}-\frac{(\mu-\mu')^2}{4(\omega+\omega')},
 \qquad w=\frac{2\omega\omega'}{\omega+\omega'},
\]
for a suitable $\bar\mu$. Integrating with respect to $x$ gives
\begin{align*}
 \Aff\{N(\mu,\omega),N(\mu',\omega')\}
 &=\frac{\sqrt{2\pi w}}{\sqrt{2\pi}\,(\omega\omega')^{1/4}}
 \exp\Bigl\{-\frac{(\mu-\mu')^2}{4(\omega+\omega')}\Bigr\} =\left(\frac{2\sqrt{\omega\omega'}}{\omega+\omega'}\right)^{1/2}
   \exp\left\{-\frac{(\mu-\mu')^2}{4(\omega+\omega')}\right\}.
\end{align*}
Since $(\omega+\omega')/(2\sqrt{\omega\omega'})
=\frac12(\sqrt{\omega/\omega'}+\sqrt{\omega'/\omega})
=\cosh\{\frac12\log(\omega/\omega')\}$, it follows that
\begin{equation}\label{eq:univariate-affinity}
 \Aff\{N(\mu,\omega),N(\mu',\omega')\}
 \le\cosh\Bigl\{\frac12\log\frac\omega{\omega'}\Bigr\}^{-1/2}
 \qquad\text{for all }\mu,\mu'.
\end{equation}

\textbf{Successive integration.}
Let $\pi_\theta$ denote the density of $\zeta^{(m)}$ under $P_\theta$, and
$\pi_\theta(\zeta_k\mid\zeta_{<k})$ its conditional densities, so that
$\pi_\theta(\zeta)=\prod_{k=1}^q\pi_\theta(\zeta_k\mid\zeta_{<k})$. For
$0\le k\le q$ put
\[
 \mathcal A_k=\int_{\R^k}\prod_{i=1}^k
 \sqrt{\pi_\theta(\zeta_i\mid\zeta_{<i})\,
 \pi_{\theta_0}(\zeta_i\mid\zeta_{<i})}\;d\zeta_1\cdots d\zeta_k,
\]
so that $\mathcal A_0=1$ and
$\mathcal A_q=\Aff(\Sigma_{m,\theta},\Sigma_{m,\theta_0})$. The integrand is
nonnegative, so by Tonelli's theorem we may integrate over $\zeta_k$
first, with $\zeta_{<k}$ fixed. By \eqref{eq:univariate-affinity},
\[
 \int_\R\sqrt{\pi_\theta(\zeta_k\mid\zeta_{<k})\,
 \pi_{\theta_0}(\zeta_k\mid\zeta_{<k})}\;d\zeta_k
 \le\rho_k:=\cosh\Bigl\{\frac12\log\frac{\omega_{\theta,k}}{\omega_{\theta_0,k}}\Bigr\}^{-1/2}.
\]
This bound is the same for every value of $\zeta_{<k}$, because the
conditional variances \eqref{eq:conditional-variance} do not depend on
$\zeta_{<k}$. The conditional means, which do depend on $\zeta_{<k}$ and
on the parameter, enter only through the exponential factor of the
univariate affinity, which is at most one. Hence
$\mathcal A_k\le\rho_k\mathcal A_{k-1}$, and by induction
$\Aff(\Sigma_{m,\theta},\Sigma_{m,\theta_0})\le\prod_{k=1}^q\rho_k$.
Finally, with $u=\log(s/s_0)$ and $v=\log(l/l_0)$,
\eqref{eq:conditional-variance} gives
$\log(\omega_{\theta,k}/\omega_{\theta_0,k})=u-2|\alpha(k)|v$. Therefore
\begin{equation}\label{eq:coefficient-affinity-product}
 \Aff(\Sigma_{m,\theta},\Sigma_{m,\theta_0})
 \le\prod_{k=1}^q\cosh\bigl(u/2-|\alpha(k)|v\bigr)^{-1/2}
 =\exp\left\{-\frac12\sum_{|\alpha|\le m}
             \log\cosh(u/2-|\alpha|v)\right\}.
\end{equation}

\textbf{A fixed fraction of separated degrees.}
To bound the sum, select a multiindex $\alpha$ uniformly from $I_m$, let
$J=|\alpha|$, and put $Z=J/m$ for $m\ge1$. The number of indices of degree
$j$ is $\binom{j+p-1}{p-1}$. Binomial summation gives
\[
 \E Z=\frac{p}{p+1},\qquad
 \E Z^2=\frac{p}{p+2}
       +\frac{p}{(p+1)(p+2)m}.
\]
The matrix
$\mathbb M=\begin{pmatrix}1&-\E Z\\-\E Z&\E Z^2\end{pmatrix}$ has
$\det\mathbb M=\Var Z\ge\frac p{p+2}-\frac{p^2}{(p+1)^2}
=\frac p{(p+2)(p+1)^2}$ and $\tr\mathbb M\le2$, so
\[
 \lambda_{\min}(\mathbb M)\ge\frac{\det\mathbb M}{\tr\mathbb M}
 \ge c_p:=\frac p{2(p+2)(p+1)^2}.
\]
Put $a=u/2$, $b=mv$, and $L_0=\sqrt{a^2+b^2}$, so that
$a-bZ=u/2-|\alpha|v$. Then
\[
 \E(a-bZ)^2=a^2-2ab\,\E Z+b^2\E Z^2
 =(a,b)\,\mathbb M\,(a,b)^\top\ge c_pL_0^2,
\]
and, since $0\le Z\le1$, $|a-bZ|\le|a|+|b|\le\sqrt2L_0$. If $L_0=0$, then
$u=v=0$ and \eqref{eq:degree-separation} below is trivial. Otherwise,
\Cref{lem:fraction} with $X=a-bZ$, $\sigma^2=c_pL_0^2$, $M=\sqrt2L_0$, and
$\lambda=1/\sqrt2$ shows that the set
\[
 \mathcal G=\bigl\{\alpha\in I_m:|u/2-|\alpha|v|\ge c_0L_0\bigr\},
 \qquad c_0=(c_p/2)^{1/2},
\]
contains at least $(c_p/4)q_m\ge(c_p/4)m^p/p!$ indices.

\textbf{Lower bound for the sum.} For all real $z$,
\begin{equation*}
 \log\cosh(1)\min(z^2,|z|)\le\log\cosh z\le\min(z^2/2,|z|).
\end{equation*}
The upper bound is elementary. For the lower bound, the function
$g(z)=2\log\cosh z-z\tanh z$ satisfies $g(0)=0$ and
$g'(z)=\operatorname{sech}^2z\,(\tfrac12\sinh2z-z)\ge0$ for $z\ge0$, so
$(\log\cosh z)/z^2$ is nonincreasing on $(0,\infty)$; this gives the
bound for $|z|\le1$. For $|z|\ge1$, it follows from convexity of
$\log\cosh$ and $\log\cosh0=0$. For $\alpha\in\mathcal G$, since
$\log\cosh$ is increasing in $|z|$ and $c_0\le1$,
\[
 \log\cosh(u/2-|\alpha|v)\ge\log\cosh(c_0L_0)
 \ge\log\cosh(1)\,c_0^2\min(L_0^2,L_0).
\]
Since $L_0^2\ge(u^2+m^2v^2)/4$, we have
$\min(L_0^2,L_0)\ge\frac14\min\{u^2+m^2v^2,(u^2+m^2v^2)^{1/2}\}$.
Discarding the nonnegative terms with $\alpha\notin\mathcal G$ and using
$|\mathcal G|\ge(c_p/4)m^p/p!$ gives
\begin{equation}\label{eq:degree-separation}
 \sum_{|\alpha|\le m}\log\cosh(u/2-|\alpha|v)
 \ge cm^p\min\bigl\{u^2+m^2v^2,(u^2+m^2v^2)^{1/2}\bigr\}.
\end{equation}

By \Cref{lem:affinity-monotone}, the affinity of the observation laws
is at most the affinity of the laws of the parameter-independent linear
statistic $H_nY_n$. By \Cref{lem:recovery}, the latter log affinity differs from that of the exact
coefficient vector by $o(1)$, uniformly over the parameters. Combining
this fact with \eqref{eq:coefficient-affinity-product},
\eqref{eq:degree-separation}, and $m_n\asymp b_n$ proves
\eqref{eq:affinity-separation}. The coefficient law has no nugget
parameter, and its recovery bound is uniform in $t$; this proves the
claimed uniformity in the nugget as well.

We now prove \eqref{eq:affinity-joint}. On the rectangle, the mean value
theorem gives $|\log(s/s_0)|\ge|s-s_0|/s_+$ and
$|\log(l/l_0)|\ge|l-l_0|/l_+$. Hence, with
$c_1=\min\{1,s_+^{-2},l_+^{-2}\}$,
\[
 \Lambda_n(\theta)^2
 \ge c_1\{(s-s_0)^2+b_n^2(l-l_0)^2\}
 =c_1b_n^{-p}\norm{h_\spat}^2 .
\]
Since $x\mapsto\min(x^2,x)$ is increasing on $[0,\infty)$,
\[
 b_n^p\min\{\Lambda_n(\theta)^2,\Lambda_n(\theta)\}
 \ge\min\bigl\{c_1\norm{h_\spat}^2,c_1^{1/2}b_n^{p/2}\norm{h_\spat}\bigr\}
 \ge c_1\min\bigl(\norm{h_\spat}^2,b_n^{p/2}\norm{h_\spat}\bigr),
\]
and \eqref{eq:affinity-separation} gives
$\Aff(P_{\theta,n},P_{\theta_0,n})
\le\exp\{e_n-cc_1\min(\norm{h_\spat}^2,b_n^{p/2}\norm{h_\spat})\}$.

For the nugget, use the noise block $W_n=V_n^\top Y_n$ of
\eqref{eq:block-notation}. Since $P_nV_n=0$,
$\Cov_\theta(W_n)=t\Id_{d_n}+E_\theta$ with
$E_\theta=V_n^\top S_n(\theta)V_n
=V_n^\top\{S_n(\theta)-P_nS_n(\theta)P_n\}V_n$. By
\eqref{eq:projection-approx}, $0\preceq E_\theta\preceq n^{-20}\Id_{d_n}$
for all $\theta\in\Theta$. Therefore the eigenvalues
$\lambda_1,\ldots,\lambda_{d_n}$ of
$\Cov_{\theta_0}(W_n)^{-1/2}\Cov_\theta(W_n)\Cov_{\theta_0}(W_n)^{-1/2}$
satisfy $|\log\lambda_k-\log(t/t_0)|\le n^{-20}/t_-$. By
\eqref{eq:gaussian-affinity},
\[
 -\log\Aff\{\Cov_\theta(W_n),\Cov_{\theta_0}(W_n)\}
 =\frac12\sum_{k=1}^{d_n}\log\cosh\Bigl(\frac12\log\lambda_k\Bigr).
\]
The function $y\mapsto\log\cosh(y/2)$ is $\tfrac12$-Lipschitz, and since
its second derivative is $\frac14\operatorname{sech}^2(y/2)$, it is at
least $c_2y^2$ for $|y|\le\log(t_+/t_-)$, where
$c_2=\frac18\operatorname{sech}^2\{\frac12\log(t_+/t_-)\}$. Using
$|\log(t/t_0)|\ge|t-t_0|/t_+$, $d_n\le n$, and $d_n\ge n/2$ for large $n$,
\[
 -\log\Aff\{\Cov_\theta(W_n),\Cov_{\theta_0}(W_n)\}
 \ge\frac{c_2}{4t_+^2}\,n(t-t_0)^2-\frac{n^{-19}}{t_-}
 =\frac{c_2}{4t_+^2}\,h_t^2-\frac{n^{-19}}{t_-}.
\]
The statistic $W_n=V_n^\top Y_n$ is linear and parameter independent, so
\Cref{lem:affinity-monotone} shows that the affinity of the observation
laws is at most this affinity. Since
$\min(e^{-x},e^{-y})\le e^{-(x+y)/2}$, the spatial and nugget bounds
together prove \eqref{eq:affinity-joint}, with $c$ replaced by
$\frac12\min\{cc_1,c_2/(4t_+^2)\}$ and $e_n$ by $\max(e_n,n^{-19}/t_-)$.
The design enters only through \Cref{lem:projection,lem:recovery}. Their
constants depend on $\calX$, $Q$, $C_h$, and $\kappa$, and their proofs use
\eqref{eq:coverage} only at the current $n$. No constant depends on
$\theta_0$: the bounds of \Cref{lem:recovery} hold uniformly over pairs
$\theta,\theta'\in\Theta$, and the remaining constants depend only on
$p$ and the endpoints of $\Theta$.
\end{proof}

\subsection{\texorpdfstring{Proof of \Cref{lem:lr-process}}{Proof of the likelihood-ratio process lemma}}\label{app:lr-process}

\begin{proof}
The matrix $C_n(\theta)$ is a smooth positive definite function of
$\theta$ on a neighborhood of $\Theta$, so $l_n$ is continuously
differentiable there and $\xi_n$ has continuous paths. Write
$\theta_h=\theta_0+D_n^{-1}h$ and
$\nabla_hl_n(h)=D_n^{-1}(\nabla l_n)(\theta_h)$. Since $Z_n(h)$ is
the density of $P_{\theta_h,n}$ with respect to $P_{\theta_0,n}$,
\[
 \E_{\theta_0,n}\xi_n(h)^m=\E_{\theta_0,n}Z_n(h)^{1/2}
 =\int\sqrt{dP_{\theta_h,n}\,dP_{\theta_0,n}},
\]
which is \eqref{eq:lr-moment}.

For \eqref{eq:lr-lipschitz}, let $u,v\in U_n$ and $w_r=v+r(u-v)$ for
$r\in[0,1]$; these points lie in the box $U_n$. Since
$\frac{d}{dr}\xi_n(w_r)=(2m)^{-1}\xi_n(w_r)\nabla_hl_n(w_r)^\T(u-v)$,
Minkowski's integral inequality gives
\[
 \norm{\xi_n(u)-\xi_n(v)}_{L^m}
 \le\frac{\norm{u-v}}{2m}\sup_{w\in U_n}
 \bigl(\E_{\theta_0,n}[Z_n(w)^{1/2}\norm{\nabla_hl_n(w)}^m]\bigr)^{1/m}.
\]
Fix $w$ and abbreviate $\theta=\theta_w$. Changing measure and applying
the Cauchy--Schwarz inequality under $P_{\theta,n}$,
\[
\begin{aligned}
 \E_{\theta_0,n}\bigl[Z_n(w)^{1/2}\norm{\nabla_hl_n(w)}^m\bigr]
 &=\E_{\theta,n}\bigl[Z_n(w)^{-1/2}\norm{\nabla_hl_n(w)}^m\bigr]\\
 &\le\bigl(\E_{\theta,n}Z_n(w)^{-1}\bigr)^{1/2}
 \bigl(\E_{\theta,n}\norm{\nabla_hl_n(w)}^{2m}\bigr)^{1/2},
\end{aligned}
\]
and $\E_{\theta,n}Z_n(w)^{-1}=\E_{\theta_0,n}1=1$. Under $P_{\theta,n}$,
$X=C_n(\theta)^{-1/2}Y_n\sim N_n(0,\Id_n)$, and the coordinates of
$\nabla_hl_n(w)$ are $\frac12\{X^\top B_a(\theta)X-\tr B_a(\theta)\}$
for $a\in\{s,l,t\}$, where
$B_a(\theta)=D_{n,aa}^{-1}C_n(\theta)^{-1/2}C_{n,a}(\theta)C_n(\theta)^{-1/2}$.
Dividing \eqref{eq:raw-frobenius} by $D_{n,aa}$ gives
$\norm{B_a(\theta)}_\Frob\le C$ for $a\in\{s,l\}$, uniformly over
$\theta\in\Theta$ and $n$, and
$\norm{B_t(\theta)}_\Frob=n^{-1/2}\norm{C_n(\theta)^{-1}}_\Frob\le t_-^{-1}$.
\Cref{lem:qf-moments} therefore gives
$\E_{\theta,n}\norm{\nabla_hl_n(w)}^{2m}\le C$, uniformly in $w$ and
$n$, which proves \eqref{eq:lr-lipschitz}.
\end{proof}

\subsection{\texorpdfstring{Proof of \Cref{lem:tail}}{Proof of the tail lemma}}\label{app:tail}

\begin{proof}
Take $m=4$, so that $m>k=3$, and let $\xi_n$ be the process of
\Cref{lem:lr-process} on the box $U_n=D_n(\Theta-\theta_0)$. Write
$u=(u_1,u_2)$, where $u_1\in\R^2$ holds the spatial coordinates and
$u_2\in\R$ the nugget coordinate. Let $n$ be large enough that $b_n\ge1$. By
\eqref{eq:lr-moment} and \eqref{eq:affinity-joint}, the envelope
\eqref{eq:envelope-moment} holds with $A=\exp(\sup_ne_n)$, with $a$ equal
to the constant $c$ of \eqref{eq:affinity-joint}, and with
$\gamma_1=b_n^{p/2}$ and $\gamma_2=\infty$. By \eqref{eq:lr-lipschitz},
\eqref{eq:envelope-lipschitz} holds with $B=C_L$. These constants do not
depend on $n$.

Since $\thetan$ maximizes $l_n$ over $\Theta$,
$\xi_n(\widehat h_n)\ge\xi_n(0)=1$. Hence Markov's inequality and
\Cref{cor:envelope} with $(H_1,H_2)=(H,0)$ give
\[
\begin{aligned}
 \PP\{\norm{\widehat h_{n,\spat}}\ge H\}
 &\le\PP\Bigl\{\sup_{u\in U_n,\,|u_1|\ge H}\xi_n(u)\ge1\Bigr\}
 \le\E\sup_{u\in U_n,\,|u_1|\ge H}\xi_n(u)^m\\
 &\le C\exp\{-a''\min(H^2,b_n^{p/2}H)\},
\end{aligned}
\]
with $a''=a/128$, since $\psi_{\gamma_2}(0)=0$. With $(H_1,H_2)=(0,H)$,
the same argument gives
$\PP\{|\widehat h_{n,t}|\ge H\}\le C\exp(-a''H^2)$.
This proves \eqref{eq:tail-bound}. The constants $A$, $a$, $B$, and
$\gamma_1$ do not depend on $\theta_0$, and neither do the constants of
\Cref{cor:envelope}, which do not depend on the box $U_n$. The design
enters only through \eqref{eq:affinity-joint} and \Cref{lem:lr-process}.
This gives the stated dependence of the constants.
\end{proof}

\subsection{\texorpdfstring{Proof of \Cref{lem:information}}{Proof of the information bounds lemma}}\label{app:information}

\begin{proof}
The spatial upper bounds follow from
\eqref{eq:fisher} and the first derivative Frobenius bounds in
\Cref{lem:projection}. The nugget upper bound follows from
$C_n^{-1}\preceq t_-^{-1}\Id_n$.

\textbf{Spatial lower bound.}
Write $D_{n,\spat}=\diag(b_n^{p/2},b_n^{p/2+1})$.
For $\norm{h}\le1$ in $\R^2$, set
$\theta_h=((s_0,l_0)+D_{n,\spat}^{-1}h,t_0)$ and
\[
 E_n(h)=C_{n,0}^{-1/2}
       \{C_n(\theta_h)-C_{n,0}\}C_{n,0}^{-1/2}.
\]
For $a\in\{s,l\}$, let
$B_{n,a}=D_{n,aa}^{-1}C_{n,0}^{-1/2}C_{n,a}(\theta_0)C_{n,0}^{-1/2}$ be the
normalized whitened first derivatives at $\theta_0$.

\textbf{Covariance comparability.}
Fix $h$ with $\norm h\le1$, put $C(u)=C_n(\theta_{uh})$ for $u\in[0,1]$,
and fix $x\in\R^n\setminus\{0\}$. The function $f(u)=x^\top C(u)x$ is
positive, and
\[
 f'(u)=\sum_{a\in\{s,l\}}h_aD_{n,aa}^{-1}\,x^\top C_{n,a}(\theta_{uh})x
 =\sum_{a\in\{s,l\}}h_a\,y^\top\widetilde B_a(u)y,
\]
where $y=C(u)^{1/2}x$ and
$\widetilde B_a(u)=D_{n,aa}^{-1}C(u)^{-1/2}C_{n,a}(\theta_{uh})C(u)^{-1/2}$.
Here the whitening is at the same parameter as the derivative, so
dividing \eqref{eq:raw-operator} by $D_{n,aa}$ gives
$\norm{\widetilde B_a(u)}_{\op}\le C\epsilon_n$, uniformly in $u$ and $h$.
Since $|h_s|+|h_l|\le\sqrt2$ and $\norm y^2=f(u)$,
\[
 |f'(u)|\le\sqrt2\,C\epsilon_n f(u),
 \qquad\text{that is,}\qquad
 \Bigl|\frac{d}{du}\log f(u)\Bigr|\le\sqrt2\,C\epsilon_n.
\]
Integrating over $[0,u]$ gives
$e^{-\sqrt2C\epsilon_n}f(0)\le f(u)\le e^{\sqrt2C\epsilon_n}f(0)$. Since $x$ was
arbitrary, after renaming the constant,
\begin{equation}\label{eq:info-comparability}
 e^{-C\epsilon_n}C_{n,0}\preceq C_n(\theta_{uh})\preceq e^{C\epsilon_n}C_{n,0},
 \qquad u\in[0,1],\ \norm h\le1.
\end{equation}

\textbf{Change of whitening.}
Let $K(u)=C_{n,0}^{-1/2}C(u)^{1/2}$. Then
$\norm{K(u)}_{\op}^2=\norm{C_{n,0}^{-1/2}C(u)C_{n,0}^{-1/2}}_{\op}$, which
is at most $e^{C\epsilon_n}$ by \eqref{eq:info-comparability}. For any
$n\times n$ matrix $A$,
\[
 C_{n,0}^{-1/2}AC_{n,0}^{-1/2}=K(u)\{C(u)^{-1/2}AC(u)^{-1/2}\}K(u)^\T,
\]
so
\begin{equation}\label{eq:rewhitening}
 \norm{C_{n,0}^{-1/2}AC_{n,0}^{-1/2}}_{\Frob}
 \le e^{C\epsilon_n}\norm{C(u)^{-1/2}AC(u)^{-1/2}}_{\Frob}.
\end{equation}

\textbf{Taylor expansion.}
Taylor's formula with integral remainder, applied to $u\mapsto C(u)$,
gives $E_n(h)=h_sB_{n,s}+h_lB_{n,l}+R_n(h)$ with
\[
 R_n(h)=\int_0^1(1-u)\,C_{n,0}^{-1/2}C''(u)\,C_{n,0}^{-1/2}\,du,
 \qquad
 C''(u)=\sum_{a,b\in\{s,l\}}h_ah_bD_{n,aa}^{-1}D_{n,bb}^{-1}
 C_{n,ab}(\theta_{uh}).
\]
By \Cref{lem:projection}, the normalized second derivatives whitened at
$\theta_{uh}$ have Frobenius norm $O(\epsilon_n)$, so
$\norm{C(u)^{-1/2}C''(u)C(u)^{-1/2}}_{\Frob}=O(\epsilon_n)$. By
\eqref{eq:rewhitening}, the same holds after whitening at $C_{n,0}$, and
therefore
\begin{equation}\label{eq:relative-increment}
 E_n(h)=h_sB_{n,s}+h_lB_{n,l}+R_n(h),\quad
 \sup_{\norm{h}\le1}\norm{R_n(h)}_{\Frob}=O(\epsilon_n).
\end{equation}
Moreover, $\norm{B_{n,a}}_\Frob\le C$ by \Cref{lem:projection}, so
$\norm{E_n(h)}_{\Frob}=O(1)$, and \eqref{eq:info-comparability} with
$u=1$ places every eigenvalue of $E_n(h)$ in
$[e^{-C\epsilon_n}-1,e^{C\epsilon_n}-1]$, so
$\norm{E_n(h)}_{\op}=O(\epsilon_n)$, uniformly on $\norm h\le1$.

For any symmetric $E$ with small operator norm,
\[
 \frac12\log\det(\Id+E/2)-\frac14\log\det(\Id+E)
 =\frac1{16}\tr(E^2)
  +O(\norm{E}_{\op}\norm{E}_{\Frob}^2).
\]
By \eqref{eq:gaussian-affinity} and \eqref{eq:relative-increment},
\begin{equation}\label{eq:local-affinity-expansion}
 -\log\Aff(P_{\theta_h,n},P_{\theta_0,n})
 =\frac18h^\top J_{n,\spat}h+o(1)
\end{equation}
uniformly on $\norm{h}\le1$.

For $\norm{h}=1$, the point $\theta_h$ has rescaled coordinates
$(h,0)$, with spatial part $h$ and nugget part $0$. Hence
\eqref{eq:affinity-joint} gives, once $b_n\ge1$,
\[
 -\log\Aff(P_{\theta_h,n},P_{\theta_0,n})
 \ge c\min(1,b_n^{p/2})-e_n=c-e_n.
\]
Comparing with \eqref{eq:local-affinity-expansion}, we obtain
$\frac18h^\top J_{n,\spat}h\ge c-o(1)$ uniformly on the unit circle, so
$J_{n,\spat}\succeq4c\Id_2$ for all sufficiently large $n$. This argument uses a finite local
shift and does not differentiate the coefficient recovery error.

\textbf{Nugget block and cross information.}
By \Cref{lem:projection}, the positive signal covariance $S_n(\theta_0)$
has at most $r_n=O(b_n^p)$ eigenvalues exceeding $n^{-20}$.
Therefore
\[
 \frac{\calI_{tt,n}(\theta_0)}n
 =\frac1{2n}\tr(C_{n,0}^{-2})
 \longrightarrow\frac1{2t_0^2}.
\]
For $a=s,l$, approximate the normalized whitened derivative $B_{n,a}$
by a matrix $B_{n,a}^{(r)}$ of rank at most $r_n$, using the same
projection. Its Frobenius norm is bounded and the discarded part is
negligible in all traces below. The normalized nugget derivative is
$B_{n,t}=C_{n,0}^{-1}/\sqrt n$, so
\[
 |J_{n,at}|
 \le\frac12\norm{B_{n,t}}_{\op}
      \norm{B_{n,a}^{(r)}}_*
      +o(1)
 \le C\sqrt{\frac{r_n}{n}}+o(1)=o(1),
\]
where $\norm{\cdot}_*$ is the sum of singular values and
$\norm{B^{(r)}}_*\le\sqrt{r_n}\norm{B^{(r)}}_{\Frob}$.
The spatial lower bound and the nugget limit establish
\eqref{eq:information-bounds} and \eqref{eq:information-block}.
\end{proof}

\subsection{\texorpdfstring{Proof of \Cref{lem:quadratic}}{Proof of the local expansion lemma}}\label{app:quadratic}

\begin{proof}
Let $X_n=C_{n,0}^{-1/2}Y_n\sim N_n(0,\Id_n)$.
Define, for $a,b\in\{s,l,t\}$,
\[
 B_a=\frac{C_{n,0}^{-1/2}C_{n,a}(\theta_0)C_{n,0}^{-1/2}}{D_{n,aa}},
 \quad
 B_{ab}=
 \frac{C_{n,0}^{-1/2}C_{n,ab}(\theta_0)C_{n,0}^{-1/2}}
 {D_{n,aa}D_{n,bb}}.
\]
The Gaussian score is
\[
 \Delta_{n,a}=\frac12\{X_n^\top B_aX_n-\tr B_a\}.
\]
It has mean zero and covariance
$J_{n,ab}=\frac12\tr(B_aB_b)$.

Directly differentiating the likelihood gives
\begin{equation}\label{eq:centered-hessian}
 [\nabla_h^2L_n(0)+J_n]_{ab}
 =\frac12\{X_n^\top M_{ab}X_n-\tr M_{ab}\},
 \quad M_{ab}=B_{ab}-B_aB_b-B_bB_a.
\end{equation}
\Cref{lem:projection} gives
$\norm{B_a}_{\op}\le C\epsilon_n$,
$\norm{B_a}_{\Frob}\le C$, and
$\norm{B_{ab}}_{\Frob}\le C\epsilon_n$; the nugget's operator
bound is smaller, and second derivatives involving it are zero.
Thus $\norm{M_{ab}}_{\Frob}\le C\epsilon_n$.
For a symmetric matrix $M$,
$\Var(X_n^\top MX_n)=2\norm{M}_{\Frob}^2$.
Equation \eqref{eq:centered-hessian} proves \eqref{eq:hessian-local}.

\textbf{Uniform third derivatives in the signal block.}
Use the fixed subspace and covariance $A_\theta$ from
\eqref{eq:block-notation}. Compression preserves the relative derivative
bounds in \Cref{lem:projection}. In particular, every nonzero covariance
derivative of total order $k\le3$ in $h$ has relative operator norm
at most $C\epsilon_n^k$. A first nugget derivative is bounded by
$Cn^{-1/2}\le C\epsilon_n$, and higher covariance derivatives involving
the nugget vanish.

For $\norm{h}\le H_n$, integration of the first
relative derivative bound along the segment from $0$ to $h$, exactly as
in the proof of \eqref{eq:info-comparability}, gives
\begin{equation*}
 e^{-C\epsilon_nH_n}A_{\theta_0}
 \preceq A_{\theta_0+D_n^{-1}h}
 \preceq e^{C\epsilon_nH_n}A_{\theta_0}.
\end{equation*}
Since $\epsilon_nH_n\to0$ and
$Z_n^\top A_{\theta_0}^{-1}Z_n\sim\chi_{r_n}^2$,
\begin{equation*}
 \sup_{\norm{h}\le H_n}
 Z_n^\top A_{\theta_0+D_n^{-1}h}^{-1}Z_n=O_{\PP}(r_n).
\end{equation*}
This bounds the entire neighborhood using the same random quadratic form.

Every third derivative of
$-\frac12\log\det A_\theta-\frac12Z_n^\top A_\theta^{-1}Z_n$
is a finite sum of traces or quadratic forms formed from products of
whitened covariance derivatives of total order three. Each product has
operator norm at most $C\epsilon_n^3$. The traces are bounded by
$Cr_n\epsilon_n^3$, and the quadratic forms by
$C\epsilon_n^3Z_n^\top A_\theta^{-1}Z_n$. Therefore,
\[
 \sup_{\norm{h}\le H_n}\max_{a,b,c}
 |\partial_{h_ah_bh_c}^3\widetilde L_{n,\mathrm{signal}}(h)|
 \le C\epsilon_n^3
 \left(r_n+\sup_{\norm{h}\le H_n}Z_n^\top A_\theta^{-1}Z_n\right)
 =O_{\PP}(\epsilon_n),
\]
because $r_n=O(b_n^p)=O(\epsilon_n^{-2})$.

In the complementary block, the log likelihood is
$-\frac12d_n\log t-\norm{W_n}^2/(2t)$. Only derivatives in $h_t$
are nonzero. Its third derivative is bounded by
$Cn^{-3/2}(d_n+\norm{W_n}^2)=O_{\PP}(n^{-1/2})$ uniformly on the
local ball. The exact likelihood derivative differs by
$O_{\PP}(n^{-16})$ from this reduced likelihood derivative, by
\eqref{eq:likelihood-approx}; normalization only decreases that error.
This proves \eqref{eq:third-local}.

The interior assumption and
$H_n\norm{D_n^{-1}}_{\op}=\epsilon_nH_n\to0$ ensure admissibility of the
whole ball. Since $\nabla_hL_n(0)=\Delta_n$, Taylor's formula gives, for
$\norm h\le H_n$,
\[
 L_n(h)-L_n(0)-h^\top\Delta_n+\frac12h^\top J_nh
 =\frac12h^\top\{\nabla_h^2L_n(0)+J_n\}h
 +\frac16\sum_{a,b,c}\partial^3_{h_ah_bh_c}L_n(\bar h)\,h_ah_bh_c
\]
for some $\bar h$ on the segment from $0$ to $h$. By
\eqref{eq:hessian-local}, \eqref{eq:third-local}, and
$\sum_a|h_a|\le\sqrt3\norm h$, the right side is
$O_{\PP}(\epsilon_nH_n^2+\epsilon_nH_n^3)=O_{\PP}(\epsilon_nH_n^3)$ uniformly
on the ball, because $H_n\ge1$. This proves \eqref{eq:quadratic-remainder},
and \eqref{eq:uniform-quadratic} follows.

\textbf{Sharpness (\Cref{rem:radius}).}
Along the variance direction $h=ue_s$, the covariance is
$C(\sigma)=\sigma R_n(l_0)+t_0\Id_n$ with $\sigma=s_0+b_n^{-p/2}u$, so
$\partial_\sigma C=R_n(l_0)$ and $\partial_\sigma^2C=0$. Differentiating
$l=-\frac12\log\det C-\frac12Y_n^\top C^{-1}Y_n$ gives $\partial_\sigma^3l = -\tr\{(C^{-1}R_n)^3\}+3Y_n^\top C^{-1}(R_nC^{-1})^3Y_n$ and $\partial_\sigma^4l = 3\tr\{(C^{-1}R_n)^4\}-12Y_n^\top C^{-1}(R_nC^{-1})^4Y_n$,
%\begin{align*}
% \partial_\sigma^3l &= -\tr\{(C^{-1}R_n)^3\}+3Y_n^\top C^{-1}(R_nC^{-1})^3Y_n%,
% \quad \text{and}\quad 
% \partial_\sigma^4l = 3\tr\{(C^{-1}R_n)^4\}-12Y_n^\top C^{-1}(R_nC^{-1})^4Y_n,
%\end{align*}
with $R_n=R_n(l_0)$. At $u=0$, in terms of $X_n=C_{n,0}^{-1/2}Y_n$ and
the normalized derivative $B_s$,
$\partial_{h_s}^3L_n(0)=-\tr(B_s^3)+3X_n^\top B_s^3X_n$. Its mean is
$2\tr(B_s^3)$, and its standard deviation is
$3\sqrt2\norm{B_s^3}_\Frob\le C\norm{B_s}_{\op}^2\norm{B_s}_\Frob
\le C\epsilon_n^2$ by \Cref{lem:projection}. The matrix $B_s$ is positive
semidefinite. By the Cauchy--Schwarz inequality over its eigenvalues,
$(\tr B_s^2)^2\le\tr(B_s)\tr(B_s^3)$. Here $\tr B_s^2=2J_{n,ss}\ge2c$ by
\Cref{lem:information}. Moreover, $B_s=b_n^{-p/2}R_n(s_0R_n+t_0\Id_n)^{-1}$
has eigenvalues $b_n^{-p/2}\mu/(s_0\mu+t_0)$, where $\mu$ ranges over the
eigenvalues of $R_n$. By \eqref{eq:projection-approx} and Weyl's
inequality, at most $r_n$ of these $\mu$ exceed $n^{-20}/s_0$, so
$\tr B_s\le b_n^{-p/2}(r_n/s_0+n^{-19}/(s_0t_0))\le C\epsilon_n^{-1}$.
Hence $\tr(B_s^3)\ge c\epsilon_n$, and
$\partial_{h_s}^3L_n(0)\ge c\epsilon_n$ with probability tending to one,
after decreasing $c$.

For the fourth derivative, all matrices involved commute with $R_n$. In
an eigenbasis of $R_n$, with $Y_n$ having independent coordinates of
variances $s_0\mu_i+t_0$, each eigenvalue of $C^{-1}R_n$ is at most
$C\min(1,\mu_i)$ for $|\sigma-s_0|\le s_0/2$. Therefore
$\sup_{|\sigma-s_0|\le s_0/2}|\partial_\sigma^4l|$ is bounded by a
constant times $\sum_i\min(1,\mu_i^4)(1+\chi_i^2)$ with $\chi_i$
independent standard normal, which is $O_{\PP}(r_n)$ by the eigenvalue
count above. After normalization, the fourth derivative in $h_s$ is
$O_{\PP}(b_n^{-2p}r_n)=O_{\PP}(\epsilon_n^2)$ uniformly on
$|u|\le\kappa\epsilon_n^{-1/3}$.

Now let $u=\kappa\epsilon_n^{-1/3}$, so that $\sigma-s_0=\kappa\epsilon_n^{2/3}\to0$.
Taylor's formula in $u$, together with \eqref{eq:hessian-local}, gives
\[
 L_n(ue_s)-L_n(0)-u\Delta_{n,s}+\frac12J_{n,ss}u^2
 =O_{\PP}(\epsilon_nu^2)+\frac16\partial_{h_s}^3L_n(0)\,u^3
 +O_{\PP}(\epsilon_n^2u^4).
\]
The first and last terms are $O_{\PP}(\kappa^2\epsilon_n^{1/3})$ and
$O_{\PP}(\kappa^4\epsilon_n^{2/3})$, and the middle term is at least
$c\kappa^3/6$ with probability tending to one. This proves
\Cref{rem:radius}.
\end{proof}

\subsection{\texorpdfstring{Proof of \Cref{lem:score}}{Proof of the score CLT lemma}}\label{app:score}

\begin{proof}
Use $X_n$ and $B_a$ from the preceding proof. For a fixed
$v\in\R^3$, write
\[
 v^\top M_n^{-\top}\nabla l_n(\theta_0)
 =\frac12\{X_n^\top B_n(v)X_n-\tr B_n(v)\},
\]
where
\[
 B_n(v)=\sum_a[D_nM_n^{-1}v]_a B_a.
\]
Since $M_n^{-1}M_n^{-\top}=(M_n^\top M_n)^{-1}=\calI_n(\theta_0)^{-1}$ and
$\calI_n(\theta_0)=D_nJ_nD_n$,
\[
 (D_nM_n^{-1})(D_nM_n^{-1})^\T=D_n\calI_n(\theta_0)^{-1}D_n=J_n^{-1},
 \qquad
 \norm{D_nM_n^{-1}}_{\op}^2=\norm{J_n^{-1}}_{\op}\le C
\]
by \Cref{lem:information}, so the coefficient vector is bounded. Hence
$\norm{B_n(v)}_{\op}\le C_v\epsilon_n\to0$.
The variance is exactly
$\frac12\tr(B_n(v)^2)
=v^\top M_n^{-\top}\calI_n(\theta_0)M_n^{-1}v=\norm{v}^2$.

For $v\ne0$, diagonalize the symmetric matrix $B_n(v)$, with
eigenvalues $\lambda_{j,n}$. The quadratic form is distributed as
$\frac12\sum_j\lambda_{j,n}(\xi_j^2-1)$ for independent standard
normal variables. Its Lyapunov ratio of order four is bounded by a
constant times
\[
 \frac{\sum_j\lambda_{j,n}^4}
      {(\sum_j\lambda_{j,n}^2)^2}
 \le
 \frac{\max_j\lambda_{j,n}^2}{\sum_j\lambda_{j,n}^2}
 \longrightarrow0.
\]
The univariate central limit theorem and the Cram\'er--Wold device
give \eqref{eq:score-clt}. The same calculation applies to any bounded
sequence of coefficient vectors for $\Delta_n$ whose variances
converge to a positive limit.
\end{proof}

\subsection{\texorpdfstring{Proof of \Cref{lem:random}}{Proof of the random coverage lemma}}\label{app:random}

\begin{proof}
Let $Q$ have side length $L$, and let $A\ge1$ be a fixed constant. For
$n\ge3$, put
\[
 k_n=\left\lfloor\left(\frac{n}{A\log n}\right)^{1/p}\right\rfloor,
\]
and suppose $k_n\ge1$. Partition $Q$ into $k_n^p$ closed subcubes
$Q_1,\ldots,Q_{k_n^p}$ of side length $L/k_n$, which overlap only on
their boundaries. Each has volume $L^p/k_n^p$, so its sampling
probability satisfies
\[
 \pi_j=\PP(X_1\in Q_j)=\int_{Q_j}f(x)\,dx\ge\frac{f_-L^p}{k_n^p}.
\]
Since $X_1,\ldots,X_n$ are independent and $1-x\le e^{-x}$,
\[
 \PP(\text{no }X_i\text{ lies in }Q_j)=(1-\pi_j)^n
 \le\Bigl(1-\frac{f_-L^p}{k_n^p}\Bigr)^n
 \le\exp\Bigl(-\frac{nf_-L^p}{k_n^p}\Bigr).
\]
A union bound over the $k_n^p$ subcubes gives
\[
 \PP(\text{some }Q_j\text{ is empty})
 \le k_n^p\exp\Bigl(-\frac{nf_-L^p}{k_n^p}\Bigr).
\]
By the definition of $k_n$, $k_n^p\le n/(A\log n)$. Hence
$k_n^p\le n$, because $A\log n\ge1$, and $n/k_n^p\ge A\log n$, so that
$\exp(-nf_-L^p/k_n^p)\le\exp(-f_-L^pA\log n)=n^{-f_-L^pA}$. Therefore
\[
 \PP(\text{some }Q_j\text{ is empty})
 \le n\exp(-f_-L^pA\log n)=n^{1-f_-L^pA}.
\]
Given $q>0$, choose $A\ge\max\{1,(q+1)/(f_-L^p)\}$, so that this
probability is at most $n^{-q}$. When every subcube contains an
observation, every $x\in Q$ lies in some $Q_j$ together with an
observation, so it is within the diameter $\sqrt p\,L/k_n$ of that
observation. Since $\lfloor y\rfloor\ge y/2$ for $y\ge1$,
\[
 h_n(Q)\le\frac{\sqrt p\,L}{k_n}\le2\sqrt p\,L\Bigl(\frac{A\log n}n\Bigr)^{1/p}
\]
on this event, for all $n\ge3$ with $k_n\ge1$. This proves the
probability bound with $A_q=2\sqrt p\,LA^{1/p}$, after enlarging $C_q$ to
cover the finitely many $n$ with $n<3$ or $k_n=0$. With $q=2$, the first Borel--Cantelli lemma proves
\eqref{eq:random-fill}. Since
$(\log n/n)^{1/p}=o(n^{-\kappa})$ for $\kappa<1/p$,
the coverage conclusion follows.
\end{proof}

\section{Additional simulation details}\label{apdx:sim}

This appendix provides additional computational details and finite-sample
results for the simulations in \Cref{sec:simulation}. The code, the simulation output, and a Jupyter notebook that performs the
checks in Appendix~\ref{apdx:sim-validation} and reproduces the figures,
\Cref{tab:sim-slopes}, and the other numerical results in
\Cref{sec:simulation} and this appendix are available at \url{https://github.com/Ameer-eng/Simulation-Studies-for-Euclidean-RBF-MLE}.

\subsection{Exact evaluation of the likelihood}\label{apdx:sim-lik}

The sample sizes considered in \Cref{sec:simulation} are too large for
direct $O(n^3)$ evaluation of the Gaussian likelihood. We therefore exploit
the analytic structure of the RBF kernel and, for $p\ge2$, the Kronecker
structure of the regular grid. These calculations evaluate the same
Gaussian likelihood in \eqref{eq:likelihood}, up to numerical truncation
at machine precision.

\textbf{Feature representation.}
On $[0,1]$, let $u=(x-\tfrac12)/l$. We use the expansion
\[
 \exp\left\{-\frac{(x-x')^2}{2l^2}\right\}
 =
 \sum_{k=0}^\infty \phi_k(x)\phi_k(x'),
 \qquad
 \phi_k(x)=e^{-u^2/2}\frac{u^k}{\sqrt{k!}}.
\]
We truncate the expansion after
\[
 M(l)=\left\lceil \bar u+10\sqrt{\bar u}+25\right\rceil,
 \qquad
 \bar u=(2l)^{-2}.
\]
Over the parameter range $l\ge0.05$, the reconstructed kernel and its
lengthscale derivative agree with direct evaluation to about $10^{-14}$
and $10^{-12}$, respectively, in maximum norm. The features are evaluated recursively using
$\phi_k=\phi_{k-1}u/\sqrt{k}$ to avoid numerical over- and underflow.

Let $\Phi$ denote the $n\times M$ feature matrix. Then
\[
 R_n(l)\approx\Phi\Phi^\top,
\]
with the approximation at machine precision under the truncation above.

\textbf{Spectral form of the likelihood.}
For fixed $l$, let
\[
 R_n(l)=U\Lambda U^\top,
 \qquad
 \Lambda=\diag(\lambda_1,\ldots,\lambda_r),
\]
where eigenvalues smaller than $10^{-12}\lambda_{\max}$ are discarded.
Writing $z=U^\top Y$, the Gaussian likelihood can be evaluated as
\[
 -2l_n(\theta)
 =
 \sum_{k=1}^r
 \left\{
 \log(\sigma^2\lambda_k+\tau^2)
 +\frac{z_k^2}{\sigma^2\lambda_k+\tau^2}
 \right\}
 +(n-r)\log\tau^2
 +\frac{\norm{Y}^2-\norm z^2}{\tau^2},
\]
up to an additive constant. The quantities $\Lambda$ and $z$ are obtained
from the much smaller Gram matrix $\Phi^\top\Phi$ and the vector
$\Phi^\top Y$.

\textbf{Linear-cost Gram matrix.}
The feature products satisfy
\[
 \phi_j(x)\phi_k(x)
 =
 e^{-u^2}\frac{u^{j+k}}{\sqrt{j!\,k!}},
\]
so $\Phi^\top\Phi$ is a scaled Hankel matrix:
\[
 [\Phi^\top\Phi]_{jk}
 =
 h_{j+k}\binom{j+k}{j}^{1/2},
 \qquad
 h_m=
 \sum_{i=1}^n
 e^{-u_i^2}\frac{u_i^m}{\sqrt{m!}}.
\]
Thus, forming the Gram matrix requires $O(nM)$ rather than $O(nM^2)$
operations. The midpoint grid is symmetric around $1/2$, so the odd
moments vanish. This symmetry further reduces the computation without
changing the likelihood.

\textbf{Kronecker structure for $p\ge2$.}
For the regular grid used in the simulations,
\[
 R_n(l)=R_G(l)^{\otimes p}.
\]
Consequently, its eigenvalues are products of the one-dimensional
eigenvalues,
\[
 \lambda_{k_1}\cdots\lambda_{k_p},
\]
and projections of the observations are obtained by multiplying the data
array by the one-dimensional eigenvector matrix along each axis. Thus,
only one-dimensional quantities involving $G=n^{1/p}$ grid points need
to be explicitly constructed.

\subsection{Data generation}\label{apdx:sim-generation}

For $p=1$, data are generated as
\[
 Y=\sigma_0\Phi\xi+\tau_0\varepsilon,
\]
where $\xi\sim N_M(0,I_M)$ and $\varepsilon\sim N_n(0,I_n)$ are
independent and $\Phi$ is evaluated at $l_0$. Therefore,
\[
 \Cov(Y)
 =
 \sigma_0^2\Phi\Phi^\top+\tau_0^2I_n
 =
 C_n(\theta_0)
\]
up to the machine-precision truncation described above.

For $p\ge2$, the feature matrix is replaced by its Kronecker product and
$\xi$ by an array of independent standard normal variables. This generates
a new independent realization of both the GP and measurement errors for
each Monte Carlo replicate.

\subsection{Computation of the MLE}\label{apdx:sim-mle}

The MLE is computed over
\[
 \Theta
 =
 [10^{-2},10^2]\times[0.05,2]\times[10^{-4},1]
\]
for $(\sigma^2,l,\tau^2)$.

For a fixed $l$, we optimize over
$(\log\sigma^2,\log\tau^2)$ using L-BFGS-B with analytic gradients,
relative function tolerance $10^{-15}$, and gradient tolerance $10^{-10}$.
Because this two-dimensional problem can have more than one local
maximum, L-BFGS-B is run from two starting values, one with $\sigma^2$ at
its lower bound and one with $\sigma^2$ at a moment estimate based on the
leading eigen-directions of $R_n(l)$, and the better solution is kept.

The profile likelihood in $l$ can also have more than one local maximum.
It is therefore first evaluated at $25$ log-equispaced values of $l$ in
$[0.05,2]$. Every local maximum on this grid is then refined by a bounded
one-dimensional Brent search in $\log l$ between its two neighboring grid
values, with tolerance $10^{-7}$, using the solution at the grid value as
an additional starting value for the inner optimization. The MLE is the
point with the largest likelihood among the grid values and the refined
local maxima.

Among the $2.7\times10^4$ fitted data sets, only three estimates occur on
the boundary of $\Theta$: one estimate of $l$ at its upper bound for
$p=1$, $n=10^2$, and two estimates of $\tau^2$ at its lower bound for
$p=3$, $n=125$.

For the three data sets with the largest $\widehat\sigma^2$ at each
$p=1$ sample size from $10^2$ to $3\times10^4$, we additionally evaluated
the profile likelihood on a fine grid of $400$ lengthscales. In all $18$
cases, no grid value gave a larger likelihood than the computed MLE,
although in four cases the profile likelihood had a second, lower local
maximum. Thus, these extreme estimates arise from the likelihood itself
rather than numerical optimization failure.

\subsection{Computation of the Fisher information}
\label{apdx:sim-information}

For each $n$, the Fisher information $\calI_n(\theta_0)$ in
\eqref{eq:fisher} is computed directly from the observation design.
Let $V$ be an orthonormal basis for the span of $\Phi$ and
$\partial_l\Phi$ at $l_0$. Both $R_n$ and $\dot R_n$ vanish, up to the
machine-precision truncation, on the orthogonal complement of this
space, where $C_n$ acts as $\tau_0^2I$.

The traces in \eqref{eq:fisher} therefore reduce to calculations involving
$V^\top R_nV$ and $V^\top\dot R_nV$, together with the contribution
\[
 \frac{n-\dim(V)}{2\tau_0^4}
\]
to the nugget information. For $p\ge2$, the Kronecker structure again
reduces these calculations to one-dimensional quantities.

Because the resulting computation depends on $n$ primarily through
one-dimensional Gram matrices, we can evaluate the Fisher-information
benchmark beyond the range used for Monte Carlo simulation. We evaluate
it up to $n=10^8$ for $p=1$ and $n=10^{12}$ for $p=2,3$. These extended
calculations are used below to examine how slowly the finite-sample
information approaches the limiting rates in \Cref{thm:rbf}.

\subsection{Monte Carlo summaries}\label{apdx:sim-summaries}

For each parameter, the main text reports the RMSE, the robust standard
deviation $\mathrm{IQR}/1.349$, and the asymptotic standard deviation
obtained from the Fisher information.

For a normal random variable with variance $v$,
\[
 \mathrm{IQR}
 =
 2\Phi^{-1}(0.75)\sqrt v
 \approx1.349\sqrt v.
\]
Thus, $\mathrm{IQR}/1.349$ estimates the standard deviation under
normality. Unlike the RMSE, it depends only on the central half of the
sampling distribution and is therefore less sensitive to a small number
of extreme estimates. Agreement between the RMSE and robust standard
deviation provides an additional indication that the estimation error is
approximately normal.

Bootstrap $95\%$ intervals for the RMSE and robust standard deviation are
computed using $2000$ bootstrap resamples of the $1000$ Monte Carlo
replicates.

\subsection{Additional normality diagnostics}
\label{apdx:sim-normality}

The Q--Q plots in \Cref{fig:sim-p1,fig:sim-p2,fig:sim-p3} provide the
primary graphical assessment of asymptotic normality. Additional numerical
diagnostics lead to the same conclusion.

For $p=3$ and $n=10^6$, the standardized errors of
$\log\widehat\sigma^2$ and $\log\widehat l$ have skewness $-0.05$ and
$0.00$, respectively, and their Kolmogorov--Smirnov distances from
$N(0,1)$ are below $0.02$. Consistent with these diagnostics, the RMSE
and robust standard deviation are nearly identical.

Convergence is slower for $p=1$. The ratio of the RMSE to the
Fisher-information standard deviation remains above one over the simulated
range. For $\log\widehat\sigma^2$, the excess decreases from $18\%$ at
$n=10^2$ to $10\%$ at $n=10^6$; for $\log\widehat l$, it decreases from
$54\%$ to $26\%$. The robust standard deviation is closer to the
Fisher-information benchmark, with deviations between $4\%$ and $20\%$.
Because these ratios continue to decrease with $n$, the Monte Carlo slopes
for $p=1$ over the simulated range are steeper than the corresponding
slopes of the Fisher-information standard deviation.

\subsection{Finite-sample behavior of the nugget}
\label{apdx:sim-nugget}

The particularly rapid convergence of the nugget estimator can be
understood from the effective rank of the RBF kernel matrix. By
\Cref{lem:projection}, the signal component has numerical rank
$O(b_n^p)$. The remaining directions are approximately pure noise with
variance $\tau^2$. Since $b_n^p$ grows much more slowly than $n$, most
directions eventually contribute almost exclusively to estimation of the
nugget.

This explains the limit
\[
 \sqrt n(\widehat\tau^2-\tau_0^2)
 \Longrightarrow
 N(0,2\tau_0^4):
\]
asymptotically, the nugget behaves approximately like a variance estimated
from $n$ independent $N(0,\tau_0^2)$ observations.

The same interpretation explains the deviations at the smallest sample
sizes. The approximation becomes accurate only when $n$ is large relative
to the number of signal directions. For example, for $p=3$ and $n=125$,
$96$ of the $125$ eigen-directions of $R_n(l_0)$ have signal variance
$\sigma_0^2\lambda_k$ exceeding the noise variance $\tau_0^2$. Thus,
relatively few directions behave as nearly pure noise, and the nugget
estimator is farther from its asymptotic regime.

\subsection{Finite-sample behavior of the spatial parameters}
\label{apdx:sim-spatial}

The limiting rates for $\sigma^2$ and $l$ are powers of
\[
 b_n=\frac{\log n}{\log\log n},
\]
which grows extremely slowly. Over the entire simulated range
$10^2\le n\le10^6$, $b_n$ increases only from approximately $3.0$ to
$5.3$. As a result, finite-sample slopes need not be close to their
limiting exponents even when the Monte Carlo errors already closely track
the Fisher-information benchmark.

\textbf{Lower-order terms in the effective dimension.}
The proof suggests approximately
\[
 N_p(m)=\binom{m+p}{p}
\]
recoverable Taylor coefficients through degree $m\asymp b_n$. The
limiting rate for $\sigma^2$ uses
\[
 N_p(m)\sim\frac{m^p}{p!},
\]
but the lower-order terms are substantial when $m$ is only around $3$--$5$.
For example,
\[
 \frac{d\log N_2(m)}{d\log m}=1.55
\]
at $m=5$, rather than its limiting value $2$. Since the variance of
$\widehat\sigma^2$ is of order $N_p(m)^{-1}$, this produces a shallower
finite-sample slope than the limiting value $-p/2$.

The exact Fisher-information calculations show this effect directly.
Near $n=10^6$, the local slopes of the asymptotic standard deviation
against $b_n$ are $-0.37$, $-0.74$, and $-1.13$ for $\sigma^2$ when
$p=1,2,3$, compared with the limiting values $-0.5$, $-1$, and $-1.5$.
The convergence is extremely slow: the corresponding slopes reach only
$-0.39$ at $n=10^8$ for $p=1$ and $-0.82$ and $-1.22$ at $n=10^{12}$
for $p=2,3$. For $p=2,3$, the Monte Carlo slopes in
\Cref{tab:sim-slopes} closely reproduce these finite-$n$
Fisher-information slopes.

The behavior of the lengthscale is somewhat different because its
information is weighted toward the highest recoverable degrees. Near
$n=10^6$, the local Fisher-information slopes are $-1.57$, $-2.06$, and
$-2.57$ for $p=1,2,3$, already close to the limiting values
$-(p+2)/2$. At smaller sample sizes they are steeper; for example, the
local slopes between $n=10^2$ and $10^3$ are approximately $-2.57$ for
$p=2$ and $-4.46$ for $p=3$. This explains why slopes fitted over the
entire simulated range can be steeper than the limiting exponent.

\textbf{The Fisher information is the relevant finite-sample benchmark.}
\Cref{thm:rbf} determines the orders of the Fisher information but does
not require $J_n$ to converge. Consequently, the limiting powers of
$b_n$ describe the eventual rate but need not provide an accurate
finite-$n$ approximation. The design-specific quantity
\[
 [\calI_n(\theta_0)^{-1}]_{jj}^{1/2}
\]
is therefore the appropriate benchmark for the Monte Carlo errors at a
given $n$. The close agreement between this quantity and the simulated
errors for $p=2,3$ shows that the MLE is already well described by the
Fisher approximation even though the Fisher information itself has not
yet reached its limiting power law.

\textbf{The one-dimensional case.}
The normal approximation is slower for $p=1$ because the effective number
of signal directions remains small. Over $10^2\le n\le10^6$, only
approximately $7$--$11$ eigen-directions have signal variance exceeding
the noise variance. Thus, $\widehat\sigma^2$ behaves roughly like a
variance estimator based on a small number of independent Gaussian
components, leading to visible skewness on the original scale.

The log transformation removes much of this skewness. The remaining
tail behavior is associated with the strong positive dependence between
$\widehat\sigma^2$ and $\widehat l$ along the likelihood ridge. As noted
in \Cref{apdx:sim-mle}, direct evaluation of the profile likelihood for
the extreme cases confirms that these estimates are genuine likelihood
maximizers rather than numerical optimization failures.

Correspondingly, at $n=10^6$ the asymptotic standard deviations of
$\log\widehat\sigma^2$ are $0.59$, $0.26$, and $0.12$, and those of
$\log\widehat l$ are $0.059$, $0.022$, and $0.010$, for $p=1,2,3$,
respectively. The larger number of informative directions in higher
dimensions also helps explain why the normal approximation improves from
$p=1$ to $p=3$.

\subsection{Validation against dense computations}
\label{apdx:sim-validation}

For designs small enough to permit direct $O(n^3)$ linear algebra, we
compared the fast implementation with dense calculations that form the
full $n\times n$ covariance matrix. The comparisons used $n=300$, $1000$,
and $4000$ for $p=1$, $G=20$ for $p=2$, and $G=6$ for $p=3$.

The log likelihoods evaluated at four parameter values, including
$\theta_0$, agreed to within $10^{-6}$, a relative error below $10^{-9}$.
For $n\le1000$, the MLEs agreed with those obtained by direct
Nelder--Mead optimization of the dense likelihood to at least five
significant digits. The Fisher information matrices agreed to relative
error below $10^{-13}$. Finally, for designs with $n=40$, $36$, and $64$
points for $p=1,2,3$, the empirical covariance matrices
computed from $2\times10^4$ independently generated data sets
agreed with $C_n(\theta_0)$ within Monte Carlo error.

These checks confirm that the feature and Kronecker calculations used in
the simulations reproduce the exact Gaussian likelihood and Fisher
information to numerical precision.

\subsection{Monte Carlo implementation and software}
\label{apdx:sim-software}

Each Monte Carlo replicate uses an independent random seed generated from 
\[\texttt{numpy.random.SeedSequence((20261003, p))}.\] Consequently, the
results do not depend on the number of parallel workers, and any
individual replicate can be regenerated on its own.

Computations used Python 3.13 with NumPy 2.4 and SciPy 1.18 on a laptop
with a four-core (eight-thread) Intel Core i7-1165G7 processor. The total
computation time was approximately $8.5$ hours:
about $8$ hours for $p=1$, dominated by the cases
$n\ge3\times10^5$, approximately $15$ minutes for $p=2$, and
approximately $20$ minutes for $p=3$. For $p\ge2$, a single MLE fit at
$n=10^6$ requires approximately one second.

\subsection{Results on the original scale}
\label{apdx:sim-original}

\Cref{fig:sim-raw-p1,fig:sim-raw-p2,fig:sim-raw-p3} repeat the simulation
results in \Cref{fig:sim-p1,fig:sim-p2,fig:sim-p3} for
$(\sigma^2,l,\tau^2)$ on the original parameter scale. The reference
rates are unchanged, and the asymptotic standard deviation is
$[\calI_n(\theta_0)^{-1}]_{jj}^{1/2}$ rather than its log-scale
counterpart
$[\calI_n(\theta_0)^{-1}]_{jj}^{1/2}/\theta_{0j}$.

The conclusions concerning the convergence rates are unchanged. The
main visible difference is greater right skewness of
$\widehat\sigma^2$, particularly for $p=1$, and also of $\widehat l$ in
the one-dimensional case. This skewness increases the RMSE relative to
the robust standard deviation and is substantially reduced by the log
transformation, consistent with the discussion in
\Cref{sec:simulation}.

\renewcommand{\thefigure}{D\arabic{figure}}
\renewcommand{\theHfigure}{D\arabic{figure}}
\setcounter{figure}{0}

\begin{figure}[!h]
\centering
\includegraphics[width=\textwidth]{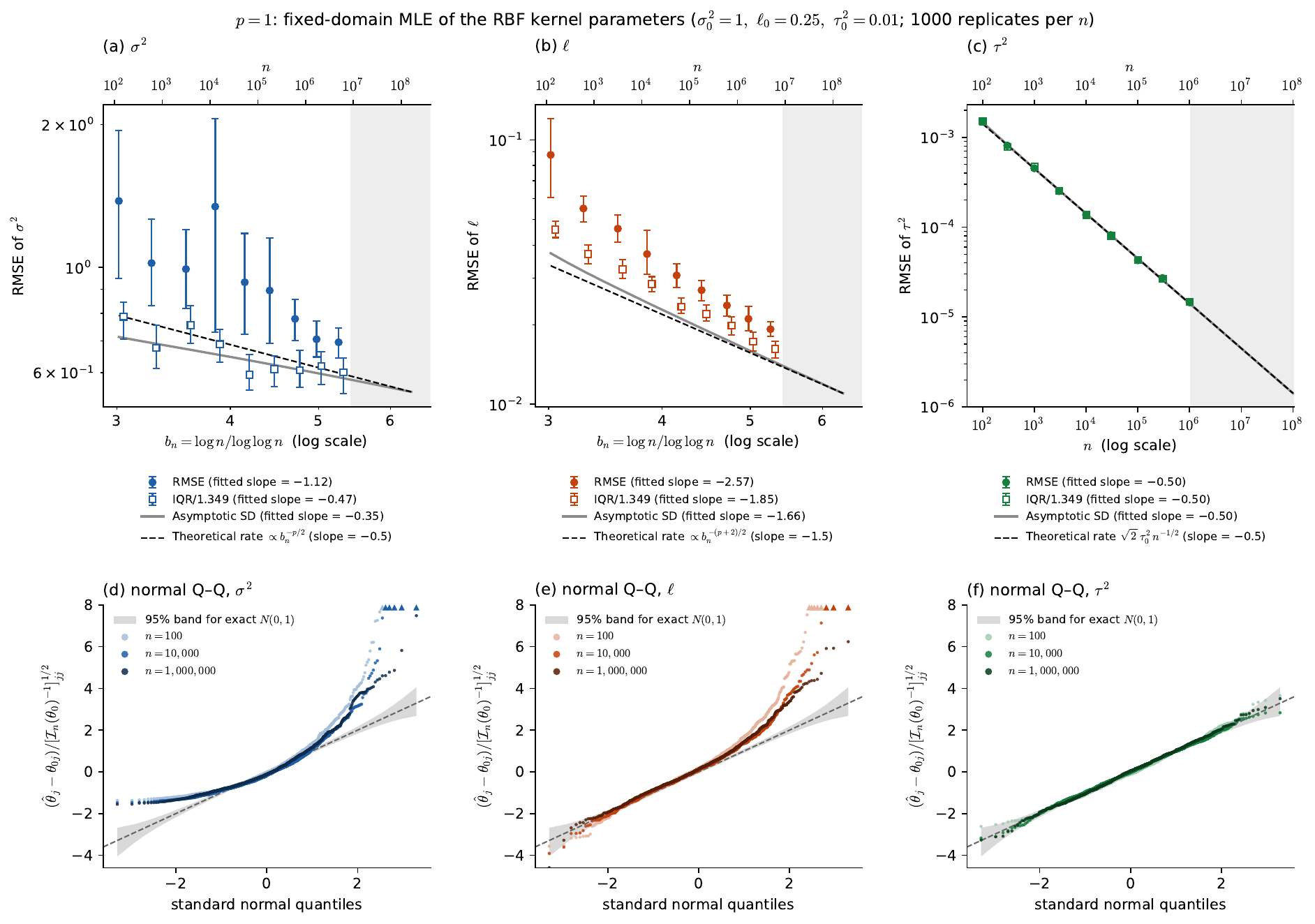}
\caption{Simulation results for $p=1$ on the original parameter scale,
with the layout and graphical elements as in \Cref{fig:sim-p1}.}
\label{fig:sim-raw-p1}
\end{figure}

\begin{figure}[!h]
\centering
\includegraphics[width=\textwidth]{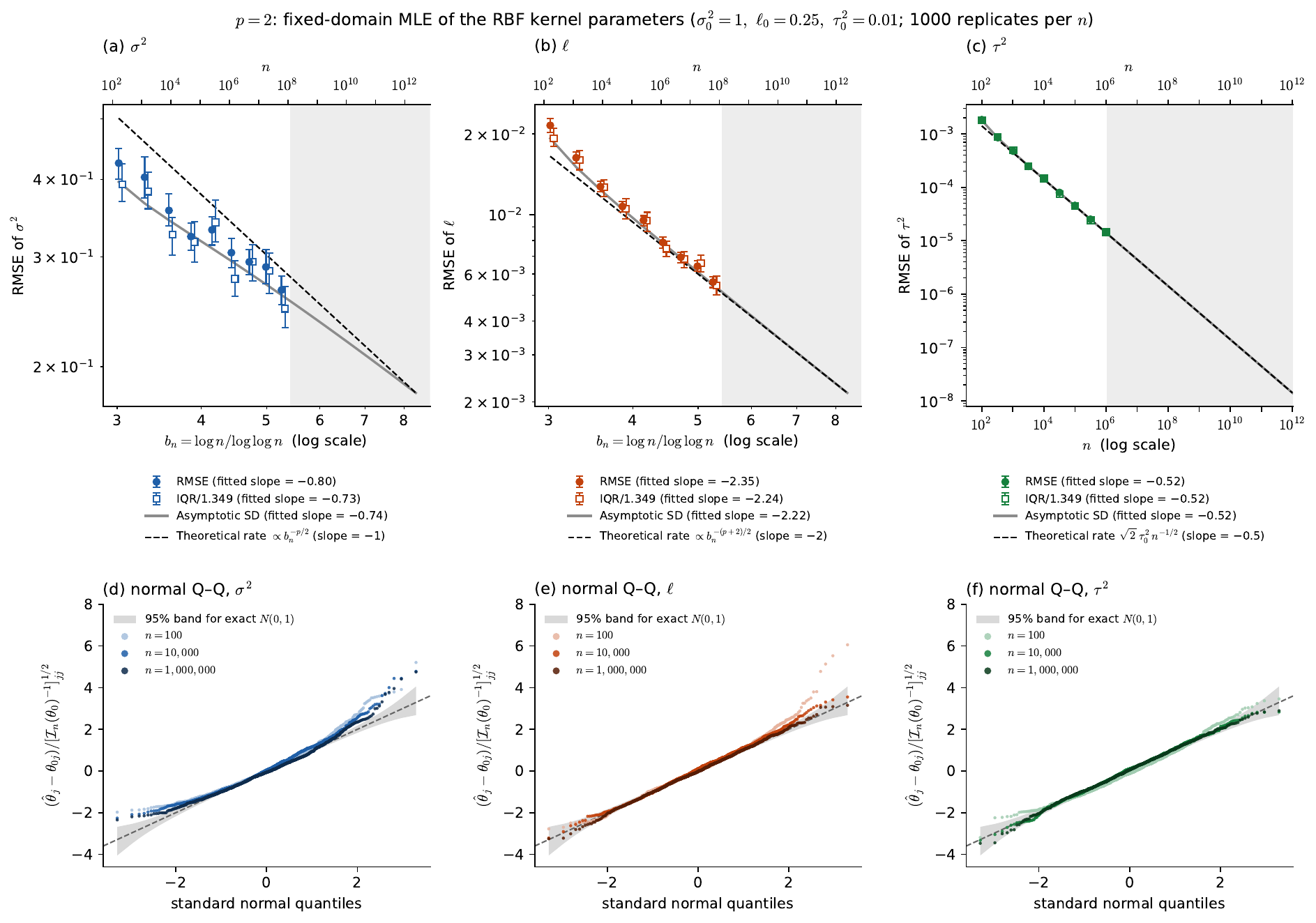}
\caption{Simulation results for $p=2$ on the original parameter scale,
with the layout and graphical elements as in \Cref{fig:sim-p1}.}
\label{fig:sim-raw-p2}
\end{figure}

\begin{figure}[!h]
\centering
\includegraphics[width=\textwidth]{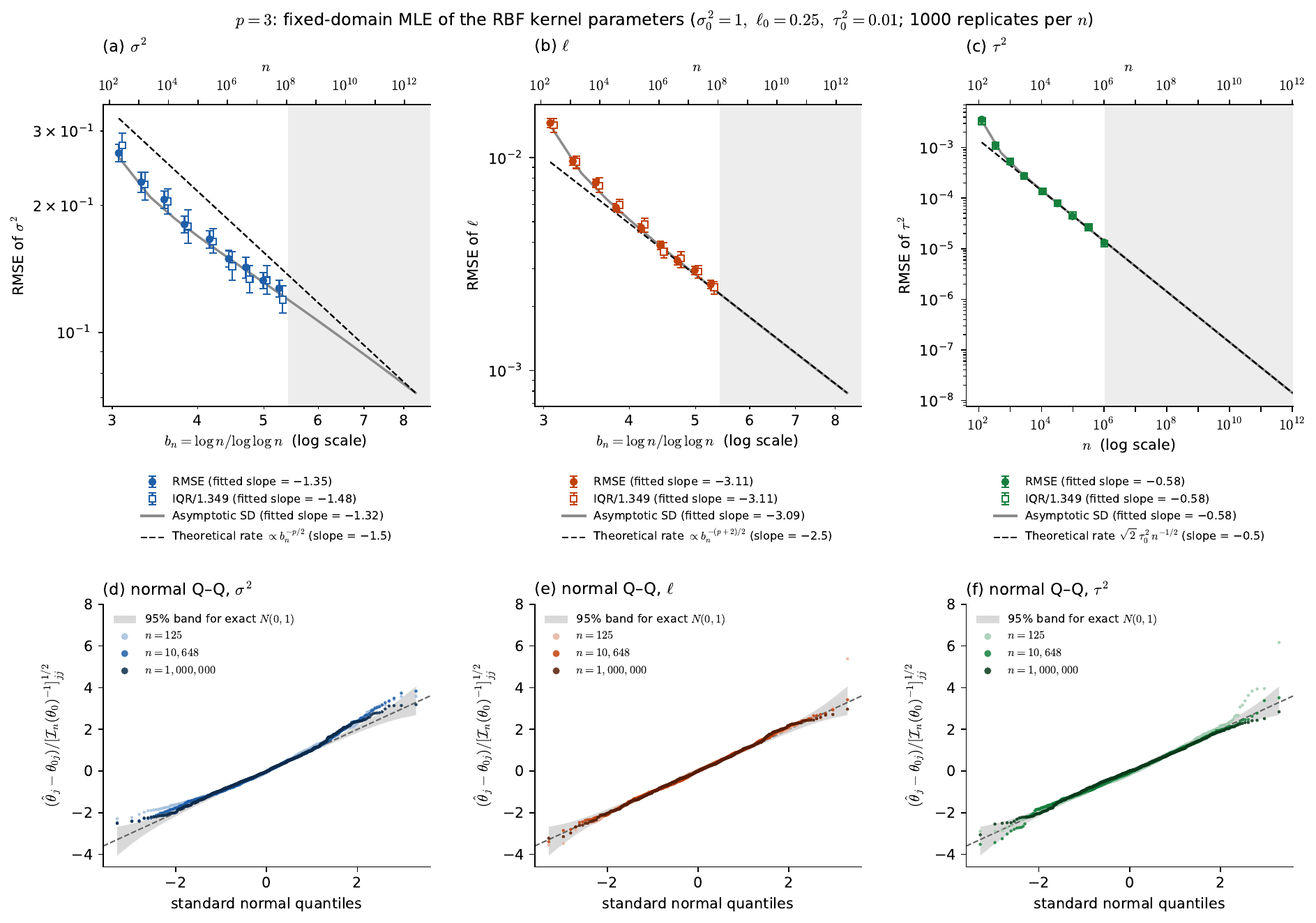}
\caption{Simulation results for $p=3$ on the original parameter scale,
with the layout and graphical elements as in \Cref{fig:sim-p1}.}
\label{fig:sim-raw-p3}
\end{figure}

\section*{Disclosures} The authors interacted with GPT-6 and Claude Opus 5.5 to enhance mathematical clarity with proofs and presentation of the manuscript, and Claude Opus 5.5 to assist with coding the simulation experiments (available from \url{https://github.com/Ameer-eng/Simulation-Studies-for-Euclidean-RBF-MLE}). The authors independently derived all results and take full responsibility for the content of this paper.

\bibliographystyle{plainnat}
\bibliography{ref}

\end{document}